%% file: main.tex
\documentclass[reqno]{amsart}
\usepackage[utf8]{inputenc}
\usepackage{amssymb,amsmath,amsthm,amscd,epsf,latexsym,verbatim,graphicx,amsfonts,hyperref,epstopdf,xcolor,wasysym, cleveref}
\usepackage[margin=1in]{geometry}
\usepackage{mathtools}
\usepackage{wrapfig}
\usepackage{subcaption}
\input{insbox}
\usepackage{mathrsfs}
\usepackage[normalem]{ulem}

\usepackage{todonotes}
\usepackage{thmtools} 
\usepackage{thm-restate}
\usepackage{blkarray}

\usepackage{cutwin}
\usepackage{tikz}
\usepackage{tkz-graph}
\usepackage{caption}
\usepackage{float}
\usetikzlibrary{shapes,snakes}
\usetikzlibrary{arrows}
\usetikzlibrary{shapes.misc}

\usepackage{array}  
\usepackage{longtable} 
\usepackage{booktabs}
\newcolumntype{C}{>{$}c<{$}} 
\input{insbox}

\tikzset{cross/.style={cross out, draw=black, fill=none, minimum size=2*(#1-\pgflinewidth), inner sep=0pt, outer sep=0pt}, cross/.default={2pt}}

\newtheorem{theorem}{Theorem}[section]

\newtheorem{lemma}[theorem]{Lemma}
\newtheorem{proposition}[theorem]{Proposition}
\newtheorem{corollary}[theorem]{Corollary}

\theoremstyle{definition}

\theoremstyle{definition}

\theoremstyle{definition}
\newtheorem{remark}[theorem]{Remark}
\theoremstyle{definition}
\newtheorem{construction}[theorem]{Construction}
\theoremstyle{definition}
\newtheorem{example}[theorem]{Example}
\theoremstyle{definition}
\newtheorem{definition}[theorem]{Definition}
\theoremstyle{definition}

\theoremstyle{definition}

\theoremstyle{definition}

\newcommand{\HH}{\mathbb{H}}

\newcommand{\NN}{\mathbb{N}}

\newcommand{\cl}{\mathrm{cl}}
\newcommand{\tl}{\mathrm{tl}}
\newcommand{\td}{\mathrm{td}}
\newcommand{\calA}{\mathcal{A}}

\newcommand{\prim}{\mathrm{prim}}
\newcommand{\nclprim}{N_{\cl,\prim}}
\newcommand{\ncl}{N_{\cl}}
\newcommand{\ntd}{N_{\td}}
\newcommand{\nwlupper}{N^{\mathrm{upper}}_{\mathrm{wl}}}
\newcommand{\nwllower}{N^{\mathrm{lower}}_{\mathrm{wl}}}

\newcommand{\gd}{\mathrm{gd}}

\newcommand{\calB}{\mathcal{B}}

\title{Complexity of Billiards in Polygons Associated to Hyperbolic $(p,q)$-Tilings}
\author{Sunrose T. Shrestha and Jane Wang}
\date{\today}

\begin{document}

\maketitle
\begin{abstract}
The complexity of the billiard language of regular polygons in the hyperbolic plane with $p$ sides and $2\pi/q$ internal angles is known to grow exponentially and the exponential growth rate is known to equal the topological entropy of the billiard system. In this paper we compute these exponential growth rates explicitly when $q$ is even and give bounds when $q$ is odd. Additionally, for the $q$ even case, we give complete grammar rules that establish when a word (finite, infinite or bi-infinite) in $p$ letters is realized by a billiard path. This latter result is roughly stated and not rigorously proved in \cite{GiannoniUllmo}. In this paper, we provide a precise statement and a complete proof using new methods relating to minimal tiling paths. 
   
\end{abstract}
\input{intro.tex}

\input{tiling_growth_rates.tex}

\input{language_complexity.tex}

\input{word_characterization.tex}

\input{computations.tex}

\newpage
\bibliographystyle{alpha}
\bibliography{references}

\input{appendix.tex}

\end{document}

%% file: intro.tex
\section{Introduction and Preliminaries}

In the setting of polygonal billiards, a point mass representing a billiard ball travels along a unit speed geodesic trajectory in a polygonal table. Upon hitting a side of the table, the trajectory bounces off the side with angle of reflection equal to the angle of incidence. Much is known about the dynamics of polygonal billiards in the Euclidean case, especially when the angles of the polygons are rational multiples of $2\pi$, thus allowing the polygonal tables to unfold nicely to \textit{translation surfaces}. In comparison, much less is known about dynamics of hyperbolic polygonal billiards, which occur on tables of constant negative curvature.

\textit{Language complexity} is a notion of the complexity of a polygonal billiard system. While there are broad results known about language complexity in both the Euclidean and hyperbolic billiard case, we know of only a few cases where the language complexity has been computed explicitly in the Euclidean case (for certain tables that tile the plane in \cite{CassaigneHubertTroubetzkoy}, and for regular polygonal tables in \cite{AthreyaHubertTroubetzkoy}). In the hyperbolic setting, we are only aware of computations for the language complexity in the case of ideal polygons. One of the major goals of this paper is to compute the language complexity growth for regular hyperbolic $p$-gons whose internal angles are $2\pi/q$, which provides a hyperbolic analog to the result of \cite{AthreyaHubertTroubetzkoy}. We will find the explicit growth rate when $q$ is even and will provide upper and lower bounds when $q$ is odd.

Following Cassaigne, Hubert, and Troubetzkoy (\cite{CassaigneHubertTroubetzkoy}), we first give some definitions. Given $P$, a polygonal billiard table, we can label the sides of $P$ using unique letters from an alphabet, where the number of letters in the alphabet is equal to the number of sides of $P$. Then, any oriented, finite, infinite or bi-infinite billiard trajectory gives a word with letters from the alphabet corresponding to the sequence of sides that the billiard trajectory hits. We will refer to this word as the \textbf{bounce sequence} or the \textbf{(billiard) word} of the billiard trajectory. 

\begin{definition}[Language complexity] Given a polygonal billiard table $P$, let $\mathcal{L}(n)$ denote the set of possible $n$-letter words that arise as bounce sequences of a finite billiard trajectory. Then, let 
\begin{equation} 
\label{eq:p(n)}
p(n) = \#\mathcal{L}(n)
\end{equation}
be the \textbf{complexity function} of the language $\mathcal{L}(n)$. 
\end{definition}

The \emph{language complexity} of a polygonal billiard system is the growth rate of the complexity function $p(n)$. Katok proved in \cite{katok} that for a general Euclidean polygonal billiard, $p(n)$ grows slower than any exponential.  On other other hand, Troubetzkoy proved in \cite{Trou} that $p(n)$ grows at least quadratically for any Euclidean polygon.

In \cite{CassaigneHubertTroubetzkoy}, Cassaigne, Hubert, and Troubetzkoy make $p(n)$ easier to understand for Euclidean billiard tables by relating it to the number of \textit{generalized diagonals} up to some \emph{combinatorial length}, which have historically been easier to count. 

\begin{definition}[Generalized diagonals and combinatorial length]\label{def:combinatorial_length} A \textbf{generalized diagonal} in a polygonal billiard table is a billiard trajectory that begins and ends at vertices, without hitting any vertices in between. Edges of the polygon do not count as generalized diagonals. If a generalized diagonal $\gamma$ minus its endpoints makes $k$ edge hits, then we say that the \textbf{combinatorial length} of this generalized diagonal is 
\begin{equation}
\label{eq:cl(gamma)}
    \cl(\gamma) = k+1,
\end{equation}
which can be thought of as the number of straight segments that the generalized diagonal breaks up into, with a break at each edge hit. 
\end{definition}

Alternatively, if we were to \textit{unfold} the billiard trajectory (i.e. reflecting the polygon over edges that the billiard trajectory passes through while keeping the trajectory straight), we can visualize a generalized diagonal $\gamma$ as a straight trajectory from vertex to vertex in the unfolding that does not intersect any vertices except at the endpoints of $\gamma$. See Figure~\ref{fig:unfolding} for an illustration.

\begin{figure}
\centering
\includegraphics[width=10cm]{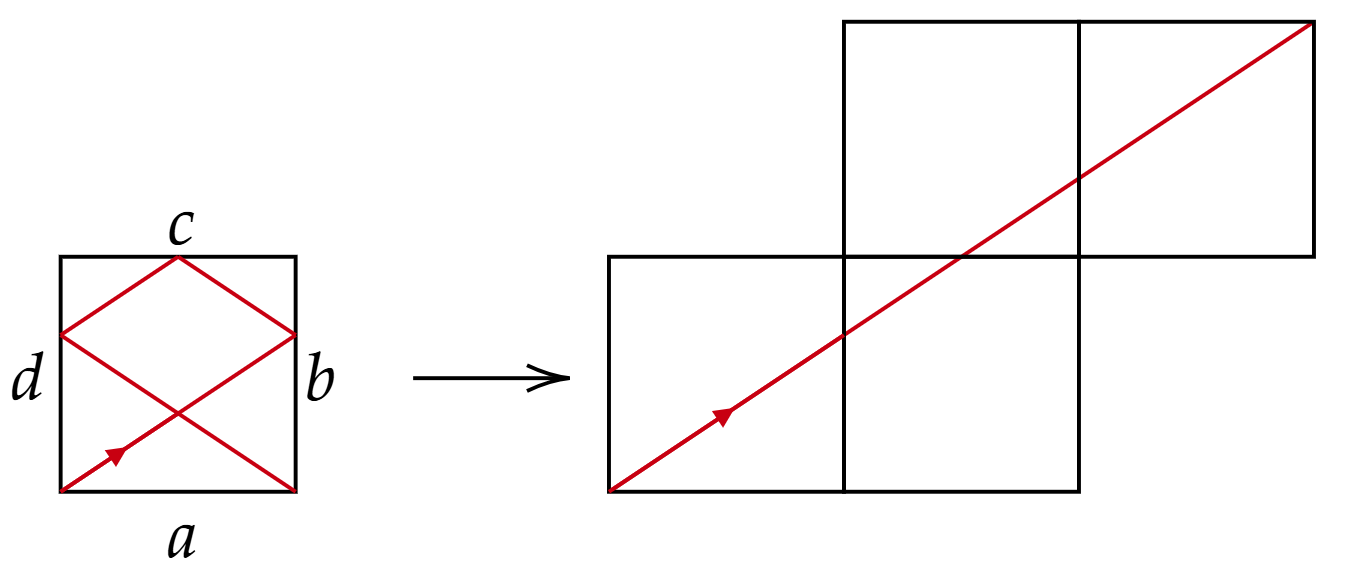}
\caption{On the left, a generalized diagonal with combinatorial length $4$ and billiard word $bcd$. On the right, the unfolding of this generalized diagonal.}
\label{fig:unfolding}
\end{figure}

\begin{definition}[Generalized diagonal counts] For any $j \geq 0$, we let $\gd(j)$ denote the number of generalized diagonals of combinatorial length $j$, where $\gd(0)$ is defined to be the number of vertices of the polygonal table and edges are not counted in $\gd(1)$. 
\end{definition}

In \cite{CassaigneHubertTroubetzkoy}, Cassaigne, Hubert, and Troubetzkoy prove the following general theorem in the Euclidean case. We note that the notation we use is different than the notation in \cite{CassaigneHubertTroubetzkoy} where they use $N_c(t)$ for the number of generalized diagonals of combinatorial length \emph{at most} $t$. 

\begin{theorem} [\cite{CassaigneHubertTroubetzkoy}] 
\label{thm:CassaigneHubertTroubetzkoy} 
For any Euclidean, convex, polygonal billiard table, 
$$p(n) = \sum_{k=0}^{n-1} \sum_{j=0}^k \gd(j).$$
\end{theorem}
We note that this result was extended to the non-convex case by Bedaride in \cite{bedaride}.
Results of Masur (\cite{masur88, masur90}) imply that on Euclidean polygons with angles that are all rational multiples of $\pi$, the quantity $\sum_{j=0}^k \gd(j)$ has quadratic upper and lower bounds. We call such polygons \textbf{rational polygons}. Cassaigne, Hubert, and Troubetzkoy then obtain the following natural corollary to their theorem.

\begin{corollary}[\cite{CassaigneHubertTroubetzkoy}] For any rational, Euclidean, convex, polygonal billiard, there exist constants $D_1$ and $D_2$ for which $$D_1 < \frac{p(n)}{n^3} < D_2$$ for all positive integers $n$. 
\end{corollary}

The authors then compute exact cubic asymptotics $\lim_{n\rightarrow\infty} \frac{p(n)}{n^3}$ for the Euclidean square, equilateral triangle, and isosceles triangle using combinatorial techniques and the fact that these polygons unfold to tilings of the Euclidean plane. Using counting techniques for saddle connections of lattice translation surfaces, Athreya, Hubert, and Troubetzkoy compute the exact cubic asymptotics for the language complexity of regular Euclidean $n$-gons (\cite{AthreyaHubertTroubetzkoy}).

In general, it appears difficult to compute the exact asymptotics of the language complexity function $p(n)$ for general polygonal tables, although some general bounds exist. In \cite{scheglov2}, Scheglov proved that for a typical triangle (i.e. for a full measure set of triangles) and any $\epsilon > 0$, there exists a constant $C$ such $\sum_{j=0}^k \gd(j) < Ce^{k^\epsilon}$. Using Theorem~\ref{thm:CassaigneHubertTroubetzkoy}, this implies an explicit sub-exponential upper bound for the language complexity of typical triangles. Along similar vein, in \cite{KruNogTrou}, the authors show that for a typical polygon (i.e for a dense $G_\delta$ set), $p(n)$ grows cubically. This is the first case where cubic complexity for irrational billiards has been established.

\subsection{Complexity for hyperbolic polygonal billiards}

There has also been interest in the complexity of polygonal billiards in non-Euclidean geometries. In \cite{GutTab}, Gutkin and Tabachnikov study the asympototics of the language complexity of polygonal billiards in spherical, Euclidean, and hyperbolic geometries by relating the language complexity to the complexity of a related family of \textit{2D piecewise convex transformations}. In particular, they prove the following theorem. 

\begin{theorem}[Gutkin-Tabachnikov \cite{GutTab}] 
\label{thm:GutTab-general}
Let $P$ be a polygonal billiard table. 
\begin{enumerate}
    \item If $P$ is a rational, Euclidean table, there exists a constant $c$ such that $p(n) \leq cn^3$. 
    \item If $P$ is spherical, then $p(n)$ grows subexponentially. 
    \item If $P$ is hyperbolic, let $h_{top}$ denote the topological entropy of the billiard map on $P$. Then, $h_{top} > 0$, and there exists a subexponential function $s(n)$ for which $p(n) = s(n)e^{h_{top}n}$. That is, the exponent of exponential growth of $p(n)$ is exactly the topological entropy, $h_{top}$. 
\end{enumerate}
\end{theorem}

Note that the conclusion in case of hyperbolic polygonal billiards in Theorem~\ref{thm:GutTab-general} can be restated as 

$$\lim_{n\rightarrow \infty} \frac{\log_{h_{top}}(p(n))}{n} = 1.$$

Additionally, in this case, Gutkin and Tabachnikov prove this theorem by relating the language complexity $p(n)$ to counts of generalized diagonals $\gd(k)$ of combinatorial length $k$. This is the hyperbolic analogue of Theorem~\ref{thm:CassaigneHubertTroubetzkoy} of Cassaigne, Hubert, and Troubetzkoy.

\begin{theorem}[Gutkin-Tabachnikov\cite{GutTab}] \label{thm:GutTab}Let $P \subset \HH$ be a convex geodesic polygon. Then, 
\begin{equation}\label{eq:ComplexityFormula}
    p(n) = 2p(1) + n(p(2)-p(1)) + \sum_{k=3}^n \sum_{j=3}^k \gd(j)
\end{equation}
\end{theorem}

Motivated by computations of exact asymptotics for families of Euclidean polygons in \cite{CassaigneHubertTroubetzkoy} and  (\cite{AthreyaHubertTroubetzkoy}), we will study the asymptotics of the language complexity of hyperbolic polygonal billiard tables that tile the hyperbolic plane. In particular, we will focus on regular hyperbolic tables with $p$ sides and internal angle $2\pi/q$. The unfoldings will then give us $(p,q)$-tilings, which are highly symmetric and therefore easier to analyze.

\begin{definition}[Hyperbolic $(p,q)$-tiling] For $p,q \in \mathbb{N}$ such that $\frac{1}{p}+ \frac{1}{q} < \frac12$, a \textbf{$(p,q)$-tiling} is a tiling of the hyperbolic plane $\mathbb{H}$ by regular $p$-sided polygons, with $q$ polygons meeting at each vertex. Given that $\frac{1}{p} + \frac{1}{q} < \frac12$, these tilings exist and are unique up to isometry. 
\end{definition}

By Theorem~\ref{thm:GutTab-general}, asymptotic growth rates for language complexity will give us estimates for the topological entropy of the billiard maps on these polygons. 

\subsection{Main Results}

For the rest of the paper, our setting will be a hyperbolic polygonal billiard table $P$ with $p$ sides and internal angle $2\pi/q$ that unfolds to a hyperbolic $(p,q)$-tiling where $\frac{1}{p}+\frac{1}{q} < \frac{1}{2}$. In this process, we also note that billiard paths unfold to hyperbolic geodesics in $\HH$.

Our main results in this paper regarding such billiard systems are two fold. First, for the $q$ even case we compute the billiard language complexity exactly and for the $q$ odd case we give improved bounds (see Theorem~\ref{thm:complexity_evenodd}).  Our main approach is to relate the billiard language complexity to the growth of tiles with respect to \emph{tiling distance} (see Definition~\ref{def: tiling_length}) in associated hyperbolic tilings. To our knowledge, this is a novel approach towards computing billiard language complexity. Secondly, we completely characterize words arising from finite, infinite and bi-infinite billiard paths for the $q$ even case (see Theorems \ref{thm:admissible-tiles}, \ref{thm:infinite_bijection}, \ref{thm:biinfinite_realization}). We state these results in the next two subsections.

\subsubsection{Complexity for hyperbolic polygonal billiards from $(p,q)$-tilings}

Given a hyperbolic billiard table $P$, our main tool in computing its language complexity will be to relate it to the \textit{tiling growth rate} of the associated $(p,q)$-tiling. To do so, we first define tiling paths and their lengths. 

\begin{definition}[Tiling paths and length]
\label{def:tiling_path}
In a given tiling, any sequence of tiles $\mathcal{A} = (A_1, A_2, \dots, )$ (finite, infinite, or bi-infinite) where each $A_i$ and $A_{i+1}$ share an edge is called a \textbf{tiling path}. The \textbf{tiling length} of a finite tiling path $\mathcal{A}$, denoted $\tl(\mathcal{A})$, is the number of steps from tile to successive tile and one fewer than the number of tiles in the path. A finite tiling path $(A_1, \dots,A_n)$ is called \textbf{minimal} if there does not exist a tiling path between $A_1$ and $A_n$ with fewer than $n$ tiles. An infinite or bi-infinite tiling path is \textbf{minimal} if every finite  subpath is minimal. 
\end{definition}

We note that we adopt the convention that tiling length is one fewer than the number of tiles in the path so that the tiling length of a path equals the length (number of edges) of the tiling path in the dual graph.

\begin{definition}[Tiling distance between tiles]\label{def:tiling_distance}
 Given two tiles, $A$ and $B$, the \textbf{tiling distance} between $A$ and $B$, denoted $\td(A, B)$ is the length of a minimal tiling path between $A$ and $B$.
\end{definition}

\begin{definition}[Tiling distance counts]
\label{def: tiling_length}
Given a tiling of the hyperbolic plane, we can choose a base tile. We then let $\ntd(k)$ denote the number of tiles of tiling distance $k$ away from the base tile. By the symmetry of a tiling, $\ntd(k)$ is independent of the choice of base tile.
\end{definition}

Our first main result demonstrates that the growth rate of $\ntd(\cdot)$ for a hyperbolic $(p,q)$-tiling precisely yields the billiard language complexity for a regular hyperbolic polygon with $p$ sides  and $2 \pi/q$ angles when $q$ is even and upper bounds the billiard language complexity when $q$ is odd. 

\begin{restatable}[Hyperbolic Billiard Language Complexity]{theorem}{ComplexityEvenOdd}
\label{thm:complexity_evenodd}
Consider a hyperbolic $(p,q)$-tiling.   
Let $\alpha > 1$  such that $\lim_{n\rightarrow \infty} \frac{\log_\alpha \ntd(n)}{n} = 1$. 
Let $p(n)$ be the language complexity function of the associated hyperbolic billiard system. 

\begin{enumerate}
    \item When $q$ is even, the language complexity function $p(n)$ satisfies 
    $$\lim_{n\rightarrow \infty} \frac{\log_\alpha(p(n))}{n} = 1.$$ It follows that the exponential growth rate of the language complexity function $p(n)$, which is equal to the  topological entropy of the billiard system, is equal to $h_{top} =\alpha$. 
    \item When $q$ is odd, the language complexity function $p(n)$ satisfies
    $$\frac{q-1}{q+1} \leq \lim_{n\rightarrow \infty} \frac{\log_\alpha(p(n))}{n} \leq 1.$$
    It follows that the exponential growth rate of the language complexity function $p(n)$, which is equal to the topological entropy of the billiard system $h_{top}$, is bounded between $\alpha^{\frac{q-1}{q+1}}$ and $\alpha$ in this case.  
    
\end{enumerate}
\end{restatable}

Tiling growth rates (i.e. the growth rate of $\ntd(\cdot)$) of hyperbolic tilings in turn relate to the growth rates of certain hyperbolic groups, a subject which is well-studied (see, for example, \cite{wagreich} and \cite{floydplotnick}). However, since we could not find a direct formulation or proof for $(p,q)$-tilings, in Section~\ref{sec:tiling} we provide one and prove the following theorem for completeness and for the benefit of the reader. We use recursive counting techniques, similar to ones used in \cite{montee} in a more specific setting.

\begin{restatable}[Tiling growth rates]{theorem}{tiling}
\label{thm:tiling_growth}
For $p \geq 3$ and $\frac1p + \frac1q < \frac{1}{2}$, the growth series $f(x) = \sum_{n=0}^\infty \ntd(n) x^n$ of the hyperbolic $(p,q)$-tiling is given by the following generating functions: 
\begin{enumerate}
    \item If $q$ is even, then $$f(x) = \frac{1 + 2x  + \ldots + 2x^{\frac{q}{2}-1} + x^{\frac{q}{2}}}{1 - (p-2)x - \ldots - (p-2)x^{\frac{q}{2}-1} + x^{\frac{q}{2}}}.$$
    \item If $q$ is odd, then $$f(x) = \frac{1 + 2x + \ldots + 2x^{\frac{q-1}{2}-1} + 4x^{\frac{q-1}{2}}+ 2x^{\frac{q-1}{2} + 1}+ \ldots + 2x^{q-2} + x^{q-1}}{1  - (p-2)x - \ldots -(p-2)x^{\frac{q-1}{2}-1} - (p-4)x^{\frac{q-1}{2}} - (p-2)x^{\frac{q-1}{2} + 1}- \ldots - (p-2)x^{q-2} + x^{q-1}}.$$
\end{enumerate}
Subsequently, $\lim_{n\rightarrow \infty}\frac{\log_\alpha \ntd(n)}{n} = 1$ where $\alpha$ is the largest root of the denominator of $f(x)$. 
\end{restatable}

\subsubsection{Characterization of hyperbolic billiard words}

In addition to determining the language complexity growth for billiards in a hyperbolic polygonal billiard table, we can also hope for the existence of a set of grammar rules that determine or at least restrict what billiard trajectories are possible. Here, we mainly consider grammar rules of the form ``words that contain the subword $w$ can never be achieved by a billiard path'' or ``words that do not contain any of a set of restricted subwords must be achieved by a billiard path". 

In \cite{GiannoniUllmo}, Giannoni and Ullmo propose a set of grammar rules that they claim completely classify billiard trajectories in hyperbolic $(p,q)$-tilings with $q$ even. However, their theorem statement contains some ambiguities that we will discuss in Section~\ref{sec:billiard_words}. Also, at the suggestion of the journal editor, they did not provide a proof of their theorem in their published paper, opting only to discuss some of the ideas that went into the proof. Our contribution is to provide a statement of language rules that classify infinite billiard trajectories, in the spirit of but more precise than what is stated in \cite{GiannoniUllmo}. We will then provide a complete proof of this result using new techniques relating to characterizing both finite and infinite minimal tiling paths. The main result we will prove in Section~\ref{sec:billiard_words} is the following. 

\begin{restatable}[Language rules for $q$ even]{theorem}{GUeven}\label{thm:words_even}
In a hyperbolic polygonal billiard corresponding to a $(p,q)$-tiling with $q$ even, bi-infinite words coming from billiard trajectories must follow the following two rules: 
\begin{enumerate}
    \item[E1] Any two-letter subword $kk$ is not allowed for any $k \in \{1,\ldots,p\}$. 
    \item[E2]  Any subword of length greater than $\nu_q = \frac{q}{2}$ consisting of alternating letters $k$ and $k+1$ (or $p$ and $1$) is not allowed. 
\end{enumerate}

Equivalence classes (as defined in Definition~\ref{def:equivalence}) of basepointed bi-infinite words  that do not violate grammar rules E1 and E2 are realized by a unique geodesic in the $(p,q)$-tiling.
\end{restatable}

We will postpone the formal definition of the equivalence relation on words referred to in this theorem to Definition~\ref{def:equivalence}, but informally, under the equivalence, words containing a $\nu_q$ length string of alternating $k$ and $k+1$ starting with $k$ is equivalent to a string of the same length of alternating $k+1$ and $k$ starting with $k+1$. 

In their paper, Giannoni and Ullmo suggest that their main method of proof is analyzing the shape of shrinking sets of billiard trajectories that achieve longer and longer subwords of a bi-infinite word. 

We found their methods difficult to verify and a proof was not given in their paper, so in Section~\ref{sec:billiard_words}, we will prove this theorem of Giannoni and Ullmo using different methods. Specifically, we will define \textit{admissibility classes} of words, which are sets of words that do not violate grammar rules E1 and E2 that are equivalent, under the definition of equivalence given in Definition~\ref{def:equivalence}. We will then carefully separate the finite, (one-sided) infinite and bi-infinite cases and draw an appropriate conclusion for each. First, we show that admissible classes of finite words are in bijection with tiles that are finite distance away from a base tile. 

\begin{restatable}[Finite length admissible classes and tiles bijection for $q$ even]{theorem}{GUfinite}\label{thm:admissible-tiles} In the $q$ even case, given a base tile, there is a bijection between admissible word classes of length $n$ and tiles of distance $n$ away from the base tile. 
\end{restatable}

For the infinite case, recall that two geodesic rays in the hyperbolic plane are equivalent if they have the same endpoint. We then show that admissible classes of infinite words are in bijection with equivalence classes of geodesic rays beginning at a base tile. 

\begin{restatable}[Infinite word admissible classes and geodesic rays bijection for $q$ even]{theorem}{GUinfinite}
\label{thm:infinite_bijection}
Given a hyperbolic $(p,q)$-tiling with $q$ even, fix a base tile $A_0$. Then there is a bijection between equivalence classes of geodesic rays beginning in $A_0$ and admissible word classes of infinite words. 
\end{restatable} 

Finally, for the bi-infinite case, we \textbf{center} a given bi-infinite word $\dots w_{-1}w_{0}w_{1}\dots$ at a tile $A$ by considering the labels $w_{-1}$ and $w_0$ to be edges of tile $A$. Subsequently, we define a \textbf{centered} bi-infinite admissible word class by picking a representative and centering it at a chosen tile. We then show that centered bi-infinite admissible word classes are realized by unique geodesics. 

\begin{restatable}[Bi-infinite centered admissible classes and unique geodesic realization for $q$ even]{theorem}{GUbiinfinite}
\label{thm:biinfinite_realization}
Given a hyperbolic $(p,q)$-tiling with $q$ even, every centered bi-infinite admissible word class is realized by a unique geodesic.
\end{restatable}

Noting that billiard paths correspond to hyperbolic geodesics via unfolding, Theorem~\ref{thm:biinfinite_realization} in particular implies that each centered bi-infinite word class is achieved by a unique bi-infinite billiard path.

The grammar rules of Theorem~\ref{thm:words_even}, while useful for getting a better understanding of the set of possible billiard words, are not enough to directly give us language complexity asymptotics. This is because the characterizations involve equivalence classes of words, which are hard to count. For exact computations of the language complexity in the case of even $q$, we can apply our Theorem~\ref{thm:complexity_evenodd}. 

In \cite{GiannoniUllmo}, Giannoni and Ullmo also give an incomplete set of rules in the $q$ odd case (see Theorem~\ref{thm:GUodd}). Again, we found their method of proof difficult to verify. However, assuming that their theorem is correct, it can be used to give upper and lower bounds to language complexity growth rates for $(p,q)$-tilings in the case when $q$ is odd. In Section~\ref{sec:comparison}, we compare our bounds from Theorem~\ref{thm:complexity_evenodd} and show that we obtain better upper bounds than what can be achieved by Giannoni and Ullmo's methods. In this section we also discuss examples illustrating why Theorem~\ref{thm:GUodd} does not give a complete coding of billiard trajectories. 

We summarize the exact value (in the $q$ even case) and best known bounds (in the $q$ odd case) of language complexity coming from Theorems \ref{thm:complexity_evenodd} and  \ref{thm:GUodd} for a few different values of $(p,q)$ in Table \ref{table:ComplexityValues}.

\begin{table}[h!]
\centering
\begin{tabular}{cc|p{5cm}}
$p$ & $q$ & Billiard Language Complexity\\
\hline
3 & 8 &1.72208380573904  \\
4 &6 &2.61803398874989 \\
4 &8& 2.89005363826396 \\
5 &4& 2.61803398874989 \\
5 &6& 3.73205080756888 \\
5 &8 &3.93869094197994 \\
6 &4 &3.73205080756888 \\
6 &6 &4.79128784747792 \\
6 &8 &4.96069291859917 \\
7 &4 &4.79128784747792 \\
7 &6 &5.82842712474619 \\
7 &8 &5.97262435800721 \\
8 &4 &5.82842712474619 \\
8 &6 &6.85410196624968 \\
8 &8 &6.97983577921557 
\end{tabular}
\quad
\begin{tabular}{cc|p{3cm}|p{3cm}}
  &  & \multicolumn{2}{c}{Billiard Language Complexity Range}\\
 \cline{3-4}
 $p$ & $q$ & Lower Bound & Upper Bound\\
 \hline
3 &7 &1.39320015609277 & 1.55603019132268  \\
3 &9 &1.62242661883033 & 1.83107582510231 \\
4 &5 &1.74071386619816 & 2.29663026288654 \\
4 &7 &$2.41421356237309^*$ & 2.82320193241387  \\
4 &9 &$2.83117720720834^*$ & 2.94699466977899 \\
5 &5 &2.30788179437972 & 3.50606805595024   \\
5 &7 & $3.56155281280883^*$ & 3.89797986736932  \\
5 &9 & $3.90057187491196^*$ & 3.97594397745373  \\
6 &5 & $3^*$ & 4.61158178930871 \\
6 &7 & $4.64575131106459^*$ & 4.93282638839610  \\
6 &9 & $4.93394490940640^*$ & 4.98704581211217 \\
7 &3 & 1.61803398874989 & 2.61803398874989   \\
7 &5 & $4^*$ & 5.67798309021366  \\
7 &7 & $5.70156211871642^*$ & 5.95225287964244
\end{tabular}
\caption{Billiard language complexity values and best known bounds for some regular hyperbolic polygons with $p$ sides and internal angle $2\pi/q$. The values for the $q$ even case, and the upper bounds and the unstarred lower bounds in the $q$ odd case come from our contribution (Theorem~\ref{thm:complexity_evenodd}). The starred lower bounds for the $q$ odd case are deduced from Theorem~\ref{thm:GUodd} by Giannoni and Ullmo.}
\label{table:ComplexityValues}
\end{table}

\subsection{Additional Notes and references}

The language complexity of polygonal billiards and related systems has been studied by many authors. For example, other closely related works include the study of language complexity of billiards in Euclidean triangles (\cite{scheglov1}), in general Euclidean polygons (\cite{GutkinRams}), and in higher dimensional polyhedra (\cite{baryshnikov}, \cite{bedaride}). Language complexity has also been studied in other contexts such as for cutting words of tilings of regular polyominos (\cite{HubertVuillon}).

The question of determining the possible billiard words in a polygonal billiard is also well-studied. In \cite{Duchin-HearTheShape}, Duchin et. al. obtain rigidity results for the bounce spectrum of a Euclidean polygonal billiard table. That is, they show that for on right-angled affinely-equivalent polygons, the set of possible billiard words completely determines the shape of the table. This can be considered a polygonal billiard version of the famous geometry question, ``Can one hear the shape of a drum?" (\cite{Kac-ShapeOfDrum}). In the hyperbolic setting, Erlandsson et. al. prove a similar bounce spectrum rigidity theorem in \cite{Erlandsson-HyperbolicBounceSpectrum}. There has also been interest in determining the set of possible words on other geometric structures such as translation surfaces (see, for example, \cite{Davis_CuttingSequences}, \cite{DavisPasquinelliUlcigrai}, \cite{SmillieUlcigrai}). 

In the \emph{ideal} hyperbolic setting (where all vertices of the billiard polygon are on the boundary) Giannoni and Ullmo prove in \cite{GiannoniUllmoNonCompact} that infinite billiard trajectories are in one-to-one correspondence with bi-infinite sequences of words in the sides without repeated letters, relating such billiard systems to subshifts of finite type (SFT). Coding hyperbolic billiards and their relation to SFTs have also been studied in \cite{NagarSingh}, \cite{NagarSingh2024} and \cite{Singh}, the latter work being a generalization to polyhedral billiards in ideal hyperbolic polyhedra. 

In our proof of Theorem~\ref{thm:complexity_evenodd}, where we relate language complexity with tiling growth rate, one of our key tools is relating the tiling distance between two tiles to the number of special curves (which we call \textit{edge geodesics} for $q$ even and \textit{zigzags} for $q$ odd) separating the two tiles. A variant of this idea has been used successfully in other contexts, for example by Erlandsson in \cite{erlandsson} to show that the word length of an element in fundamental group of a hyperbolic surface is given by its intersection number with a collection of well-chosen arcs.

\subsection{Outline of the paper}
In Section~\ref{sec:tiling}, we provide self-contained proofs of the generating functions for the growth of tiles by tiling distance in hyperbolic $(p,q)$-tilings, which is the content of Theorem~\ref{thm:tiling_growth}. In Section~\ref{sec:billiard_complexity}, we relate the combinatorial length of a finite geodesic path to the tiling distance between the initial and final tiles of the path. By doing so, we are then able to relate tiling distance growth rates to combinatorial length growth rates and prove Theorem~\ref{thm:complexity_evenodd} about the growth rates of language complexity for particular regular hyperbolic billiard systems. In Section~\ref{sec:billiard_words}, we discuss difficulties with the theorems by Giannoni and Ullmo (\cite{GiannoniUllmo}) that give grammar rules for billiard paths in hyperbolic $(p,q)$-tilings. We provide a precise reformulation of their theorem in the $q$ even case and prove this theorem using new techniques. Along the way, we prove results about both finite and infinite minimal tiling paths when $q$ is even.  Finally, in Section 5, we compare growth rate bounds for the $q$ odd case coming both from our Theorem~\ref{thm:complexity_evenodd} and from a theorem of Giannoni and Ullmo (\cite{GiannoniUllmo}), and show that in many cases our growth rate bounds are better. Throughout this paper, we use various facts about hyperbolic tilings whose proofs are mostly elementary hyperbolic geometry. We defer the proofs of these facts to Appendix~\ref{appdx:hyperbolictilings}.

\subsection{Acknowledgements}

The authors would like to thank Aaron Fenyes, MurphyKate Montee, and  Serge Troubetzkoy for helpful conversations about this project. J.W. was partially supported by an AMS-Simons Travel Grant and NSF grant DMS-2453391. S.S. was partially supported by an AMS-Simons Research Enhancement Grant for Primarily Undergraduate Institutions.

%% file: tiling_growth_rates.tex
\section{Tiling Growth Rates}
\label{sec:tiling}

In this section, we provide a self-contained derivation of the generating functions for the growth of tiles in $(p,q)$-tilings, which is given in Theorem~\ref{thm:tiling_growth}. We believe that this result is well-known, as evidenced by related results found in various sources (e.g. \cite{floydplotnick}, \cite{wagreich}). Since we found it difficult to find a straightforward proof of the tiling growth rate result in the literature, we provide our own proof here. Some ideas from this proof will also be used in later sections of this paper. A reader eager to get to the main results of this paper could skip this section and refer back to parts of it when they are referenced later in the paper.

We recall that a $(p,q)$-tiling is a tiling of $2$-dimensional space by regular, equiangular $p$-gons, meeting $q$ at a vertex. We will restrict our attention to the case when $p \geq 3$ and $
\frac1p + \frac1q < \frac{1}{2}$, with the latter condition guaranteeing that tiling can be constructed in the hyperbolic plane.

After choosing an arbitrary base tile, we wish to understand the growth of the number of tiles of tiling distance $n$ away from the base tile (for the definition of tiling distance, see Definition~\ref{def: tiling_length}). To do so, we define a \textbf{growth series} of a tiling to be a function $$f(x) = \sum_{n=0}^\infty \ntd(n) x^n,$$ such that $\ntd(n)$ is the number of tiles of tiling distance $n$ away from our base tile. Our goal is to find a generating function for the growth series of a $(p,q)$-tiling. 

To do so, we first need a firmer understanding of how to iteratively generate a $(p,q)$-tiling by starting from a fixed base tile and then inductively generating all tiles of distance $n$ away from the base tile. Equivalently, we can consider how to generate the \textbf{dual graph} or \textbf{dual tiling} (which is a $(q,p)$-tiling), where we start with a base vertex and inductively generate all vertices of distance $n$ away from the base vertex.We note that vertices of the dual graph correspond to faces or tiles of the original tiling. 

\begin{construction}
    \label{con:tiling}
The graph of a $(p,q)$-tiling can be generated iteratively and is a planar graph, with each vertex having a well defined distance from a fixed base vertex: 
\begin{enumerate}
    \item We start with one base vertex, which has distance $0$. This will be called \textbf{stage $0$} of the construction.
    \item Given that the graph has been defined up to distance $k$, from each vertex of distance $k$, we create new edges to new vertices of distance $k+1$ until the original vertex has degree $q$. 
    \item Since the graph is planar, it makes sense to check if there are any open faces with $p$ edges or $p-1$ edges. If there is an open face with $p$ edges, we merge the two distance $k+1$ vertices to create a closed face with $p$ edges. If there is an open face with $p-1$ edges, we add an edge between the two distance $k+1$ vertices to create a closed face with $p$ edges (we note that this latter case only can happen when $p$ is odd). Steps 2 and 3 together will be called \textbf{stage $k+1$} of the construction. Alternately, stage $k+1$ is precisely the stage where vertices of distance $k+1$ are formed.
\end{enumerate}
\end{construction}

\begin{remark} 
\label{rem:construction_unique} We note here that by definition, the process in Construction \ref{con:tiling} generates a planar $(p,q)$-tiling, which has a planar $(q,p)$-tiling as its dual. As all $(p,q)$-tilings are isomorphic as graphs, this process must generate the dual graph to a hyperbolic $(q,p)$-tiling. 
\end{remark}

\begin{remark} 
\label{rem:cyclic_ordering}
Let $b$ be a chosen base vertex of a $(p,q)$-tiling and let the \textbf{distance} from a vertex $v$ to $b$ be defined in the sense of graph theory (the minimum length of an edge path connecting $b$ and $v$). The distance $k$ vertices are exactly those that are constructed in stage $k$ of the construction. Furthermore, we can see from Construction \ref{con:tiling} and the planarity of a $(p,q)$ tiling that there is a fixed structure to the distances of vertices on any given tile: 
\begin{enumerate}
    \item If $p$ is even, then for every tile there is some $k$ such that the vertices of the tile cyclically have distances $\{k, k+1, \ldots, k+ \frac{p}{2}-1, k+ \frac{p}{2}, k+ \frac{p}{2}-1, \ldots, k+1\}$.  
    \item If $p$ is odd, there are two kinds of tiles: those with one vertex of distance $k$ and two vertices of each of distances $k+1, k+2, \ldots, k+ \frac{p-1}{2}$, and those with two vertices of distances $k, k+1, \ldots, k + \frac{p-1}{2}-1$ and one of distance $k+\frac{p-1}{2}$. The cyclic ordering of the vertices in either case is again from $k$ up to $k+\frac{p-1}{2}$ and then back down.
\end{enumerate}

\end{remark}

We also have the following lemma, which will be useful for our next proposition. 

\begin{lemma} In Construction \ref{con:tiling} for generating a $(p,q)$-tiling, the valence of every distance $k$ vertex at stage $k$ of the construction is $1$ or $2$ if $p\geq 4$, and valence $3$ or $4$ if $p = 3$. 
\label{lem:valence}
\end{lemma}
\begin{proof}
We start with the case where $p\geq4$ and induct on the stage $k$ of the construction. As the base case, we see by inspection that every distance $1$ vertex has valence $1$ at stage $1$. Suppose that distance $k$ vertices all had valence $1$ or $2$ at stage $k$. Then, in stage $k+1$, each distance $k$ vertex would create new edges to at least $q-2 \geq 2$ new distance $k+1$ vertices. Each such new vertex must then be adjacent to at least one open face with only two edges, which will not be completed in stage $k+1$. Thus, each new vertex can only have valence one (if neither open polygon on either side of it is completed in this step) or valence two (if exactly one adjacent open polygon is completed). The proof of the $p=3$ case follows a similar structure, except that we note that at every stage $k$, every open face is completed to a triangle.
\end{proof}

Construction \ref{con:tiling} implies the following fact that gives a well-defined map from the vertices of the original tiling to the faces when $p \geq 4$ and vice versa when $p = 3$, which will be useful while relating tiling length (to be defined) to combinatorial length (see Definition~\ref{def:combinatorial_length}) of generalized diagonals in Section~\ref{sec:billiard_complexity}.

\begin{proposition}\label{prop:vertex_tile_mapping}
    Consider a hyperbolic $(p,q)$-tiling. Let $\mathcal{V}$ be the set of all vertices and $\mathcal{T}$ be the set of all tiles. 
    \begin{enumerate}
        \item When $p \geq 4$, there exists a surjective map $\varphi:\mathcal{V} \rightarrow \mathcal{T}$ such that for each $v \in \mathcal{V}$, $\varphi(v) \in \mathcal{T}$ is a tile adjacent to $v$ (i.e. containing $v$ as a vertex) and the map $\varphi$ is at most $p$ to 1.  
        \item When $p =3$, there exists a surjective map $\psi: \mathcal{T} \rightarrow \mathcal{V}$ such that for each $A \in \mathcal{T}$, $\psi(A) \in \mathcal{V}$ is a vertex of tile $A$ and the map $\psi$ is at most $q$ to 1.  
    \end{enumerate}
    
\end{proposition}
\begin{proof}

It will suffice to show that when $q\geq4$, there exists a surjective map $\psi: \mathcal{T} \rightarrow \mathcal{V}$ such that for each tile $A$, $\psi(A) \in \mathcal{V}$ is a vertex of tile $A$ and the map $\psi$ is at most $q$ to $1$. Then, case 1 follows by considering the dual graph of the $(p,q)$-tiling, and case 2 follows because when $p=3$, we must have that $q \geq 7$. 

Fix a hyperbolic $(p,q)$-tiling with $q \geq 4$. After choosing a base vertex $b$ of the tiling, recall from Remark~\ref{rem:cyclic_ordering} the distance of a vertex $v$ is defined as the distance between $b$ and $v$ in the sense of graph theory (the minimum length of an edge path connecting $b$ and $v$).

Let $A$ be a tile of a hyperbolic $(p,q)$-tiling with $q \geq 4$. By Remark \ref{rem:cyclic_ordering}, $A$ has either one minimum distance vertex or two minimum distance vertices that are adjacent. Then, $\psi(A)$ will be defined as the unique minimum distance vertex, or the first of the two adjacent minimum distance vertex when traversing the boundary of the tile counterclockwise. By definition, for each $A \in \mathcal{T}, \psi(A)$ is adjacent to $A$. Additionally, since each vertex is adjacent to $q$ tiles, this map is at most $q$ to $1$. 

To show that the map $\psi: \mathcal{T} \rightarrow \mathcal{V}$ is surjective when $q \geq 4$, it suffices to show that any vertex $v$ is the unique minimal distance vertex for some tile. To see this, let $v$ be a vertex of distance $k$ from the base vertex. We first consider the case when $p \geq 4$. Then by Lemma~\ref{lem:valence}, at stage $k$ of the construction of the tiling (where $v$ first appears in the graph), there are either $1$ or $2$ incoming edges from distance $k-1$ vertices (see Figure~\ref{fig:SurjectiveVtoT} for a depiction of the latter case). 

\begin{figure}[h!]
\centering
\begin{subfloat}[Stage $k$]{
    \includegraphics[scale=1]{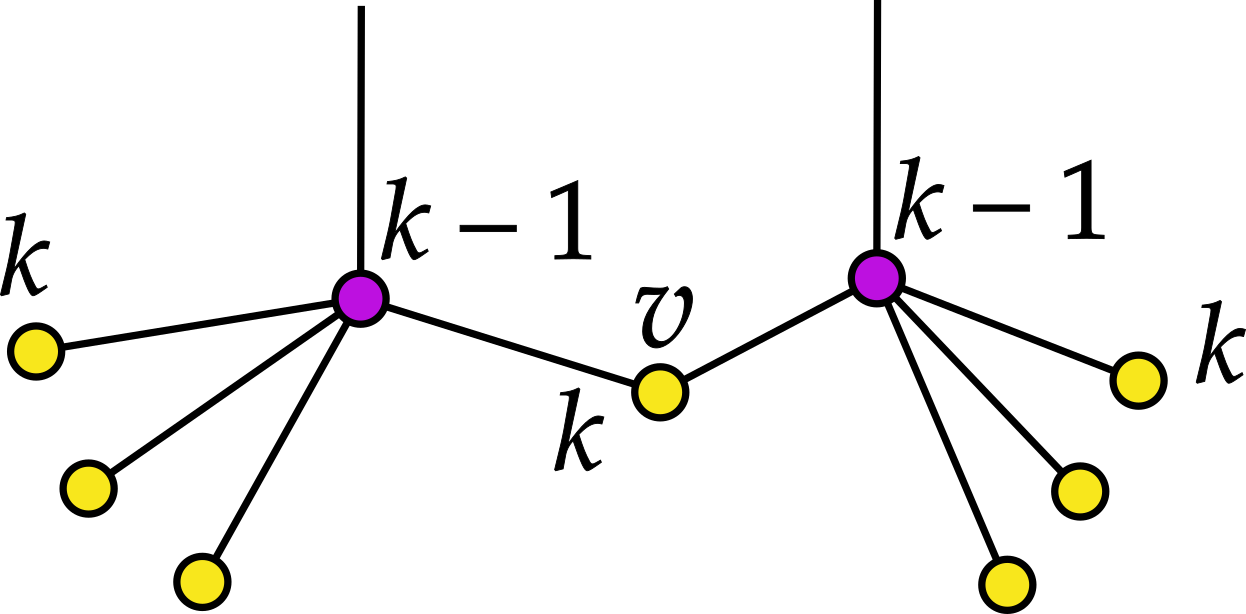}
    \vspace{1.6cm}}  
\end{subfloat}
\hspace{1cm}
\begin{subfloat}[Stage $k+1$]{
    \includegraphics[scale=1]{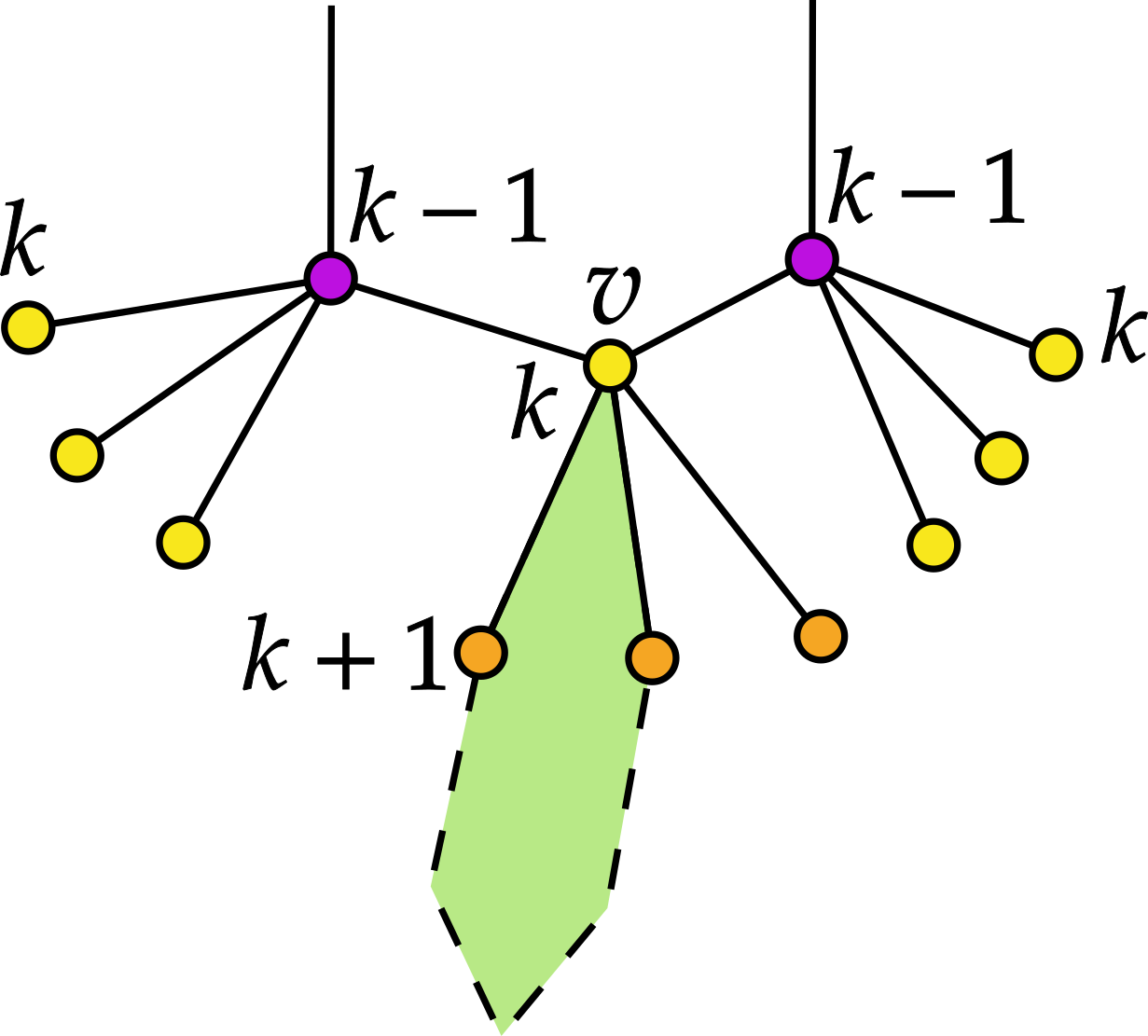} }   
\end{subfloat}
\caption{The vertex $v$ is the unique minimal vertex for the shaded face.}
    \label{fig:SurjectiveVtoT}
\end{figure}

Since $q \geq 4$, there are at least two new outgoing edges at this stage which form the start of a new open face that $v$ belongs to. For this open face, $v$ is the unique minimal vertex. The case when $p=3$ is similar since $q \geq 7$, and so at stage $k+1$ each vertex $v$ forms at least $q-4\geq 3$ new edges. 
\end{proof}

We are now ready to prove Theorem~\ref{thm:tiling_growth}.

\tiling*

\begin{proof}
Our strategy for proof is to use the dual graph to generate recursive relationships among the number of different types of vertices of given types. The \textbf{type} of a vertex in the dual graph will describe how far away the vertex is from the base vertex in the faces that it belongs to. This will give us recursive relationships between types of tiles in the original tiling.

\begin{figure}[h!]
\centering
\begin{subfloat}[Types of vertices in the dual graph of a $(6,8)$-tiling. In the face shown, assume the top vertex is the minimal distance vertex of distance $k$. Then the vertices adjacent to it below it are distance $k+1$ vertices, the ones on the next level are distance $k+2$ vertices and so on.]{
\includegraphics[scale=0.9]{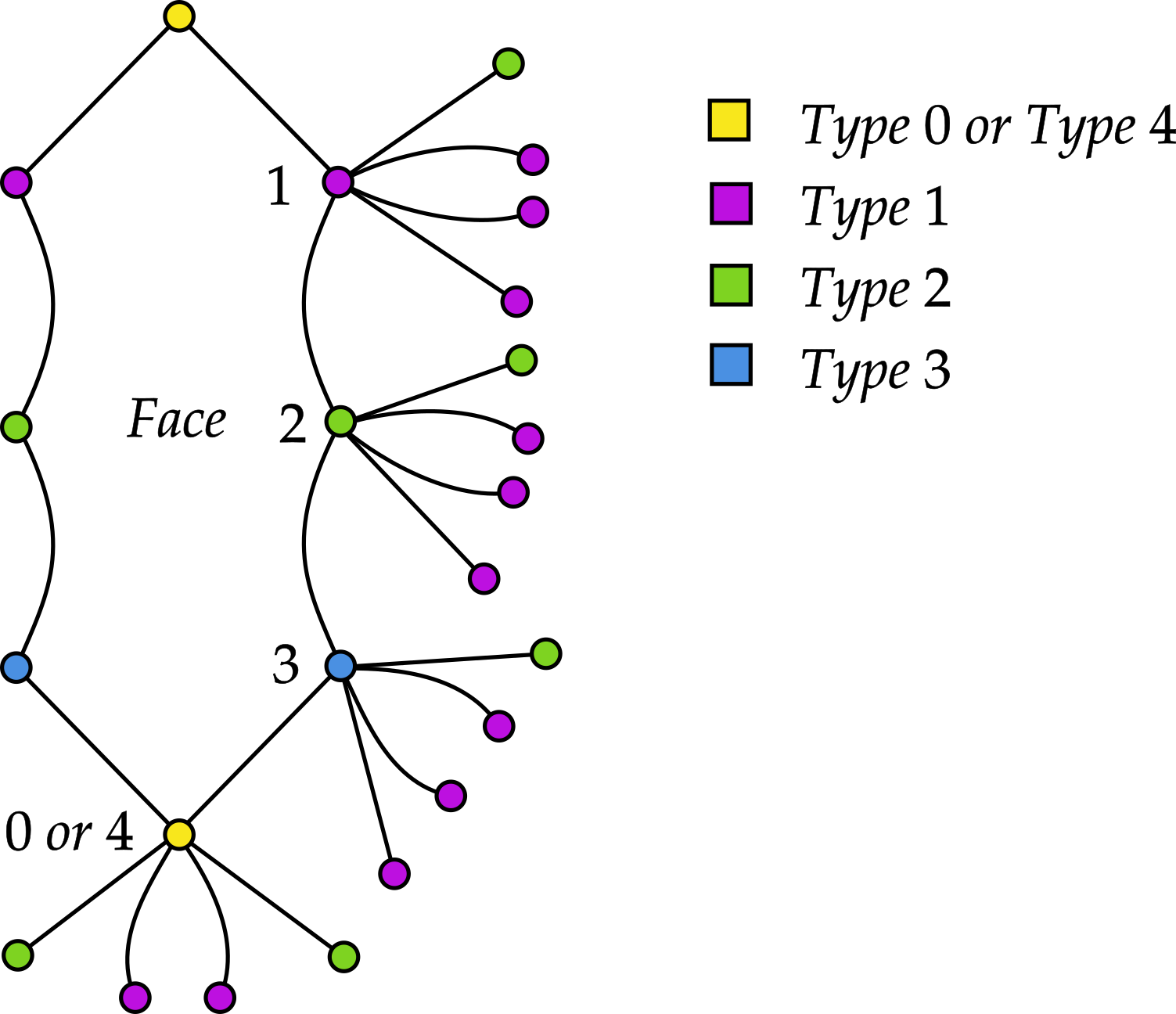}\label{fig:typeseven}}
\end{subfloat}
\hspace{0.5cm}
\begin{subfloat}[Types of vertices in the dual graph of a $(6,5)$-tiling. In order to define the type of vertices, the two faces of different kinds are combined. Then, the type of vertices is defined similar to the $q$-even case.]{
\includegraphics[scale=0.9]{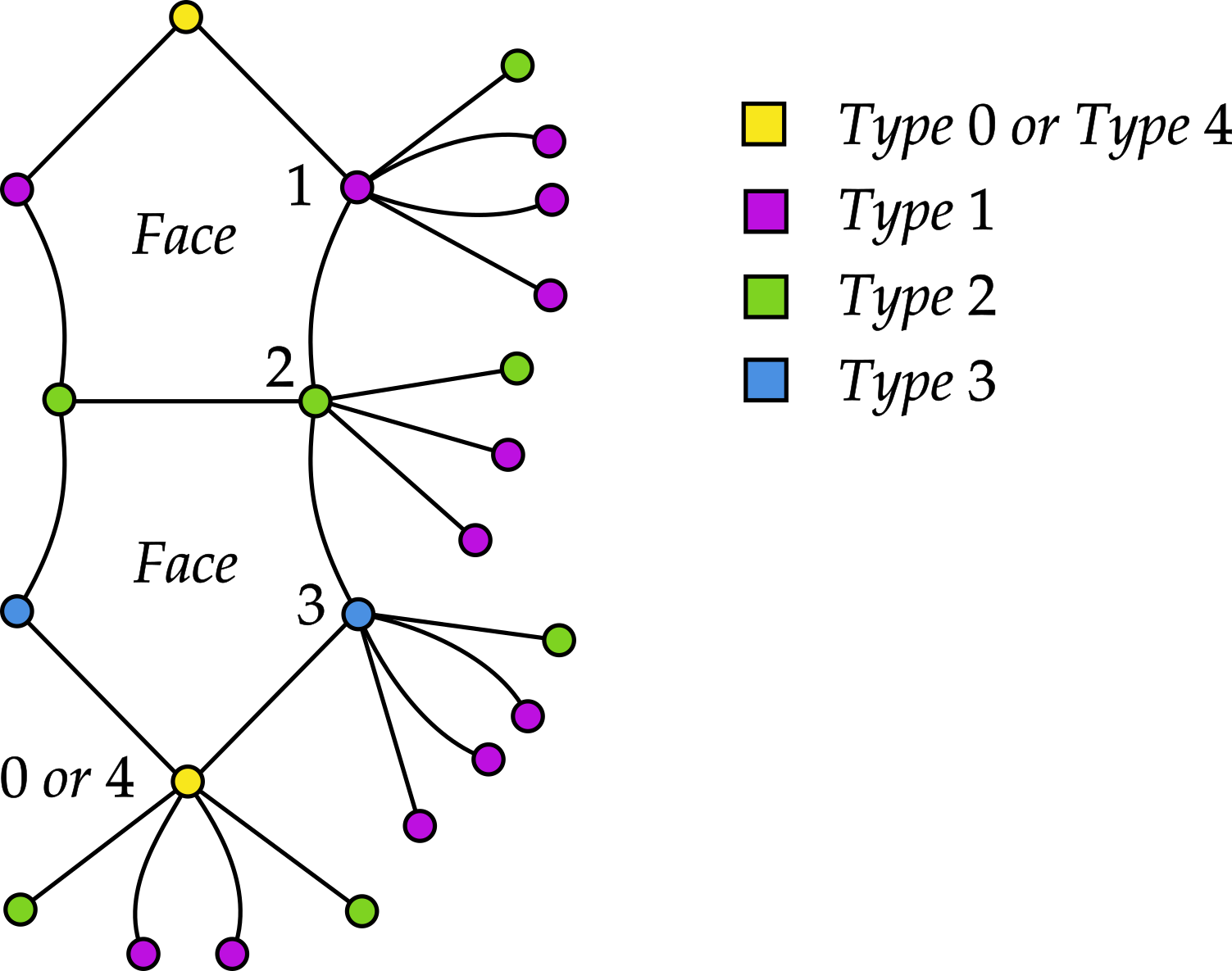}\label{fig:typesodd}}
\end{subfloat}
    \caption{A schematic showing types of vertices in the dual graphs of certain $(p,q)$-tilings illustrating the case of $q$ being even $q$ odd separately.} \end{figure}

We first assume that $q$ is even. Following the process for generating the dual graph, we can see that every face in the dual graph has one vertex of distance $k$ for some $k$, two vertices of distance $k+1, k+2, \ldots, k+ \frac{q}{2}-1$, and one vertex of distance $k+\frac{q}{2}$. Suppose that we have a vertex of distance $n$. Then, we define the \textbf{type} of the vertex to be the largest $i$ for which there is a face that the vertex is a part of that has a vertex of distance $n-i$. For ease of computation, we will use type $0$ and type $\frac{q}{2}$ to mean the same thing. Figure~\ref{fig:typeseven} illustrates how the types would be defined in some vertices in a $(6,8)$-tiling.

Let $T_k^i$ be the number of type $i$ vertices of distance $k$, for $1 \leq i \leq \frac{q}{2}$. Using the structure and the symmetry of the dual graph, we can then define each $T^j_{k+1}$ recursively from the $T^i_k$ terms. For example, for a type $1$ vertex, there are $p$ adjacent edges, exactly one of which is a type $q/2$ or type $0$ vertex. Thinking about the faces generated when we create the next level of vertices of one greater distance, we see that there are two new type $2$ vertices and $p-3$ new type $1$ vertices. The reasoning is similar for the other types. The one case that is a little bit different is when we are considering type $\frac{q}{2}-1$ vertices. Two of these vertices will generate a single type $\frac{q}{2}$ vertex, so we say that $\frac12$ of a type $\frac{q}{2}$ vertex is generated by every type $\frac{q}{2}-1$ vertex. We then put this information together into a matrix, where the $(i,j)$ entry indicates the number of type $j$ vertices of distance $n+1$ generated by each type $i$ vertex of distance $n$. 
\[
\begin{blockarray}{cccccccc}
& T^1 & T^2 & T^3 & T^4 & \cdots & T^{\frac{q}{2}-1} & T^{\frac{q}{2}} \\
\begin{block}{c(ccccccc)}
  T^1 & p-3 & 2 & 0 & 0 & \cdots & 0 & 0 \\
  T^2 & p-3 & 1 & 1 & 0 & \cdots  & 0 & 0\\
  T^3 & p-3 & 1 & 0 & 1 & 0 & \cdots  & 0 \\
  \vdots & \vdots & \vdots & \vdots & \ddots & \ddots &   \ddots & \vdots\\
  T^{\frac{q}{2}-2} & p-3 & 1 & 0 & 0 & \ddots  & 1 &0 \\
  T^{\frac{q}{2}-1} & p-3 & 1 & 0 & 0 & \cdots   & 0 &  \frac12\\
  T^{\frac{q}{2}} & p-4 & 2 & 0 & 0 & \cdots   & 0&  0\\
\end{block}
\end{blockarray}
\]
Note that due to degeneracy, for the $q=4$ case, type $1$ and type $q/2-1$ are the same type. Hence, the matrix has the form

\[\begin{blockarray}{ccc}
& T^1 & T^2 \\
\begin{block}{c(cc)}
  T^1 & p-3 & 1 \\
  T^2 & p-4 & 1 \\
\end{block}
\end{blockarray}\]

Alternatively, we can rewrite these recursive relations in $T_n^i$ notation. For example, by reading down the columns of the matrix, the first relation reads $$T^{1}_{n+1} = (p-3)(T^1_n + T^2_n + \ldots + T^{\frac{q}{2}-1}_n) + (p-4)T^{\frac{q}{2}}_n$$ and the second one reads 

\begin{equation} \label{eq:t2} T^2_{n+1} = 2T^1_n + T^2_n + \ldots + T^{\frac{q}{2} - 1}_n + 2T^{\frac{q}{2}}_n \end{equation}
If we add up all $\frac{q}{2}$ of these recursive relations, we get that for $n \geq 2$, the total number of distance $n+1$ vertices $\ntd(n+1)$ is given by \begin{equation} \label{eq:c} \ntd(n+1) = (p-1)(T^1_n + T^2_n + \ldots + T^{\frac{q}{2}}_n) - \frac12 T^{\frac{q}{2}-1}_n - T^{\frac{q}{2}}_n. \end{equation}

 We are going to find the generating function $f(x) = \sum_{n=0}^\infty \ntd(n) x^n$ by reducing the recursions to be just in terms of the $T^2_k$ terms as well as the $\ntd(k)$ terms. We notice that for $3 \leq i \leq \frac{q}{2} - 1$, $T^{i}_{n+1} = T^{i-1}_n$, and similarly that $T^{\frac{q}{2}}_{n+1} = \frac12 T^{\frac{q}{2}-1}_n$. We can use these relations to rewrite equation \ref{eq:c} as \begin{equation} \label{eq:csmall} \ntd(n+1) = (p-1)\ntd(n) - \frac12 T^2_{n - (\frac{q}{2}-1-2)} - \frac12 T^2_{n - (\frac{q}{2} - 2)}. \end{equation} We can also rewrite equation \ref{eq:t2} as \begin{equation}\label{eq:t2small} T_{n+1}^2 = 2\ntd(n) - (T_n^2 + T_n^3 + \ldots + T_n^{\frac{q}{2}-1}) = 2\ntd(n) - (T_n^2 + T_{n-1}^2 + \ldots + T_{n-(\frac{q}{2}-1 - 2)}^2).\end{equation} Both equations \ref{eq:csmall} and \ref{eq:t2small} hold for all $n \geq 2$. For $n=0$ and $n=1$, we can check that $T_0^2 = 1, T_1^2 = 0,$ and $\ntd(0) = 1, \ntd(1) = p$. We  get recursive equations for the generating functions $f(x) = \sum_{n=0}^\infty \ntd(n) x^n$ and $t(x) \coloneqq \sum_{n=0}^\infty T_n^2 x^n$ from equations \ref{eq:csmall} and \ref{eq:t2small} and by manipulating the coefficients so that the low $n$ terms are correct: 
\begin{equation} \label{eq:frecursion} f(x) = (p-1)xf(x) - \frac12 x^{\frac{q}{2}-2} t(x) - \frac12 x^{\frac{q}{2}-1} t(x) + 1 + x + \frac{1}{2}x^{\frac{q}{2}-2} + \frac{1}{2} x^{\frac{q}{2}-1}, \end{equation}
 
 and 

\begin{equation} \label{eq:trecursion} t(x) = 2xf(x) - xt(x) - x^2 t(x) - \ldots - x^{\frac{q}{2}-2} t(x) - 2x +1 .\end{equation}

Solving for $t(x)$ in equation \ref{eq:trecursion} gives us $$t(x) = \frac{2x(f(x) - 1) + 1}{1 + x + x^2 + \ldots + x^{\frac{q}{2}-2}},$$ and substituting this back into equation \ref{eq:frecursion} and solving for $f(x)$ gives us that $$f(x) = \frac{1 + 2x  + \ldots + 2x^{\frac{q}{2}-1} + x^{\frac{q}{2}}}{1 - (p-2)x - \ldots - (p-2)x^{\frac{q}{2}-1} + x^{\frac{q}{2}}},$$ as desired.
 
The proof of the $q$ odd case follows a similar structure as the $q$ even case, except that the recursive formulas for the vertices of different types is different. 

Since $q$ vertices surround every face in the dual graph, there will be two kinds of faces: faces where there is one vertex of minimum distance and faces where there are two vertices of minimum distance. To define the type of each vertex, we imagine combining adjacent faces of different kinds (that share an edge connecting vertices of the same distance) into a larger face with $2(q-1)$ vertices. With these larger faces, we then define type in the same way as when $q$ was even, as a measure of the maximum distance of the vertex from the minimum distance vertex in all of these larger faces that it belongs to. Figure~\ref{fig:typesodd} demonstrates this for part of a $(p,q) = (6,5)$ tiling. We notice that for a $(p,q)$-tiling, we will have vertices of types $0$ through $q-1$, with type $q-1$ and type $0$ being interchangeable.

Like we did for the $q$ even case, we can create a matrix where the $(i,j)$ entry displays how many type $j$ vertices are generated of distance $k+1$ for every type $i$ vertex of distance $k$. 

\[
\begin{blockarray}{cccccccccc}
& T^1 & T^2 & T^3 &  \cdots & T^{\frac{q-1}{2}} & T^{\frac{q+1}{2}} & \cdots  & T^{q-2} & T^{q-1}\\
\begin{block}{c(ccccccccc)}
  T^1 & p-3 & 2 & 0 &  \cdots & 0 & 0 & \cdots &  0 & 0 \\
  T^2 & p-3 & 1 & 1 &  \ddots  & 0 & 0 & \cdots & 0 & 0 \\
 T^3 & p-3 & 1 & 0 &  \ddots  & 0 & 0 & \cdots & 0 & 0 \\
  \vdots & \vdots  & \vdots  & \vdots & \ddots & \ddots  & \ddots & \ddots & \ddots & \vdots \\
  T^{\frac{q-1}{2}} & p-4 & 1  & 0 & \ddots &  0 & 1 & \ddots & 0 & 0 \\
  \vdots & \vdots & \vdots  & \vdots & \ddots   & \ddots & \ddots & \ddots &  \ddots &  \vdots \\
  T^{q-3} & p-3 & 1 & 0  & \cdots   & 0&  0 & \cdots & 1& 0\\
  T^{q-2} & p-3 & 1 & 0  & \cdots   & 0&  0 & \cdots & 0 & \frac12\\
  T^{q-1} & p-4 & 2 & 0  & \cdots   & 0&  0 & \cdots & 0 & 0\\
\end{block}
\end{blockarray}
\]

With these recursions in place, the same strategy of the proof of the $q$ even case can be applied to obtain the generating function for the $q$ odd case from the theorem statement. \end{proof}
As a direct corollary to the theorem, we can find the growth series of the tilings generated by regular $4n$- and $(4n+2)$-gon surfaces. 

\begin{corollary}
    The growth series of the tilings associated to surfaces created by gluing the opposite sides of the polygon together are as follows. 

    \begin{enumerate}
            \item For the $4n$-gon surface, $$f(x) = \frac{1 + 2x  + \ldots + 2x^{2n-1} + x^{2n}}{1 - (4n-2)x - \ldots - (4n-2-2)x^{2n-1} + x^{2n}}.$$
        \item For the $(4n+2)$-gon surface, $$f(x) = \frac{1 + 2x + \ldots + 2x^{n-1} + 4x^{n}+ 2x^{n + 1}+ \ldots + 2x^{2n-1} + x^{2n}}{1  - (4n)x - \ldots -(4n)x^{n-1} - (4n-2)x^{n} - (4n)x^{n + 1}- \ldots - (4n)x^{2n-1} + x^{2n}}.$$
    \end{enumerate}
\end{corollary}

\begin{proof}
    For the $4n$-gon surface, the universal cover is tiled by $4n$-gons, $4n$ meeting at a vertex. Thus, we have that $(p,q) = (4n,4n)$. For the $(4n+2)$-gon surface, the cover is tiled by $(4n+2)$-gons, $(2n+1)$ meeting at a vertex, so $(p,q) = (4n+2, 2n+1)$. We then substitute these $p$ and $q$ into Theorem~\ref{thm:tiling_growth} to get the corollary. 
\end{proof}

%% file: language_complexity.tex
\section{From Tiling Growth Rates to Billiard Language Complexity}\label{sec:billiard_complexity}

Recall that a generalized diagonal is a billiard trajectory between two vertices without any vertices in the interior of the trajectory. In this section, we work in the universal cover, a hyperbolic $(p,q)$-tiling. We will call the lift of a generalized diagonal to the $(p,q)$-tiling by the same name. Generalized diagonals in a tiling are \textbf{primitive} geodesic segments connecting two vertices, in that the interior of the trajectory does not pass through any vertices. 

We will relate two different notions of length of a generalized diagonal: the combinatorial length and the tiling length.

\begin{definition}[Combinatorial length of a generalized diagonal]
\label{def:gd_comb_length} The \textbf{combinatorial length} $\cl(\gamma)$ of a generalized diagonal $\gamma$ in a $(p,q)$-tiling is the number of tiles that $\gamma$ minus its endpoints passes through, which is also one more than the number of edges of the tiling that $\gamma$ minus its endpoints passes through. 
\end{definition}

We note that the combinatorial length of a generalized diagonal in a hyperbolic polygonal billiard table (from Definition~\ref{def:combinatorial_length}) is the same as the combinatorial length of its lift to the $(p,q)$-tiing.  Given a generalized diagonal $\gamma$, we can naturally obtain a tiling path by considering the sequence of tiles that $\gamma$ minus its endpoints passes through, allowing us to define the tiling length for generalized diagonals.

\begin{definition}[Tiling length of a generalized diagonal]
\label{def:gd_tiling_length} 
Given a generalized diagonal $\gamma$, we say that the \textbf{starting tile} and the \textbf{ending tile} of $\gamma$ are the first and last tiles that $\gamma$ minus its endpoints passes through. The \textbf{tiling length} $\tl(\gamma)$ of $\gamma$ is the tiling length of the tiling path obtained from $\gamma$ (see Definition~\ref{def:tiling_path}) between the starting and ending tile of $\gamma$, which is one fewer than the number of tiles in the tiling path.
\end{definition}


By Theorem~\ref{thm:GutTab}, the language complexity of a hyperbolic polygonal billiard can be determined by counting generalized diagonals by their combinatorial length. In this section, we will relate these counts to the counts of generalized diagonals by their tiling length, which we can then relate to the counts of tiles by tiling distance. We understand the asymptotics of these last counts via the tiling growth asymptotics of Theorem~\ref{thm:tiling_growth}. This will allow us to directly relate or to bound the language complexity growth rate by the tiling growth rate, depending on whether $q$ is even or odd.

\subsection{Relating tiling length to combinatorial length when $q$ is even}

Our main goal in this section will be to prove that the combinatorial length of a geodesic segment between vertices is one more than its tiling length (in the $q$ even case) by relating both of these quantities precisely to the number of \emph{edge geodesics} (defined below) separating the two vertices of the generalized diagonal. 

\begin{definition}[Edge geodesic]
\label{def:edge_geodesic} In a hyperbolic $(p,q)$-tiling with $q$ even, every edge $e$ extends via edges of a tiling to a bi-infinite geodesic $g_e$, which we call the \textbf{edge geodesic} of $e$.
\end{definition}

First, we extend the notion of combinatorial and tiling lengths to a geodesic segment between two vertices that is not necessarily a generalized diagonal (i.e., when the segment is not primitive). To do this, we first define a nudge of a geodesic segment, so that we may resolve issues that arise when a geodesic segment passes through a vertex of the tiling.

\begin{definition}[Nudged geodesic segment]
\label{def:nudge} Let $\gamma$ be a geodesic segment that passes through vertices $v_1, v_2, \ldots$ in its interior. A \textbf{nudge} of $\gamma$ around $v_i$ is a local deformation where we deform $\gamma$ so that it traverses a small counter-clockwise half-circle through tiles adjacent to $v_i$, as depicted in Figure~\ref{fig:nudge}. We defined the \textbf{nudged geodesic segment} $\gamma'$ of $\gamma$ as the path where we deform $\gamma$ by nudges at every $v_i$. If $\gamma$ does not pass through any vertices in its interior, then the nudged $\gamma$ is just $\gamma' = \gamma$. 
\end{definition}

To any geodesic segment $\gamma$, we can then associate a tiling path, and define the combinatorial and tiling length of $\gamma$. 

\begin{definition}[Tiling path of a geodesic segment]
\label{def:TilingPathGeodesicSegment}
Given a geodesic segment $\gamma$, consider any nudged segment $\gamma'$. Then, the tiling path of $\gamma$, denoted by $\mathcal{A}_\gamma$ is the sequence of tiles that $\gamma'$ passes through the interiors of. In the case where $\gamma$ is a portion of an edge geodesic, we do not double-list the tiles that are the last tile traversed by one nudge and the first tile traversed by the subsequent nudge (see Figure~\ref{fig:nudge}
\end{definition}

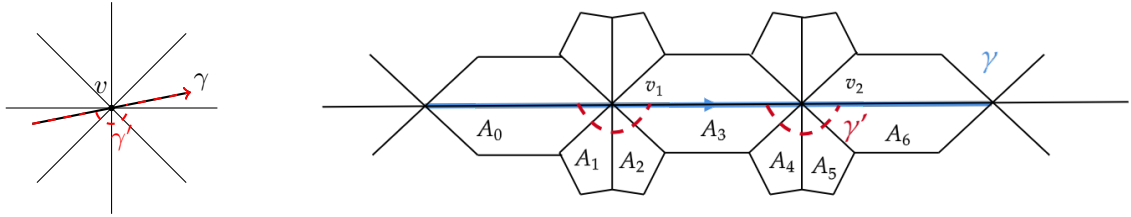
\begin{figure}[h]
    \centering

\begin{tikzpicture}[scale=0.7]
\draw (-2,0) -- (2,0);
\draw (0,-2) -- (0,2); 
\draw (-1.41,1.41) -- (1.41,-1.41);
\draw (-1.41, -1.41) -- (1.41, 1.41);
\draw[thick, ->] (-1.5, -.3) -- (1.5, .3); 
\draw[thick, dashed, color = red] (-1.5, -.3) -- (-.5, -.1); 
\draw[thick, dashed, color = red, ->] (.5, .1) -- (1.5, .3);
\draw [color = red, dashed, thick, domain=195:375] plot ({.3*cos(\x)}, {.3*sin(\x)});
\node at (0.2, -.55) {\textcolor{red}{$\gamma'$}}; 
\node at (1.7, .5) {$\gamma$}; 
\node at (-.2, .4) {$v$};
\draw[fill] (0,0) circle [radius = 0.05];
\end{tikzpicture} \hspace{1cm}
\includegraphics[width=11cm]{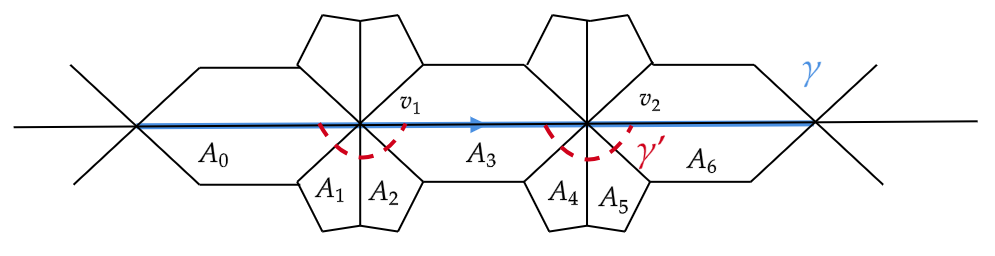}
\caption{Left: a portion of a geodesic segment $\gamma$ and nudged segment $\gamma'$. Right: A geodesic segment $\gamma$ that is a portion of an edge geodesic, and a nudge $\gamma'$, and the corresponding tiling path $\mathcal{A}_\gamma = (A_0, A_1, \ldots, A_6)$.} 
\label{fig:nudge}
\end{figure}

\begin{definition}[Combinatorial and tiling lengths of a geodesic segment]\label{def:CombAndTilingLengthGeodesicSegment}
Let $\gamma$ and let $\gamma'$ be a nudge of $\gamma$. The \textbf{tiling length} $\tl(\gamma)$ is then defined as the length of the associated tiling path $\mathcal{A}_\gamma$. The \textbf{combinatorial length} $\cl(\gamma)$ is then one more than the number of edges of the tiling that $\gamma'$ minus its endpoints passes through transversely. If $\gamma$ is a portion of an edge geodesic $g$, then $\gamma'$ will traverse portions of the edges of $g$, but those edges will not count toward the combinatorial length. 
\end{definition}

\begin{remark}
We note that we made a choice to apply just counter-clockwise nudges to geodesic segments $\gamma$. If we had chosen instead to apply clockwise nudges instead, then the tiling path $\mathcal{A}_\gamma$ would be different, but $\tl(\gamma)$ and $\cl(\gamma)$ would not have changed, because of the symmetry in the number of edge geodesics crosses when traversing halfway around a vertex $v$ in either direction. 
\end{remark}

Having extended the notion of combinatorial/tiling lengths, we now establish the following relationship between minimal tiling paths (see Definition~\ref{def:tiling_path}) and edge geodesics.

\begin{proposition}
\label{prop:minimal_tiling_path_characterization}
    Let $\mathcal{A} = (A_0, \ldots, A_n)$ be a tiling path. Then, $\mathcal{A}$ is a minimal tiling path between tiles $A_0$ and $A_n$ if and only if it crosses every edge geodesic separating $A_0$ and $A_n$ exactly once and it does not cross any other edge geodesic. Consequently, minimal tiling paths do not double cross any edge geodesic. 
\end{proposition}

\begin{proof}

We assume that $\mathcal{A} = (A_0, \ldots, A_n)$ is a minimal tiling path. By the connectedness of the tiling path, it must intersect every edge geodesic separating $A_0$ and $A_n$ an odd number of times and every edge geodesic not separating $A_0$ and $A_n$ an even number of times. 

\begin{figure}[h!]
\centering
\includegraphics[width=8cm]{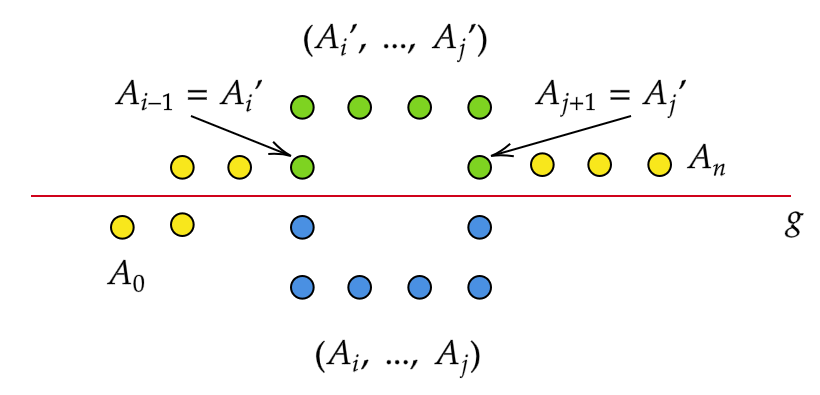}
\caption{A way to turn a tiling path in a $(p,q)$-tiling into a shorter tiling path if it does not minimally intersect an edge geodesic $g$. Here, dots represent tiles. The original tiling path is $\calA = (A_0, \dots, A_n)$ and the shortened tiling path is $(A_0, \dots, A_{i-2}, A_i', \dots, A_j', A_{j+2}, \dots, A_n)$.}
\label{fig:tiling_path}
\end{figure}

Suppose that $\mathcal{A}$ crossed an edge geodesic $g$ at two distinct places: from $A_{i-1}$ to $A_i$ and $A_{j}$ to $A_{j+1}$, as shown in Figure~\ref{fig:tiling_path}. Since the tiling is symmetric over $g$, we can reflect the tiling subpath $(A_i, A_{i+1}, \ldots, A_j)$ over $g$ to get a different subpath $(A_i', A_{i+1}', \ldots, A_j')$ where $A_i' = A_{i-1}$ and $A_j' = A_{j+1}$. Then, the tiling path $\mathcal{A}' = (A_1, \ldots, A_{i-2}, A_i', \ldots, A_j', A_{j+2}, \ldots, A_n)$ is a tiling path connecting $A_0$ and $A_n$ which is length $2$ shorter than $\mathcal{A}$, contradicting the minimality of $\mathcal{A}$. Thus, $\mathcal{A}$ crosses $g$ once if $g$ separates $A_0$ and $A_n$ and zero times otherwise. 

In contrast, we notice that every move from tile $A_i$ to successive tile $A_{i+1}$ crosses exactly one edge geodesic, and every edge geodesic separating $A_1$ and $A_n$ must be crossed at least once. Thus, a tiling path $\mathcal{A} = (A_0,\ldots, A_n)$ that crosses each edge geodesic separating $A_0$ and $A_n$ exactly once and all other edge geodesics zero times must be minimal. 
\end{proof}

As a consequence of Proposition~\ref{prop:minimal_tiling_path_characterization}, the tiling length of a minimal tiling path is precisely the number of edge geodesics separating the initial and final tiles. 

We now establish our main result of this subsection: the combinatorial length of a geodesic segment between vertices is one more than its tiling length.

\begin{proposition} 
\label{prop:q-evenCombEqualsTiling}
Let $\gamma$ be a geodesic segment between vertices in a hyperbolic $(p,q)$-tiling with $q$ even. With combinatorial length $\cl(\gamma)$ and tiling length $\tl(\gamma)$ as defined in Definition~\ref{def:CombAndTilingLengthGeodesicSegment}, $$\cl(\gamma) = \tl(\gamma)+1.$$
\end{proposition}

\begin{proof}

    Let $\gamma$ be a geodesic segment in a hyperbolic $(p,q)$-tiling with $q$ even. Since hyberbolic geodesics intersect minimally, $\gamma$ crosses every edge geodesic separating the end vertices exactly once. If we nudge $\gamma$ to $\gamma'$ as in Definition~\ref{def:nudge}, $\gamma$ and $\gamma'$ intersect exactly the same set of edge geodesics transversely. Thus $\cl(\gamma)$ is equal to one more than the number of edge geodesics separating the endpoints of $\gamma$.

    Now, given $\gamma$, let $\calA_{\gamma} = (A_1, \dots, A_n)$ be the tiling path associated to $\gamma$ (see Definition~\ref{def:TilingPathGeodesicSegment}). Since edge geodesics do not pass through the interiors of tiles, the set of edge geodesics separating $A_1$ and $A_n$ are exactly the set of edge geodesics separating the endpoints of $\gamma$. These edge geodesics, by the definition of $\mathcal{A}_\gamma$, are exactly the edge geodesics generated by the edges separating successive tiles in $\mathcal{A}_\gamma$. Each edge geodesic separating $A_1$ and $A_n$ is intersected exactly once because $\gamma$ and its nudge $\gamma'$ intersects these edge geodesics exactly once. Thus, the tiling path $\mathcal{A}_\gamma$ is minimal by Proposition~\ref{prop:minimal_tiling_path_characterization}, and $\tl(\gamma)$ is then equal to the number of edge geodesics separating the endpoints of $\gamma$. The statement of the proposition then follows.
\end{proof}

\subsection{Relating tiling length to combinatorial length when $q$ is odd}

Contrary to when $q$ is even, edges of a $(p,q)$-tiling when $q$ is odd do not connect to form geodesics. We will define quasigeodesics called \textit{zigzags} that will play a similar role as edge geodesics in the $q$ even case, and will show that there is a relation between tiling length and the number of zigzags separating two tiles. 

For every edge in the tiling adjacent to some vertex, there are two edges adjacent to the same vertex that are nearly opposite our original edge which we call the \textit{opposite left edge} and the \textit{opposite right edge}.

\begin{definition}[Zigzag] \label{def:zigzag} In a hyperbolic $(p,q)$-tiling with $q$ odd, a \textbf{zigzag} to be a bi-infinite sequence of edges which alternate between opposite left and opposite right edges, as depicted in Figure~\ref{fig:zigzag_crossing}. 
\end{definition}

We note that each edge of the tiling belongs to exactly two zigzags. Zigzags are also quasigeodesics and split the hyperbolic plane into two halves: 

\begin{restatable}[Zigzags are quasigeodesics]{proposition}{zigzagQuasiGeod}\label{prop:zigzag_halfspace}
    Let $\zeta$ be a zigzag in a hyperbolic $(p,q)$-tiling with $q$ odd. Then there exists a hyperbolic geodesic that passes through all the midpoints of the edges of the zigzag. Consequently, zigzags are quasigeodesics and hence divide the hyperbolic plane into two unbounded halves.     
\end{restatable}
\begin{proof}
   The proof uses elementary hyperbolic geometry arguments. Hence, we postpone the proof to Appendix~\ref{appdx:hyperbolictilings}.
\end{proof}

Unlike hyperbolic geodesics which either intersect at exactly one point or do not intersect at all, two zigzags can intersect along a whole edge of the tiling, as depicted in Figure~\ref{fig:zigzag_crossing}

\begin{figure}[h]
\centering
\includegraphics[width=7cm]{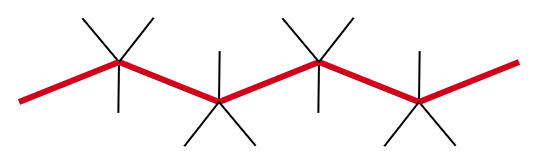} \vspace{.5cm} \includegraphics[width=7cm]{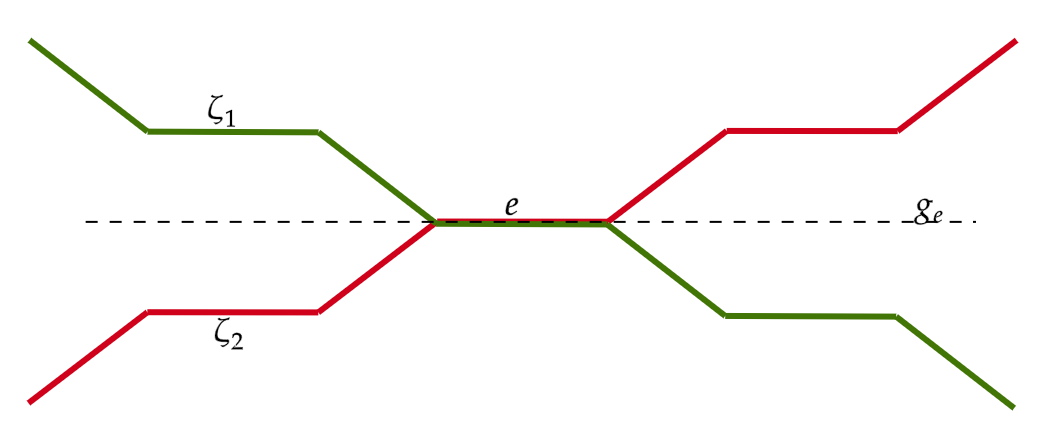}
\caption{Left: A schematic diagram of part of a zigzag path of edges (in bold) in a $(p,q)$-tiling with $q=5$. Right: Two zigzags $\zeta_1$ and $\zeta_2$ that intersect at an edge $e$.}
\label{fig:zigzag_crossing}
\end{figure}

However, besides at that edge, such zigzags do not intersect as we state below:

\begin{restatable}[Behavior of intersecting zigzags]{proposition}{zigzagTransverse}
    \label{prop:zigzag_crossing} Let $\zeta_1$ and $\zeta_2$  be two zigzags that share an edge $e$, in a hyperbolic $(p,q)$-tiling with $q$ odd. Then, $\zeta_2$ is the reflection of $\zeta_1$ across the edge geodesic $g_e$ that contains $e$. Each $\zeta_i$ crosses $g_e$ at $e$ and otherwise does not intersect $g_e$. As a result, $\zeta_1$ and $\zeta_2$ split the tiling into four distinct quarter-spaces. 
\end{restatable} 

\begin{proof}
    See Appendix~\ref{appdx:hyperbolictilings}.
\end{proof}

For $q$ odd, zigzags will play the role that edge geodesics played in the $q$ even case. That is, we will relate the tiling length and combinatorial length of a generalized diagonal to the number of zigzags that separate the endpoints of the generalized diagonal. Towards this, we are now ready to establish the following proposition stating that minimal tiling paths intersect zigzags minimally.

\begin{proposition} \label{prop:zigzig_tiling} Let $\mathcal{A} = (A_0, \ldots, A_n)$ be a minimal tiling path connecting its initial and terminal tiles. Then, this path intersects every zigzag path minimally. That is, the tiling path intersects a zigzag path once if $A_0$ and $A_n$ are on opposite sides of the zigzag and zero times otherwise. 
\end{proposition}

 \begin{proof}
Let $\zeta$ be a zigzag and $\mathcal{A}$ be a minimal tiling path. If suffices to show that $\mathcal{A}$ cannot intersect $\zeta$ more than once. We will induct on the length of $\mathcal{A}$.

For the base case, assume the length of a minimal tiling path $\mathcal{A}$ is $1$. Then, $\mathcal{A}$ consists of two tiles and trivially cannot double cross any zigzags. Now, let $n \in \NN$ and assume as induction hypothesis that no minimal tiling paths of length $\leq n$ can double cross any zigzag. 

Consider a minimal tiling path $\mathcal{A} = (A_0, \dots, A_{n+1})$ of length $n+1$ and assume towards a contradiction that it double crosses a zigzag $\zeta$ twice. We can assume that the double crossing occurs at the edge between $A_0$, $A_1$ and the edge between $A_{n}$, $A_{n+1}$ without any crossing in between or else we would have a subpath of $\mathcal{A}$ (and hence minimal) of length $\leq n$ that double crosses $\zeta$, violating the induction hypothesis. Note that this means $A_0$ and $A_{n+1}$ are on the same side of $\zeta$.
 
Let $e_1$ be the edge between $A_0$, $A_1$ where the first crossing occurs and let $e_m$ be the edge between $A_{n}, A_{n+1}$ where the second crossing occurs. Label the rest of the edges of $\zeta$ in between $e_1$ and $e_m$ as $e_2, \dots, e_{m-1}$. 

We now deal with two cases. 

\textbf{Case 1:} The first case is when the edge geodesic $g_{e_2}$ containing $e_2$ bisects the tile $A_0$. See Figure~\ref{fig:twowaysAcrossesGeod}.

Define $\rho$ to be the hyperbolic reflection across the edge geodesic $g_{e_2}$. This reflection preserves the tiling and sends $A_0$ to itself. Hence, the image of $\mathcal{A}$ is another tiling path $\rho(\mathcal{A}) := (\rho(A_0) = A_0, \rho(A_1), \dots, \rho(A_{n+1}))$. 

Note that $\zeta$ and $\rho(\zeta)$ both contain the edge $e_2$. So, by Proposition~\ref{prop:zigzag_crossing}, they do not share any other edge. Hence, $\rho(e_1) \not \in \zeta$. As $\rho(e_1)$ is the edge between $\rho(A_0) = A_0$ and $\rho(A_1)$, the tiling path $\rho(\mathcal{A})$ does not cross $\zeta$ between $\rho(A_0)$ and $\rho(A_1)$. Since $\rho(A_0) = A_0$, we obtain that $A_0$ and $\rho(A_1)$ are on the same side of $\zeta$. As $A_{n+1}$ and $A_0$ are on the same side of $\zeta$, we then observe that $A_{n+1}$ and $\rho(A_1)$ are on the same side of $\zeta$.

Next, we note that the path $(A_1, \dots, A_{n+1})$ must cross the geodesic $g_{e_2}$ since $A_1$ and $A_{n+1}$ are on its opposite sides. Now, this crossing can occur in two ways: Either there is a tile $A_i$, with $2 \leq i \leq n$ that gets bisected by $g_{e_2}$ or there are two tiles $A_{i-1}$, $A_{i}$ such that their shared edge belongs to $g_{e_2}$. See Figure~\ref{fig:twowaysAcrossesGeod}.

\begin{figure}[h!]
\centering
    \begin{subfloat}[Tile $A_i$ gets bisected by $g_{e_2}$]{
    \includegraphics[scale=0.9]{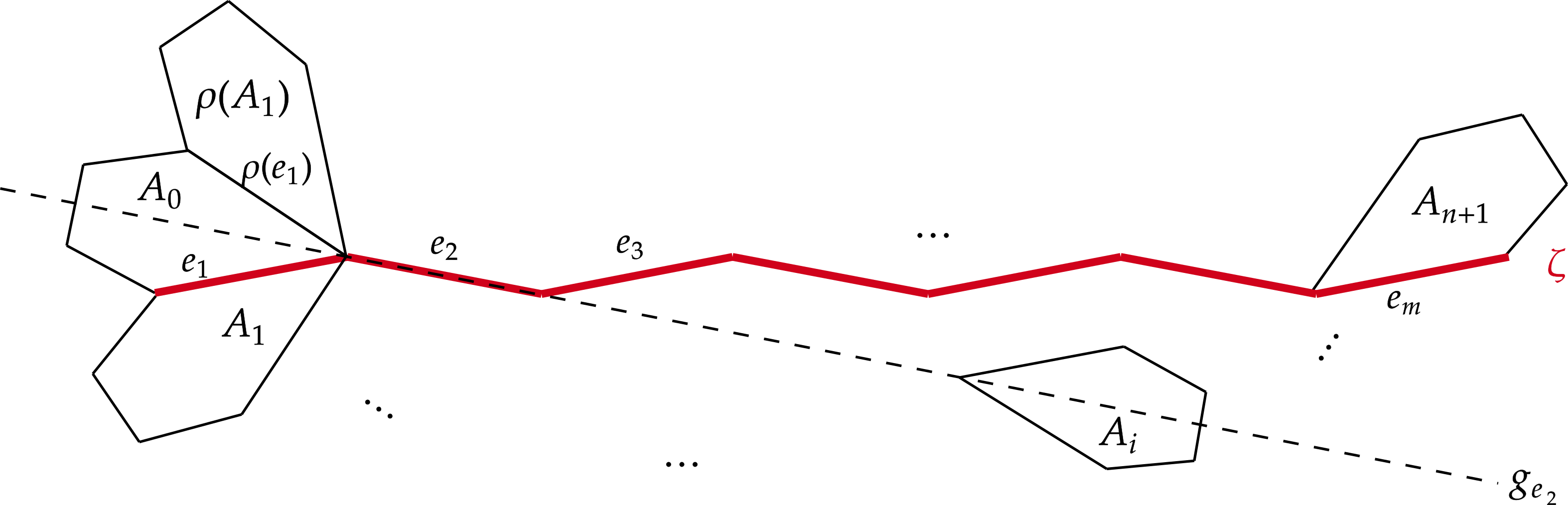}
    \vspace{1.6cm}}  
    \end{subfloat}
    \begin{subfloat}[Edge geodesic $g_{e_2}$ pass through the edge between $A_{i-1}$ and $A_{i}$]{
    \includegraphics[scale=0.9]{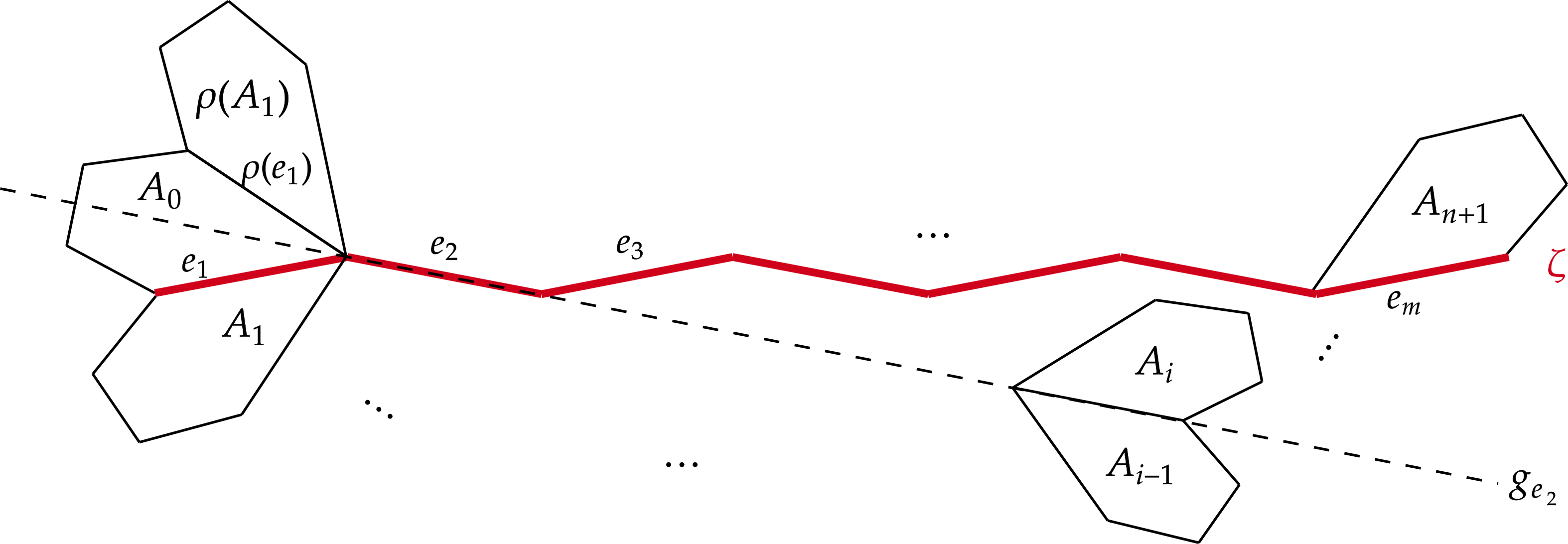}
    \vspace{1.6cm}}  
    \end{subfloat}
    \caption{In the setting of Case 1 where the edge geodesic $g_{e_2}$ bisects tile $A_0$, there are two ways in which the tiling path $\mathcal{A}$ crosses $g_{e_2}$.}
    \label{fig:twowaysAcrossesGeod}
\end{figure}

If there is a tile $A_i$ with $2 \leq i \leq n$ that gets bisected, we have $\rho(A_i) = A_i$, and that $A_i$ must be on the opposite side of $\zeta$ as $A_0$ and $A_{n+1}$. So, we then consider the tiling path $\mathcal{B} = (\rho(A_1), \rho(A_{2}), \dots, \rho(A_i) = A_i, A_{i+1}, \dots, A_{n+1})$. This path has length $\leq n$. Its extension, \\$(A_0, \rho(A_1), \dots, \rho(A_i) = A_i, A_{i+1}, \dots, A_{n+1})$ is minimal since it is a length $\leq n+1$ path between $A_0$ and $A_{n+1}$. Hence, $\mathcal{B}$ is minimal of length $\leq n$. However, as we showed $\rho(A_1)$ and $A_{n+1}$ are on the same side of $\zeta$, the path must double cross $\zeta$ to reach $A_i$, a contradiction to the induction hypothesis.

If there are two tiles $A_{i-1}$, $A_{i}$, such that their shared edge belongs to $g_{e_2}$ then $\rho(A_{i-1}) = A_{i}.$ If the shared edge is not $e_2$, then $A_{i-1}$ and $A_i$ are both on the opposite side of $\zeta$ as $A_0$ and $A_{n+1}$. So we consider the tiling path $\mathcal{B} = (\rho(A_1), \rho(A_{2}), \dots, \rho(A_{i-1}) = A_{i}, A_{i+1}, \dots A_{n+1})$. Similar to the former case, this path has length $\leq n$, is minimal and must double cross $\zeta$, a contradiction to the induction hypothesis. 

If the shared edge between $A_i$ and $A_{i-1}$ is $e_2$, then $A_0, A_1, \ldots, A_i$ double-crosses $\zeta$ at adjacent edges in the zigzag $e_1$ and $e_2$. Because of how a $(p,q)$-tiling can be iteratively constructed (see Construction \ref{con:tiling}), the unique minimal tiling path between $A_0$ and $A_i$ is a vertex traversal of length $\frac{q}{2} - \frac{1}{2}$ that does not double cross $\zeta$, contradicting that $\mathcal{A}$ was minimal.

\textbf{Case 2:} The second case is when the edge geodesic $g_{e_2}$ containing $e_2$ does not bisect the tile $A_0$. Assume Case 1 is not satisfied. In this case, let $\zeta'$ be the other zigzag that contains the edge $e_1$. See Figure~\ref{fig:case2zigzag}. 
\begin{figure}[h!]
    \centering
    \includegraphics[scale=0.9]{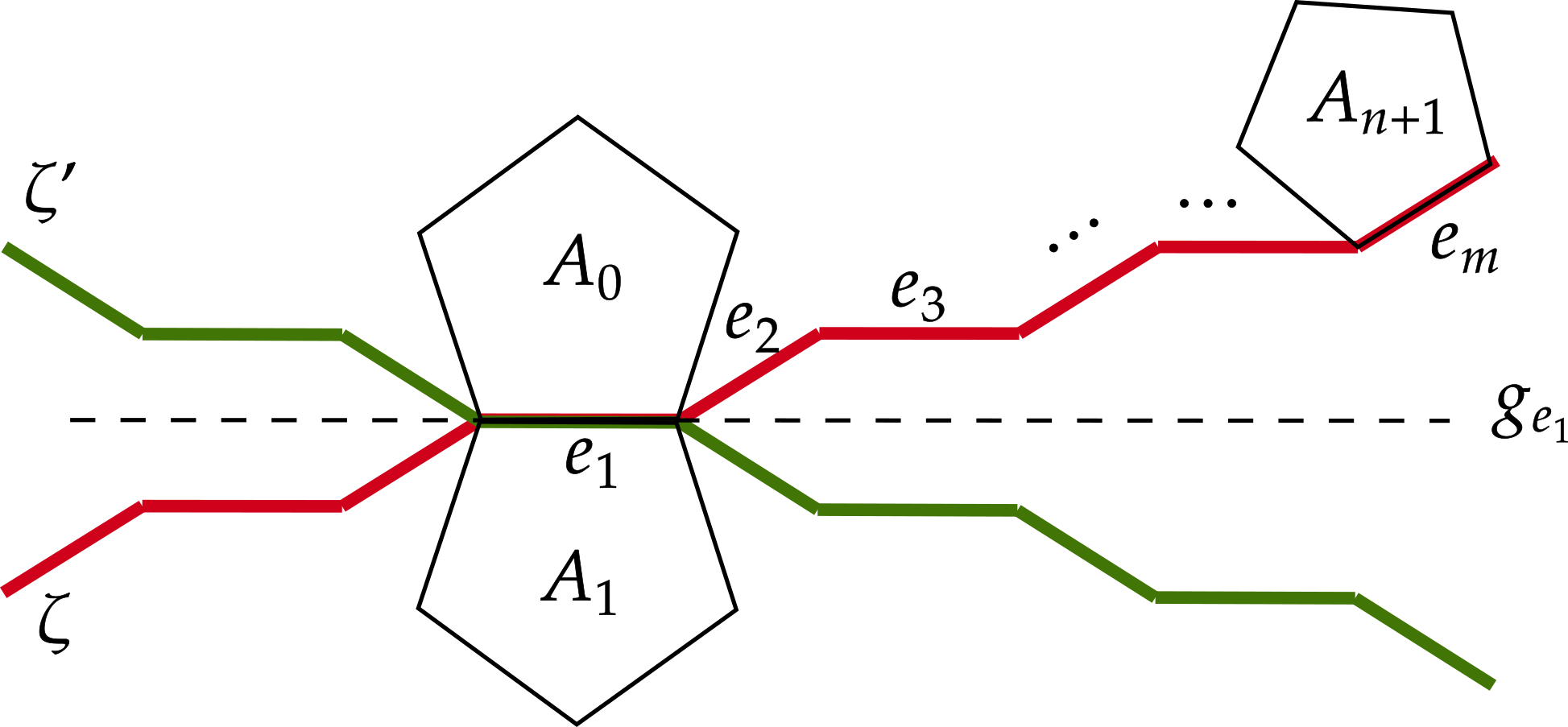}
    \caption{In Case 2, the edge geodesic $g_{e_2}$ does not go through tile $A_0$ so we consider the zigzag $\zeta'$.}
    \label{fig:case2zigzag}
    
\end{figure}

By Proposition~\ref{prop:zigzag_crossing}, $\zeta$ and $\zeta'$ do not intersect anywhere else other other than $e_1$. In particular, all the edges $e_2, \dots, e_m$ are on the same side of $\zeta'$ as $A_0$. As $e_m$ belongs to  $A_{n+1}$ we see that $A_0$ and $A_{n+1}$ are on the same side of $\zeta'$.   

Now consider the path $\mathcal{A}$. It crosses $\zeta'$ once at edge $e_1$. As $A_1$ and $A_{n+1}$ are on opposite sides of $\zeta'$, $\mathcal{A}$ must cross $\zeta'$ again at another edge. However, as $\zeta$ and $\zeta'$ only share edge $e_1$, the edge $e_m$ (i.e. the edge between the final tiles of $\mathcal{A}$ where it crosses $\zeta$) does not belong to $\zeta'$. Hence, the subpath $(A_0, A_1, A_2, \dots,A_{n})$ of $\mathcal{A}$ of length $\leq n$ is minimal and must cross $\zeta'$ twice, a contradiction to the induction hypothesis.

Hence, a minimal tiling path $\mathcal{A}$ cannot double cross a zigzag. \end{proof}

The next proposition serves as a generalization of the idea that when $q$ is even, the tiling length between two tiles can be computed by counting the number of edge geodesics that separate the two tiles. 

\begin{proposition}
\label{prop:tiling-distance-zigzag}
    Given a $(p,q)$ tiling with $q$ odd, let $A$ and $B$ be two tiles. Let $\td(A,B)$ be the tiling length between the two tiles and let $\mathrm{z}(A,B)$ be the number of zigzags separating $A$ and $B$. Then, $$2\td(A,B) = \mathrm{z}(A,B).$$ 
\end{proposition}

\begin{proof}

Let $\mathcal{A}$ be a minimal tiling path between tiles $A$ and $B$. By Proposition~\ref{prop:zigzig_tiling}, $\mathcal{A}$ intersects each zigzag that separates $A$ and $B$ exactly once. Each edge that the tiling path crosses is part of exactly two of these separating zigzags. Since the number of edges crossed by the minimal tiling path is exactly the tiling distance, we have that $2\td(A,B) = \mathrm{z}(A,B)$. 
\end{proof}

Using the above relation of tiling distance with number of zigzag intersections, we can now deduce the following proposition which relates the combinatorial length of a generalized diagonal with the tiling distance between the first and last tiles that the generalized diagonal passes through.

\begin{proposition}\label{prop:q-oddCombboundsTiling} Consider a hyperbolic $(p,q)$-tiling with $q$ odd. The tiling distance between two tiles is upper bounded by one less than the combinatorial length of a geodesic connecting two interior points in these tiles, is lower bounded by $\frac{q-1}{q+1}$ times this quantity. In particular, for any generalized diagonal $\gamma$, \[\tl(\gamma) + 1 \leq \cl(\gamma) \leq  \frac{q+1}{q-1} \tl(\gamma)+1.\]  
\end{proposition}

As a tool for proving this proposition, we will use the following hyperbolic geometry lemma, whose proof will be deferred until the appendix. 

\begin{restatable}{lemma}{ZigzigIntersections}
\label{lem:zigzag_intersections}
Consider a hyperbolic $(p,q)$-tilling with $q$ odd. Let $\gamma$ be a geodesic that intersects two consecutive edges $e_1$ and $e_2$ on a zigzag $\zeta_1$, where $e_1$ and $e_2$ meet at a vertex $v$. Let $\zeta_2\neq \zeta_1$ be a zigzag containing vertex $v$. Then $\gamma$ intersects $\zeta_2$ exactly once.
\end{restatable}

With this, we can now prove the Proposition.

\begin{proof}[Proof of Proposition~\ref{prop:q-oddCombboundsTiling}]
Let $A$ and $B$ be two tiles, and let $p_A$ and $p_B$ be points in the interiors of tiles $A$ and $B$ respectively. By Proposition~\ref{prop:tiling-distance-zigzag}, $\td(A,B) = \frac12\mathrm{z}(A,B)$, where $\td(A,B)$ denotes the tiling distance and $\mathrm{z}(A,B)$ denotes the number of zigzags separating $A$ and $B$. If $\gamma_{p_A, p_B}$ is a geodesic path connecting $p_A$ and $p_B$, then $\gamma_{p_A, p_B}$ must intersect every zigzag separating $A$ and $B$ at least once, but may also contain \textbf{excess intersections}, intersections with a zigzag beyond the minimally required zero or one intersection. Assuming that the intersections happen at the interiors of edges, every intersection is on exactly two zigzags (since each edge is part of exactly two zigzags). Thus, we have that $\mathrm{cl}(\gamma_{p_A, p_B}) \geq \frac12 \mathrm{z}(A,B)+1 = \td(A,B)+1$.

To prove the other bound, we upper bound the number of excess intersections of $\gamma_{p_A, p_B}$ with zigzags. We note that all excess intersections of $\gamma_{p_A,p_B}$ with a zigzag $\zeta$ are successive by Proposition~\ref{prop:consecutivezigzagintersection}, and that there are an even number of them since a zigzag divides the plane into two connected pieces by Proposition~\ref{prop:zigzag_halfspace}.

\begin{figure}[h]
    \centering
    \includegraphics[scale=0.9]{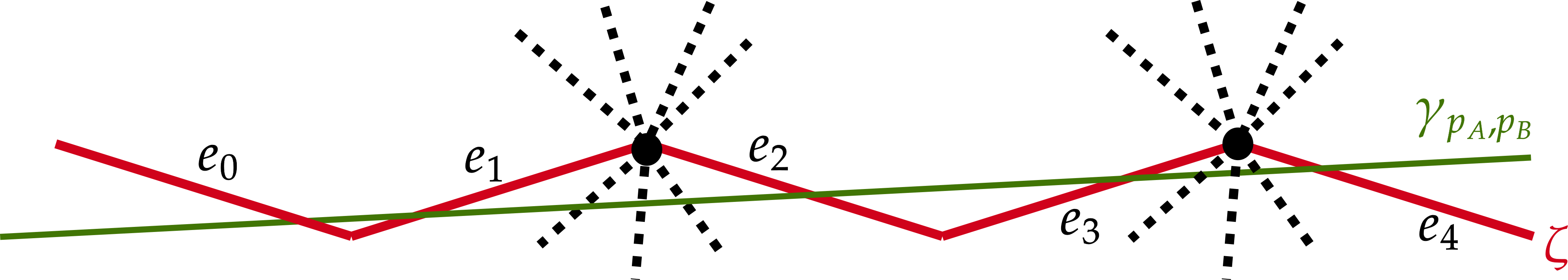}
    \caption{A schematic for a situation where the geodesic $\gamma_{p_A,p_B}$ crosses a zigzag $\zeta$ separating the tiles $A$ and $B$ multiple times. Since $\zeta$ separates $A$ and $B$, the excess intersections are even in number (in this illustration, that number is 4). We then pair the excess intersections edges (in this case $(e_1, e_2)$ and $(e_3, e_4)$ are paired) so that for each pair, we associate $\frac{q-3}{2}$ in-between edges (in this illustration $q=9$ shown using the dashed edge-crossings) that are not excess for any zigzag.}
    \label{fig:zigzag_double_cross}
\end{figure}
Given a zigzag $\zeta$ with which $\gamma_{p_A, p_B}$ has excess intersections, we pair those excess intersections into consecutive neighboring edges $e_1$ and $e_2$ meeting at a vertex $v$ on $\zeta$ (see Figure~\ref{fig:zigzag_double_cross}). By applying Lemma~\ref{lem:zigzag_intersections}, we then have that for every such pair of excess intersections (corresponding to edges that are only excess intersections for that zigzag $\zeta$), there are $\frac{q-3}{2}$ additional associated edge intersections, corresponding to the edges adjacent to $v$ that $\gamma$ crosses between $e_1$ and $e_2$, that are not excess intersections for any zigzag. We then have that $$\frac{\text{total excess intersections of } \gamma_{p_A, p_B}}{\text{total edge intersections of } \gamma_{p_A, p_B}} \leq \frac{2}{2 + \frac{q-3}{2}} = \frac{4}{q+1},$$ and so $$\text{total excess intersections of } \gamma_{p_A, p_B} \text{ with zigzags} \leq \frac{4}{q+1} (\cl(\gamma_{p_A, p_B})-1).$$

We then have that
\begin{align*} 
\cl(\gamma_{p_A, p_B}) & = \frac12 \left(z(A,B) + \text{excess intersections with zigzags}\right) +1 \\
& \leq \frac12\left( 2 \td(A,B)  + \frac{4}{q+1} (\cl(\gamma_{p_A, p_B})-1)\right) + 1. 
\end{align*}
By rearranging the inequality, we get that $\cl(\gamma_{p_A, p_B}) \leq \frac{q+1}{q-1}\td(A,B)+1$.

Then, given a generalized diagonal $\gamma$ and letting $A$ and $B$ be the first and last tiles intersected by the generalized diagonal, it follows that $\tl(\gamma)+1 \leq \cl(\gamma) \leq \frac{q+1}{q-1}\tl(\gamma)+1$. 
\end{proof}

\subsection{Language complexity of billiards in hyperbolic polygons}

We now wish to use the relation between the tiling length and combinatorial length of geodesic segments to relate the growth the numbers of tiles by tiling length and geodesic segments by combinatorial length. We will then relate the growth of such geodesic segments in the tiling to the growth of generalized diagonals in a polygon, which will subsequently yield Theorem~\ref{thm:complexity_evenodd}.

\begin{definition}[Tiling distance and combinatorial length counts]
\label{def:counts} Consider a hyperbolic $(p,q)$-tiling and fix a base tile $P$ and a base vertex $v$. Let $\ntd(n)$ denote the number of tiles of tiling distance exactly $n$ away from $P$. Let $\ncl(n)$ denote number of geodesic segments of combinatorial length $n$ with  vertex endpoints such that the one of the vertices is the base vertex $v$. Let $\nclprim(n)$ be such geodesic segments that are primitive (i.e. ones that do not have a vertex in the interior). 
\end{definition}

From Theorem~\ref{thm:tiling_growth} we see that $\ntd(n) =  g(n)\alpha^n$ where $\alpha>1$ is the largest eigenvalue of the denominator polynomial from Theorem~\ref{thm:tiling_growth} and $g(n)$ is subexponential in $n$. In other words, we have that $$\lim_{n\rightarrow \infty} \frac{\log_\alpha(\ntd(n))}{n}=1.$$

Using this notation, we first prove the following lemma relating $\ncl(\cdot)$ and $\ntd(\cdot)$ in the $q$ even case and $\nclprim(\cdot)$ and $\ntd(\cdot)$ in the $q$ odd case.

\begin{lemma}\label{lem:gd_tl_relation}
    Consider a hyperbolic $(p,q)$-tiling. 
\begin{enumerate}
    \item Let $q$ be even. There exist positive constants $C_{p,q}$ and $K_{p,q}$ depending only on $p$ and $q$ such that for any $k \in \NN$,
    $$ \ncl(k) \leq C_{p,q} \cdot \max_{|n-k+1| \leq q} \{\ntd(n)\} \,\,\, \text{ and } \,\,\, \ntd(k) \leq K_{p,q} \cdot \max_{|n-k-1| \leq q} \{\ncl(n)\}.$$
\item Let $q$ be odd. There exist positive constants $C_{p,q}$ and $K_{p,q}$ depending only on $p$ and $q$ such that for any $k\in \mathbb{N}$,
$$ \ncl(k) \leq  C_{p,q} \cdot k \cdot  \left(\max\left\{\ntd(n) : \frac{q-1}{q+1}(k-q-4) \leq n \leq k+q-1\right\}\right)$$
and $$\ntd(k) \leq K_{p,q}\cdot k  \cdot \left(\max \left\{\ncl(n): k-q+1 \leq n \leq \frac{q+1}{q-1} \cdot k+q+4\right\}\right).$$   
\end{enumerate}
\end{lemma}

\begin{proof}
    
Our strategy for this proof will be to relate geodesic segments with vertex endpoints to nearby tiling paths in a way where we can control the counts of such geodesics.

Consider a hyperbolic $(p,q)$-tiling. Let $P$ a base polygon and $v$ a base vertex. For each vertex $u \neq v$, let $\gamma_u$ be the geodesic segment between $v$ and $u$. 
We now deal with the two cases of $q$ being even and odd separately and within each case, we further distinguish between the $p \geq 4$ and $p=3$ cases.

\begin{enumerate}
    \item We first consider the case when $q$ is even, starting with when $p \geq 4$. Recall that $\tl(\gamma_u)$ is the tiling length between the first and last tiles that $\gamma_u$ passes through. Now, consider the map $\varphi$ from vertices to tiles in Proposition~\ref{prop:vertex_tile_mapping}. Recall that $\varphi$ is surjective, at most $p$ to $1$, and that each tile $\varphi(u)$ is adjacent to $u$. So, the first and last tiles of $\gamma_u$ are distance $\leq q/2$ from the tiles $P$ and $\varphi(u)$ respectively. Since tiling distance satisfies the triangle inequality, it then follows that  

$$\td(P,\varphi(u))- q \leq \tl(\gamma_u) \leq \td(P,\varphi(u)) + q$$

where $\td(P, \varphi(u))$ is the tiling distance between tiles $P$ and $\varphi(u)$. By Proposition~\ref{prop:q-evenCombEqualsTiling}, $\cl(\gamma_u) = \tl(\gamma_u)+1$ so that
\begin{equation}
\label{eq:cl_td_even}
\td(P,\varphi(u))-q + 1 \leq \cl(\gamma_u) \leq \td(P,\varphi(u)) + q + 1
\end{equation}
for each geodesic segment $\gamma_u$ counted in $\ncl(k)$. Since the map $\varphi$ from vertices to tiles is at most $p$ to $1$, we have that 
    \begin{equation} 
    \label{eq:ncl_bound_even}
    \ncl(k) \leq \sum_{n=k-q-1}^{k+q-1} p\cdot \ntd(n) \leq p \cdot (2q+1) \max_{|n-k+1|\leq q} \{\ntd(n)\}.
    \end{equation}

For the other bound, fix $k \in \NN$. Since $\varphi$ is surjective, any tile that is tiling distance $k$ away from the base tile is an image under the mapping $\varphi$ of a vertex $u$ with $\cl(\gamma_u)$ between $k-q+1$ and $k+q+1$.

Hence, we get the bound,
   \begin{equation} 
    \label{eq:ntd_bound_even}
    \ntd(k) \leq \sum_{n=k-q+1}^{k+q+1} \ncl(n) \leq (2q+1) \max_{|n-k-1|\leq q}\{\ncl(n)\}.
    \end{equation}
 
The argument for the $p=3$ case is similar. Here we use the map $\psi$ from tiles to vertices stated in Proposition~\ref{prop:vertex_tile_mapping}. The $q$ to $1$ and surjective nature of $\psi$ yields the two required bounds on $\ntd(k)$ and $\ncl(k)$ respectively.

\item The proof in the $q$ odd case follows the structure of the proof in the $q$ even case. In place of the equality $\cl(\gamma_u) = \tl(\gamma_u) +1$, we apply the inequality from Proposition~\ref{prop:q-oddCombboundsTiling}, $$\tl(\gamma_u) + 1 \leq \cl(\gamma_u) \leq \frac{q+1}{q-1}\tl(\gamma_u) + 1. $$ As a result, the inequality in Equation \eqref{eq:cl_td_even} is replaced with $$\td(P, \varphi(u)) - q + 1 \leq \cl(\gamma_u) \leq \frac{q+1}{q-1}(\td(P, \varphi(u)) + q) + 1 \leq \frac{q+1}{q-1}\td(P, \varphi(u)) + q + 4.$$
The inequalities in Equations \eqref{eq:ncl_bound_even} and \eqref{eq:ntd_bound_even} then become $$\ncl(k) \leq \sum_{n=\lceil \frac{q-1}{q+1}(k-q-4)\rceil}^{k+q-1} p\cdot \ntd(n) \leq p \cdot \left(k+q- \frac{q-1}{q+1}(k-q-4)\right)  \cdot \left(\max_{ \frac{q-1}{q+1}(k-q-4) \leq n \leq k+q-1} \{\ntd(n)\}\right)$$ 
and $$\ntd(k) \leq \sum_{n=k-q+1}^{\lfloor \frac{q+1}{q-1} k+q+4\rfloor} \ncl(n) \leq  \left(\frac{q+1}{q-1} k - k +2q+4\right) \cdot \left( \max_{k-q+1 \leq n \leq \frac{q+1}{q-1} k+q+4} \{\ncl(n)\}\right)$$ respectively. The proposition then follows by noting that we can upper bound the constants by $C_{p,q}\cdot k$ and $K_{p,q}\cdot k$ by noting that $k \geq 1$. As in the $q$ even case, the $p=3$ case follows similarly when we use the map $\psi$ from tiles to vertices from Proposition~\ref{prop:vertex_tile_mapping}. 
\end{enumerate}
\end{proof}

Since our goal in this section is to relate the growth rate of $\ntd(\cdot)$ to the growth of the billiard language complexity function $p(\cdot)$, which in turn is related to the count of the number of generalized diagonals $\gd(\cdot)$ by \eqref{eq:ComplexityFormula}, we make the following remarks on the relation between $\ncl(\cdot)$, $\nclprim(\cdot)$ and $\gd(\cdot)$.

\begin{remark}\label{rem:gd_vs_Ncl}
Given a polygonal billiard table $P$, the number of generalized diagonals in $P$ of combinatorial length $k\geq 2$ satisfies the relation:
\begin{equation}\label{eq:gdToNclPrim}
\gd(k) = \frac{p}{2q}\cdot \nclprim(k)
\end{equation}
 Let $P$ be a base tiling in a hyperbolic $(p,q)$-tiling, and $v$ be a vertex of $P$. Then, $\nclprim(k)$ counts the number of primitive geodesic segments of combinatorial length $k$ from $v$ to an arbitrary vertex $u \neq v$. By rotational symmetry, only $\frac{1}{q}$ of these segments pass through the base tile $P$, so we divide by $q$. The multiplication by $p$ comes from the $p$ possible base vertices in polygon $P$. We then divide by $2$ because each generalized diagonal is double counted. 
\end{remark}

\begin{remark}\label{rem:primiveVSnonprimitivegrowth}
Since the number of tiles and vertices in a hyperbolic tiling grow exponentially, one can show via a standard argument that primitive geodesic segments with vertices as endpoints (i.e. generalized diagonals) are of asymptotic density one among all geodesic segments with vertices as endpoints. The heuristic for this argument is that if $\ncl(k)$ grows as $\alpha^k$, then the nonprimitive elements in this count have growth upper bounded by $\sum_{d|k, d < k}\ncl(d) \approx \sum_{d|k, d < k}\alpha^d$, which is asymptotically proportionally $0$ compared with $\alpha^k$. 

Hence, the exponential growth rate of $\nclprim(\cdot)$ and $\ncl(\cdot)$ is the same. In other words, for any $\alpha > 1$, 
$$\lim_{k\rightarrow \infty} \frac{\log_\alpha(\ncl(k))}{k} = 1 \qquad \text{ if and only if } \qquad \lim_{k\rightarrow \infty} \frac{\log_\alpha(\nclprim(k))}{k} = 1.$$  
The same is true for $\liminf$ and $\limsup$.    
\end{remark}

\begin{proposition}[Asymptotic growth of generalized diagonals]\label{prop:tilinggrowthequalsgendiaggrowth}
Consider a hyperbolic $(p,q)$-tiling. 

\begin{enumerate}
    \item If $q$ is even and $\displaystyle\lim_{k\rightarrow \infty} \frac{\log_\alpha (\ntd(k))}{k} = 1$ for some $\alpha > 1$, then $\displaystyle\limsup_{k\rightarrow \infty} \frac{\log_\alpha (\gd(k))}{k} = 1$.
    \item If $q$ is odd and $\displaystyle\lim_{k\rightarrow \infty} \frac{\log_\alpha (\ntd(k))}{k} = 1$ for some $\alpha > 1$, then  $\frac{q-1}{q+1} \leq \displaystyle\limsup_{k\rightarrow \infty} \frac{\log_\alpha (\gd(k))}{k} \leq 1$.
\end{enumerate}

\end{proposition}

\begin{proof} 
We will handle the $q$ even and $q$ odd cases separately. We work with a $\limsup$ to avoid issues about whether $\lim_{k \rightarrow \infty} \frac{\log_\alpha(\gd(k))}{k}$ exists. 

\begin{enumerate}

\item Assume $q$ is even. By Lemma~\ref{lem:gd_tl_relation}, there exists a $C_{p,q}$ such that for every $k \in \NN$,
$$\ncl(k) \leq C_{p,q} \cdot \max_{|n-k+1| \leq q} \{\ntd(n)\}.$$
For each $k$, let $n_k$ be the $n \in [k-q-1,k+q-1]$ for which $\max_{|n-k+1|\leq q} \{\ntd(n)\}$ is achieved.

Assume that $\lim_{k\rightarrow \infty} \frac{\log_\alpha (\ntd(k))}{k} = 1$ for some $\alpha > 0$. Since $|n_k-k|\leq q+1$ and $n_k \rightarrow \infty$ as $k \rightarrow \infty$, we then have that  
$$\limsup_{k\rightarrow \infty} \frac{\log_\alpha (\nclprim(k))}{k}  \leq \limsup_{k\rightarrow \infty} \frac{\log_\alpha (\ncl(k))}{k} \leq  \limsup_{k \rightarrow \infty}  \left(\frac {\log_\alpha (C_{p,q}) + \log_\alpha (\ntd(n_k))}{n_k} \cdot \frac{n_k}{k} \right) = 1.$$

For the lower bound, using Lemma~\ref{lem:gd_tl_relation} again, there exists $K_{p,q} > 0$ such that for any $k \in \NN$, 
$$\ntd(k) \leq K_{p,q} \cdot \max_{|n-k-1| \leq q} \{\ncl(n)\}.$$
Again, let $n_k$ be the $n \in [k-q+1, k+q+1]$ for which 
$\max_{|n-k-1| \leq q}\{\ncl(n)\}$ is achieved. Then,
$$\limsup_{k \rightarrow \infty} \frac{\log_\alpha(\nclprim(k))}{k} \geq \liminf_{k \rightarrow \infty} \frac{\log_\alpha(\ncl(n_k))}{n_k} \geq \liminf_{k \rightarrow \infty} \left(\frac{-\log_\alpha (K_{p,q}) + \log_\alpha (\ntd(k))}{k} \cdot \frac{k}{n_k}\right) = 1, $$
where the first inequality uses Remark~\ref{rem:primiveVSnonprimitivegrowth} and the second uses that $ \ncl(n_k) \geq K^{-1}_{p,q} \ntd(k) $.

Combining the lower and upper bounds gives us $$ \limsup_{k \rightarrow \infty} \frac{\log_\alpha(\nclprim(k))}{k} = 1. $$
Then, using \eqref{eq:gdToNclPrim}, we see
$$
\limsup_{k \rightarrow \infty} \frac{\log_\alpha(\gd(k))}{k} = \limsup_{k\rightarrow \infty}\left(\frac{\log_\alpha(p/2q)}{k}+ \frac{\log_\alpha(\nclprim(k))}{k} \right) = 1.$$

\item Let $q$ be odd. The proof is similar to the $q$ even case. From, Lemma~\ref{lem:gd_tl_relation}, we have that there exists a positive constant $C_{p,q}$ such that for any $k \in \mathbb{N}$, 
$$\ncl(k) \leq  C_{p,q} \cdot k \cdot  \left(\max\left\{\ntd(n) : \frac{q-1}{q+1}(k-q-4) \leq n \leq k+q-1\right\}\right).$$
Let $n_k$ be the maximal $n \in [\frac{q-1}{q+1}(k-q-4) , k+q+1]$ for which the maximum  of $\{\ntd(n)\}$ is achieved. 

By similar arguments as in the upper bound of the $q$ even case, we have that 
$$\limsup_{k \rightarrow \infty} \frac{\log_\alpha(\gd(k))}{k} \leq \limsup_{k \rightarrow \infty}\left(\frac{\log_\alpha (C_{p,q} \cdot k) + \log_\alpha(\ntd(n_k))}{n_k} \cdot \frac{n_k}{k} \right) \leq 1,$$
where we are using that $\frac{n_k}{k}$ is asymptotically between $\frac{q-1}{q+1}$ and $1$. For the lower bound, using Lemma~\ref{lem:gd_tl_relation}, there exists $K_{p,q} > 0$ such that for all $k \in \mathbb{N}$, $$\ntd(k) \leq K_{p,q}\cdot k  \cdot \left(\max \left\{\ncl(n): k-q+1 \leq n \leq \frac{q+1}{q-1} \cdot k+q+4\right\}\right).$$ Letting $n_k$ be the $n \in [k-q+1 , \frac{q+1}{q-1} \cdot k+q+4]$ for which the maximum of $\{\ncl(n)\}$ is achieved, we have by similar arguments as in the even case that $$\limsup_{k\rightarrow \infty} \frac{\log_\alpha(\gd(k))}{k} \geq \liminf_{k \rightarrow \infty} \left(\frac{-\log_\alpha (K_{p,q} \cdot k) + \log_\alpha (\ntd(k))}{k} \cdot \frac{k}{n_k}\right) \geq \frac{q-1}{q+1},$$ since $\frac{k}{n_k}$ is asymptotically between $\frac{q-1}{q+1}$ and $1$. 
\end{enumerate}
\end{proof}

The asymptotic bounds on $\nclprim(n)$ allows us now to deduce the asymptotic growth of $p(n)$:

\ComplexityEvenOdd*

\begin{proof}
    From Theorem~\ref{thm:GutTab}, we have 
    \begin{align}\label{eqn:pn}
   p(n) = c_1 + c_2 n + \sum_{k=3}^n \sum_{j=3}^k \gd(j) 
\end{align}for some non-negative constants $c_1$ and $c_2$. 
Let $\epsilon >0$ be arbitrary. We will first show that 
$$\lim_{n\rightarrow\infty}\frac{\log_\alpha(p(n))}{n} <  1+\epsilon$$
for both the $q$ even and odd cases. 
By Proposition~\ref{prop:tilinggrowthequalsgendiaggrowth}, $\limsup_{n\rightarrow \infty} \frac{\log_\alpha(\gd(n))}{n} < 1+ \epsilon.$ which means that there exists $N$ such that for all $n > N$, $\gd(n) < \alpha^{n(1+\epsilon)}$.
For such an $N$, let $g_{N} = \max_{j \in \{1, \dots, N\}} \gd(j)$. Then for all $n > N$,
\begin{align*}
p(n) &= c_1+ c_2n +\sum_{k=3}^N \sum_{j=3}^k \gd(j) + \sum_{k=N+1}^n\sum_{j=3}^N \gd(j) + \sum_{k=N+1}^n \sum_{j=N+1}^k \gd(j)\\
&\leq c_1+ c_2n +\sum_{k=3}^N \sum_{j=3}^k g_N + \sum_{k=N+1}^n\sum_{j=3}^N g_N + \sum_{k=N+1}^n \sum_{j=N+1}^k \gd(j)\\
&< c_1+c_2n + n^2g_N + \sum_{k=N+1}^n \sum_{j=N+1}^k \gd(j)
\end{align*}
where the first equality comes from rearranging the sum in (\ref{eqn:pn}) and the rest follow from the definition of $g_N$. Noting that for $j > N$, $\gd(j) < \alpha^{j(1+\epsilon)}$, where the latter function is increasing as a function of $j$ since $\alpha > 1$, we see that $$p(n) < c_1 + c_2n + n^2 g_N + n^2 \alpha^{n(1+\epsilon)}.$$ By taking a $\log_\alpha$ on both sides, it follows then that for both $q$ even and $q$ odd, $$\lim_{n \rightarrow \infty} \frac{\log_\alpha p(n)}{n} \leq 1+ \epsilon \qquad \Rightarrow \qquad \lim_{n\rightarrow \infty} \frac{\log_\alpha p(n)}{n}  \leq 1, $$ since $\epsilon > 0$ was arbitrary.

Additionally, we have that $p(n) \geq  \gd(n)$ so that,
$$\lim_{n\rightarrow\infty}\frac{\log_\alpha(p(n))}{n} = \limsup_{n\rightarrow\infty}\frac{\log_\alpha(p(n))}{n} \geq \limsup_{n\rightarrow\infty}\frac{\log_\alpha(\gd(n)))}{n},$$ which, by Proposition~\ref{prop:tilinggrowthequalsgendiaggrowth},  is equal to $1$ when $q$ is even and at least $\frac{q-1}{q+1}$ when $q$ is odd. Since $h_{top}$ satisfies that $$1 = \lim_{n\rightarrow \infty} \frac{\log_{h_{top}}(p(n))}{n} = \lim_{n\rightarrow \infty} \frac{\log_{\alpha}(p(n))}{n} \cdot \frac{1}{\log_\alpha(h_{top})},$$ it follows that $h_{top} =\alpha$ when $q$ is even, and $\alpha^{\frac{q-1}{q+1}} \leq h_{top} \leq 1$ when $q$ is odd. 
\end{proof}

\begin{remark}
    We remark that to do better in the bounds on language complexity and topological entropy in the $q$ odd case using the methods in this paper, we would have to improve the multiplicative constant in the upper bound $\cl(\gamma) \leq \frac{q+1}{q-1}\tl(\gamma)+1$ in Proposition~\ref{prop:q-oddCombboundsTiling}. A geodesic segment the connects two tiles $A$ and $B$ that are share a vertex and are almost opposite can achieve this bound ($\cl(\gamma) = \frac{q+1}{2} + 1$ and $\tl(\gamma) = \frac{q-1}{2}$, see Figure~\ref{fig:q_odd_vertex_sequence}), so the bound is sharp. However, it may be possible to improve this bound asymptotically as $\tl(\gamma) \rightarrow \infty$, which could lead to a better bound on language complexity.
\end{remark}

%% file: word_characterization.tex
\section{Characterization of billiard words for hyperbolic $(p,q)$-tilings when $q$ is even}
\label{sec:billiard_words}

Our goal in this section is to give a set of language rules that will allow us to determine when a bi-infinite word in a regular hyperbolic $(p,q)$-tiling with $q$ even is realized by a geodesic. This is equivalent to determining what billiard words in a regular hyperbolic polygonal billiard table with $p$ sides and interior angles $2 \pi /q$ are realized by a billiard trajectory. 

\subsection{Previous Work of Other Authors}

In \cite{GiannoniUllmo}, Giannoni and Ullmo investigate the set of bi-infinite words that arise from billiard trajectories in various settings, including in compact, regular, hyperbolic polygons that tile the hyperbolic plane. Such polygons with $p$ sides and $2\pi/q$ angles give rise to $(p,q)$-tilings. Their theorem in the $q$ even case, mildly restated to match the language of this paper, was as follows.

\begin{theorem}[Language rules for $q$ even (\cite{GiannoniUllmo})] 
\label{thm:GUeven} 
In a hyperbolic polygonal billiard corresponding to a regular $(p,q)$-tiling with $p$ even, we have the following language rules: 
\begin{enumerate}
    \item Any two-letter subword $kk$ is forbidden for any $k \in \{1,\ldots, p\}$
    \item Any subword of length greater than $\nu_q = \frac{q}{2}$ consisting of alternating letters $k$ and $k+1$ (or $p$ and $1$) is not allowed. 
    \item All $\nu_q$ repetitions of two successive letters $k$ and $k+1$ (or $p$ and $1$) are considered equivalent, regardless of whether the first letter is $k$ or $k+1$ (or $p$ or $1$). 
\end{enumerate}
The last rule introduces new letters $\{|1|, |2|, \ldots, |p|\}$ to the alphabet, where $|k|$ represents the equivalence class of sequences of $\nu_k$ alternating $k$ and $(k+1)$ letters, starting with either letter. With this modified alphabet, every word that does not violate a grammar rule is attained by a unique geodesic. 
\end{theorem}

While Giannoni and Ullmo omitted the proof of this theorem from their published paper at the advisement of the referee, they provided the idea for a proof, which we paraphrase here: after choosing a base tile, they consider the set of geodesics that realize an allowed word. As the allowed word grows longer, the set of geodesics shrinks. In the limit, there is one geodesic that realizes the word. 

When we attempted to independently verify that their proof idea worked, we ran into some difficulties due to some imprecise language around their theorem and proof ideas. We'll briefly comment on some of those difficulties here: 

\begin{itemize}
    \item The authors state in an unpublished version of their paper that successive equivalence class letters $|k||k+1|$ should be interpreted as a length $2\nu_q-1$ sequence as in the following example: when $q=4$, $|2||3|$ should be interpreted as $2323434$. They do not mention this in their published paper, so it is unclear if this convention still holds. 
    \item As stated, it is unclear how to tell whether words including the new letters $\{|1|,|2|, \ldots,|p|\}$ are allowed. For example, when $q=4$, the word $|4|343|4|$ (or $|4||3||4|$ under the contention mentioned in the previous bullet point) is not allowed because one of the equivalent representations is $54543434545$, which violates rule 2.
    \item It is unclear how the authors approach the case where the base tile itself is in the middle of a word $|k|$. To the best of our knowledge, all of the examples and explanations given by the authors, both in their published paper and unpublished draft, toward a proof of this theorem do not address this case. 
\end{itemize}

We note that in \cite{NagarSingh}, Nagar and Singh claim to prove that bi-infinite words that satisfy rules (1) and (2) of Giannoni and Ullmo's Theorem~\ref{thm:GUeven} are always realized by a unique billiard word. Their proof is light on details, and their theorem cannot be true as written since there are finite words that satisfy the rules (1) and (2) and cannot be realized (for example, see Figure~\ref{fig:inadmissible_word_even}). 

\begin{figure}[h]
    \includegraphics[width=10cm]{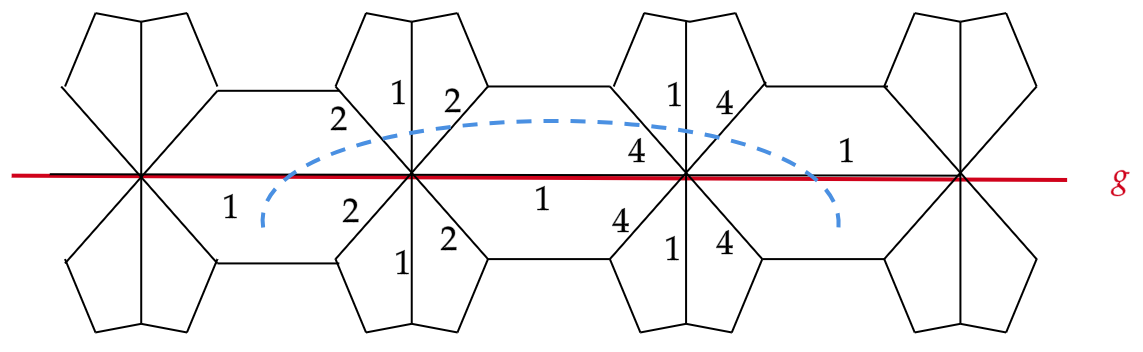}
    \caption{In a $(8,4)$ tiling, the word $12124141$ does not violate rules E1 and E2 of Theorem~\ref{thm:words_even}, but cannot be realized by a billiard word because the corresponding tiling path double crosses the edge geodesic $g$. The word is word equivalent to $12121414$, which violates rule $E2$.}
    \label{fig:inadmissible_word_even}
\end{figure}

The goal of this section is to provide an alternate and self-contained proof of a precisely stated billiard path realization theorem for bi-infinite words when $q$ is even. We will prove this theorem using new proof methods involving insights about minimal tiling paths. We will examine the $q$ odd case in Section~\ref{sec:comparison}. We start by developing preliminary definitions concerning the $q$ even case in the next subsection. 

\subsection{Preliminaries for the $q$ even case}\label{subsec:qevenpreliminaries}

Consider a hyperbolic $(p,q)$-tiling with $q$ even. Given a base tile and a cyclic labeling of that base tile $\{1,2,\ldots,p\}$, the edge labeling then unfolds to a globally defined edge labeling of the $(p,q)$-tiling. 

\begin{restatable}{proposition}{ConsistentLabelling}\label{prop:consistentlabelling}
 In a $(p,q)$-tiling with $q$ even, given any cyclic labeling of the edges of a base tile, there is a labeling of all edges of the tiling that is consistent under reflections over edges.     
\end{restatable}
\begin{proof}
    Since this is a fact from classical hyperbolic geometry, we postpone the proof to Appendix \ref{appdx:hyperbolictilings}.   
\end{proof}

From now on, we will assume that we have a base tile whose edges are cyclically labeled, and a labeling of all edges in the $(p,q)$-tiling that matches on the base tile and is consistent under reflections across edges. Since billiard trajectories in the base polygon lift to geodesics in the tiling, preserving the billiard word, we will mostly work with geodesics in the tiling and comment on the implications for hyperbolic billiards.

Our main theorem of this section is the following, which is a more precisely stated version of Theorem~\ref{thm:GUeven} from \cite{GiannoniUllmo}. Commentary about the components of the theorem will follow the theorem statement. 
\GUeven*

\begin{wrapfigure}[17]{r}{0.3\textwidth}
\centering
\begin{tikzpicture}[scale=0.65]
\draw (-2,0) -- (2,0);
\draw (0,-2) -- (0,2); 
\draw (-1.41,1.41) -- (1.41,-1.41);
\draw (-1.41, -1.41) -- (1.41, 1.41);
\draw[thick, ->] (-1.3, -.1) -- (1.3, .5); 
\draw[thick, ->] (-1.3, -.5) -- (1.3, .1);
\node at (2.3, 0) {$k$};
\node at (0, 2.3) {$k$};
\node at (-2.3, 0) {$k$};
\node at (0, -2.3) {$k$};
\node at (1.6, 1.6) {$k+1$};
\node at (1.6, -1.6) {$k+1$};
\node at (-1.6, 1.6) {$k+1$};
\node at (-1.6, -1.6) {$k+1$};
\end{tikzpicture}
\caption{A schematic of two unfolded billiard trajectories passing near a vertex in a tiling where $q=8$. We notice that these trajectories cannot intersect more than $\frac{8}{2} = 4$ edges around the vertex in a row.} 
\label{fig:q_even_vertex}
\end{wrapfigure}
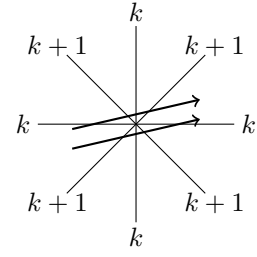

The two grammar rules are natural. Rule E1 states that a geodesic in the tiling cannot hit the same side twice in a row. When $q$ is even, recall that edge geodesics (see Definition~\ref{def:edge_geodesic}) are bi-infinite geodesics containing an edge and are entirely comprised of edges of the tiling. Rule E2 holds because a geodesic cannot double cross an edge geodesic, so a billiard word cannot contain subwords of length $> \frac{q}{2}$ consisting of alternating $k$ and $k+1$ letters (or 1 and $p$). See Figure~\ref{fig:q_even_vertex} for an illustration of this latter case. To make sense of the equivalence relation on allowed words that gives us unique realization by a billiard trajectory, we need a few definitions.

\begin{definition}[Associated word of tiling path]\label{def:associatedWordToTilingPath}
    Given a finite tiling path $\mathcal{A} = (A_0, \dots, A_n )$, the \textbf{associated word} of the path $\mathcal{A}$ is the word $w_{\mathcal{A}} = a_1 a_2 \dots a_n $ where $a_i$ is the label of the edge between $A_{i-1}$ and $A_{i}$. The associated word for infinite tiling paths is likewise defined. Assuming that all tiling paths begin at a given base tile $A_0$, there is a bijection between (finite or infinite) words $a_1a_2a_3\ldots$ and (finite or infinite) tiling paths $(A_0, A_1, A_2, \ldots)$ with that associated word.
\end{definition}

Note that Definition~\ref{def:associatedWordToTilingPath} also applies for the case when $q$ is odd, and we will use it in Section~\ref{sec:comparison} when we compare various growth rates. 

In order to obtain a tiling path from a finite or infinite word, we must choose a base tile $A_0$. When words are bi-infinite, we must choose a center of the word which corresponds to a chosen tile base tile.

\begin{definition}[Centered word]\label{def:BiInfiniteCenter}
Given a bi-infinite word on $p$ letters, $w = \ldots w_{-2} w_{-1} w_0 w_1 w_2 \ldots$,  we \textbf{center the word at tile $A_0$} of a hyperbolic $(p,q)$-tiling by considering the labels $w_{-1}$ and $w_0$ to correspond to edges of tile $A_0$. 
\end{definition}

With these definitions, we can move back and forth between (centered) words and tiling paths based at a tile $A_0$. We will mostly work with tiling paths as they are easier to visualize and work with. Next, as mentioned in Theorem~\ref{thm:words_even}, we will define equivalence classes of words and tiling paths. For this, we first give a name to subpaths that pass through tiles adjacent to a given vertex or equivalently sub-words that consist of alternating letters.

\begin{definition}[Vertex traversal] Let $v$ be a vertex in a hyperbolic $(p,q)$-tiling. A \textbf{vertex traversal} of $v$ is defined to be a tiling subpath that traverses up to $\frac{q}{2}+1$ tiles adjacent to $v$ in a row. The corresponding subword is a vertex traversal if it is of length at most $v_q=\frac{q}{2}$ consisting of alternating letters $k$ and $k+1$ (or $1$ and $p$) or $k+1$ and $k$ (or $p$ and $1$).
\end{definition}

We give a special name to vertex traversal words (equivalently, paths) that contain precisely $\frac{q}{2}$ alternating letters, since these represent the extreme case of what is allowed by rule $E2$ in Theorem~\ref{thm:words_even}.

\begin{definition}[Vertex sequence] A tiling subpath is a \textbf{vertex sequence} if it is a vertex traversal with $\frac{q}{2}+1$ tiles. Correspondingly, a subword is a vertex sequence if it is a vertex traversal of length $v_q = \frac{q}{2}$. Two tiling path vertex sequences are \textbf{vertex sequence equivalent} if they have the same starting and ending tiles but vertex traverse $\frac{q}{2}+1$ successive tiles in opposite directions around a vertex $v$. The corresponding vertex sequence subwords are vertex sequence equivalent when they interchange the order of $k$ and $k+1$ (e.g. when $q=8$, the words $2121$ and  $1212$ are vertex sequence equivalent). 
\end{definition}

Changing a word/tiling path by replacing a vertex sequence with its vertex sequence equivalent gives rise to an equivalent word/tiling path. We make this notion precise in the following definition.

\begin{definition}[Word equivalence of words/tiling paths] \label{def:equivalence} We say that two finite tiling paths or words are \textbf{word equivalent} if they differ by finitely many exchanges of vertex sequences with their vertex sequence equivalents. Two infinite tiling paths or infinite or (centered) bi-infinite words are \textbf{word equivalent} if there there is a sequence of exchanges of vertex sequence equivalence moves, such that any finite segment of the path/word is equivalent to the corresponding segment of the other path/word after finitely many moves.

\end{definition}

Similar to how bi-infinite words can be centered, we can also center an equivalence class of bi-infinite words:

\begin{definition}[Centered word class]\label{defn:centeredwordclass}
Given a bi-infinite word class $\omega$ we \textbf{center the word equivalence class at tile $A_0$} by picking a representative word $w  =\ldots w_{-2} w_{-1} w_0 w_1 w_2 \ldots \in \omega$ and centering it on tile $A_0$. This then centers the other words in the class by considering how vertex sequence equivalences change the tiling path, even though their centers might be different than tile $A_0$. 
\end{definition}

Next, following Theorem~\ref{thm:words_even}, we give terminology to words that do not violate grammar rules E1 and E2.

\begin{definition}[Admissible] A (finite, infinite or bi-infinite) word $w$ will be called \textbf{admissible} if it does not violate grammar rules E1 and E2 from Theorem~\ref{thm:words_even}. An \textbf{admissible word class} $[w]$ is an equivalence class of (finite, infinite or bi-infinite) words where each word in the class is admissible. A tiling path or equivalence class of tiling paths is admissible if the corresponding words are. 
\end{definition}

\begin{remark} 
\label{rem:even_observations}
Here are some useful observations about tiling paths corresponding to admissible words in a $(p,q)$-tiling with $q$ even. Let $\mathcal{A}$ be a tiling path and $w_\mathcal{A}$ be its associated word. 
    \begin{itemize}
        \item Words that violate rule E1 and contain a two-letter subword $kk$ correspond to tiling paths that contain a \textbf{backtracking step} $(A,B,A)$ and are thus not minimal (as per Definition~\ref{def:tiling_path}).
        \item Words that violate rule E2 and contain a subword of length $\frac{q}{2}+i$ for $i \geq 1$ of alternating $k$ and $k+1$ letters correspond to a tiling path that goes more than halfway around a vertex. Such paths can be shortened by going the other way around the vertex and are thus not minimal. 
        \item If two finite words are equivalent, then the corresponding tiling paths have the same beginning tile (by assumption) and ending tile. This is because a vertex sequence move is a local move that does not change any tiles before or after the vertex sequence (see Figure~\ref{fig:q_even_vertex}). 
    \end{itemize}
\end{remark}

\begin{remark} We also make a few observations about admissible word classes, that will be useful in relating words and tiling paths. 
\begin{itemize}
    \item It is possible for an admissible word to be equivalent to a non-admissible word. For example, consider a $(4,8)$-tiling. Then, the word $12123131$ is admissible and equivalent to the word $12121313$, which is not admissible. An admissible word class must not contain any non-admissible words. For example, $12123131$ cannot be a part of an admissible word class because it is equivalent to the non-admissible word $12121313$. The corresponding tiling path will double cross an edge geodesic. 
    \item  A consequence of Theorem~\ref{thm:words_even} is that in any admissible word class of bi-infinite words, exactly one element of the equivalence class is the word of a billiard trajectory. We note that in an admissible word class of words that are not bi-infinite, this need not be the case anymore. For example, in the hyperbolic $(4,8)$-tiling, the words $1212$ and $2121$ are in the same admissible word class, but both come from billiard trajectories, as visualized in Figure~\ref{fig:q_even_vertex}.
    \end{itemize}
\end{remark}

Finally, it will be helpful for us to be able to associate a word to a geodesic in the hyperbolic plane.

\begin{definition}[Word associated to a geodesic/geodesic segment]\label{def:AssociatedWordOfGeodesic}
Given a geodesic or geodesic segment $\gamma$ in a $(p,q)$-tiling, let $\calA_\gamma$ be its associated tiling path defined in Definition~\ref{def:TilingPathGeodesicSegment}. The \textbf{word associated to} $\gamma$, denoted $w_\gamma$, is then defined to be the word associated to the tiling path $\calA_\gamma$ using Definition~\ref{def:associatedWordToTilingPath}.
\end{definition}

We end this section with a few more definitions that we will need in the subsequent sections. 

\begin{definition}[Sectors] \label{def:sector} Given a distinguished vertex $v$ of a hyperbolic $(p,q)$-tiling with $q$ even, the edge geodesics passing through $v$ divide the hyperbolic plane into $q$ \textbf{sectors}. The interior of every tile is completely contained within such a sector.  Let $v$ be a vertex that defines $q$ sectors in the hyperbolic plane. A bi-infinite tiling path is said to begin (end) in one of the sectors if it has an initial (final) segment that remains completely in that sector. 
\end{definition}

\begin{definition}[Fellow traveling]\label{def:fellowtravel} Given a hyperbolic $(p,q)$-tiling with $q$ even and an edge geodesic $g$, a tiling path $\calA$ is said to fellow travel $g$ if $\calA$ does not back-track (i.e. does not contain a subpath of the form $(A, B, A)$) and, every tile in $\calA$ exists in the same half-space of $g$ and has at least one vertex on $g$. 

\end{definition}

We note that the idea of a tiling path fellow traveling an edge geodesic is that it vertex traverses successive vertices along the geodesic. Figure~\ref{fig:fellow_travel_conversion} includes a pictorial depiction of fellow traveling.

\subsection{A characterization of finite billiard words when $q$ is even}\label{subsec:qevenfinitecharacterization}

The goal of this subsection is to, after fixing a base tile, give a bijection between tiles and admissible word classes. Towards this end, we first state two propositions that relate minimal tiling paths and admissible words. We will give proofs of these propositions later in this subsection.

\begin{proposition}
\label{prop:minimal_tiling_paths_word_equivalent}
    In a hyperbolic $(p,q)$-tiling with $q$ even, minimal tiling paths with the same starting and ending tiles are word equivalent. 
\end{proposition}

\begin{proposition}\label{prop:minpathimpliesadmissibleword}
In a hyperbolic $(p,q)$-tiling with $q$ even, a finite tiling path $\calA$ is minimal if and only if the word-equivalence class of $w_\calA$ is admissible. 
\end{proposition}

We can now state and prove the main theorem of this subsection.
 
\GUfinite*

\begin{proof}[Proof of Theorem~\ref{thm:admissible-tiles}] Given a tile of distance $n$ away from a fixed base tile $A_0$, we can choose a minimal tiling path $\calA = (A_0, A_1, \ldots, A_n)$ from the base tile $A_0$ to the tile $A_n$. By Proposition~\ref{prop:minpathimpliesadmissibleword}, $w_\calA$ is then an admissible word, and by Remark \ref{rem:even_observations}, every word in the word class of $w_\calA$ ends at tile $A_n$. Thus, we have a map from tiles of distance $n$ away to admissible word classes of length $n$. By Proposition~\ref{prop:minimal_tiling_paths_word_equivalent}, this map is well defined since any other minimal tiling path from $A_0$ to $A_n$ would result in an equivalent word. 

On the other hand, given an admissible word class of length $n$, we can choose any word in the admissible word class and follow the word from the base tile $A_0$ to get a tiling path $\calA=(A_0, A_1, \ldots, A_n)$ to some end tile $A_n$. By Proposition~\ref{prop:minpathimpliesadmissibleword}, the path $\calA$ is a minimal tiling path so that tile $A_n$ is distance $n$ away from the base tile $A_0$.

By Remark \ref{rem:even_observations}, choosing any other word in the admissible word class would have resulted in the same end tile $A_n$. Thus, this map from admissible word classes of length $n$ to tiles of distance $n$ away from the base tile is well-defined. 

We have seen that the forward and backward maps are well-defined. By construction, they are inverses of one another, and thus we have a bijection between our two sets. 
\end{proof}

To prove Propositions \ref{prop:minimal_tiling_paths_word_equivalent} and \ref{prop:minpathimpliesadmissibleword}, we need the following lemmas. The first is a hyperbolic geometry lemma concerning two cases where edge geodesics cannot intersect, which are illustrated in Figure~\ref{fig:HypGeomLemmas}. The proof is detailed in Appendix \ref{appdx:hyperbolictilings}.

\begin{restatable}{lemma}{EdgeGeodDontIntersectTwoCases}\label{lem:EdgeGeodDontIntersectTwoCases}
    Given a hyperbolic $(p,q)$-tiling where $q$ is even, let $A$ be a tile and $g$ be an edge geodesic in the tiling that shares at least one vertex with $A$. Let $e$ be an edge of $A$ that does not have any vertices on $g$. Then the edge geodesic  extending $e$ does not intersect $g$.
\end{restatable}

\begin{figure}
    \centering
    \begin{subfigure}{0.35\textwidth}
    \centering
    \includegraphics[scale=0.8]{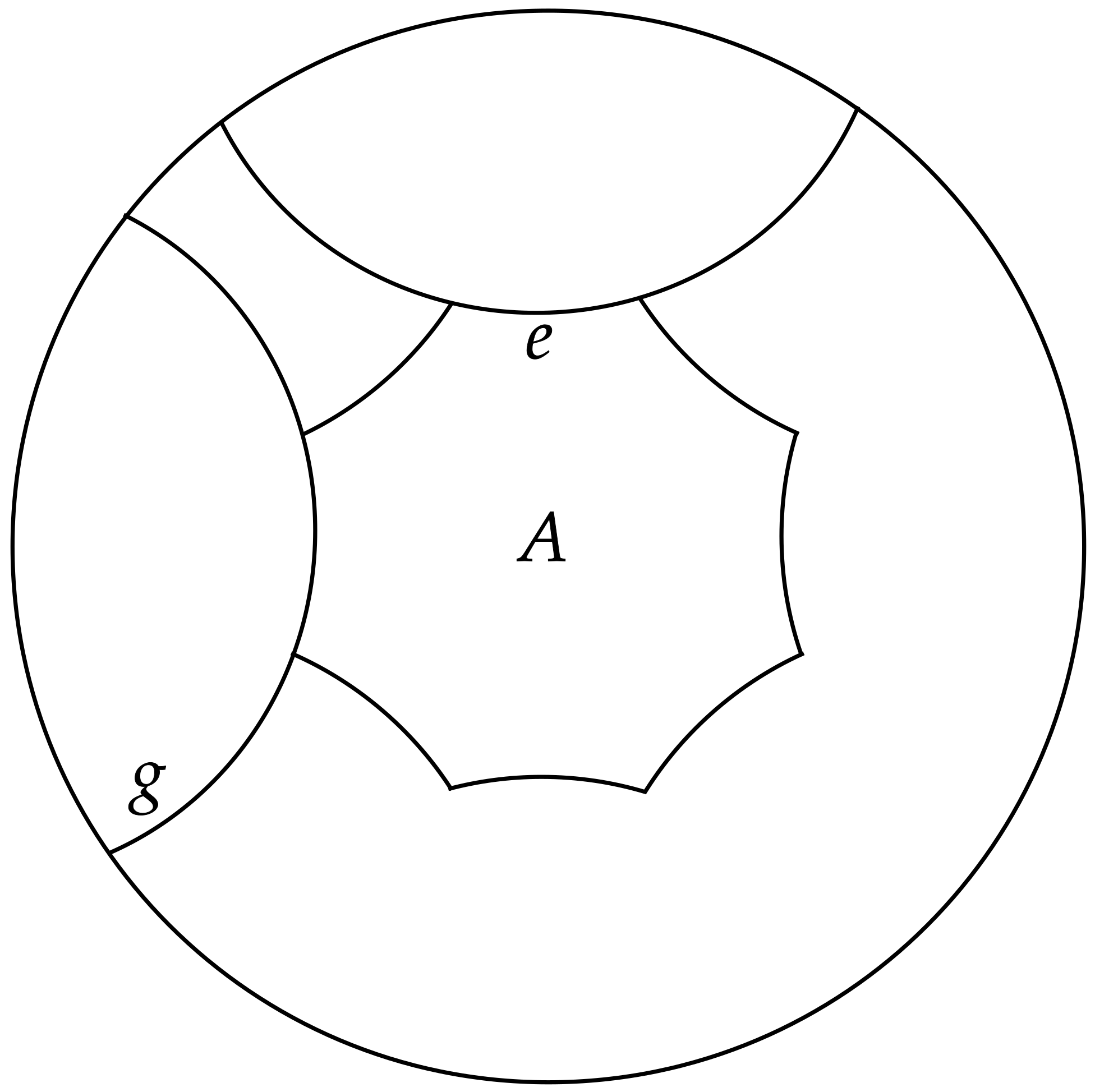}
    
    \caption{Edge geodesic $g$ shares an edge with $A$.}
    \label{fig:EdgeGeodesicNonIntersection1}
    \end{subfigure}
    \hspace{1cm}
    \begin{subfigure}{0.35\textwidth}
    \centering
    \includegraphics[scale=0.8]{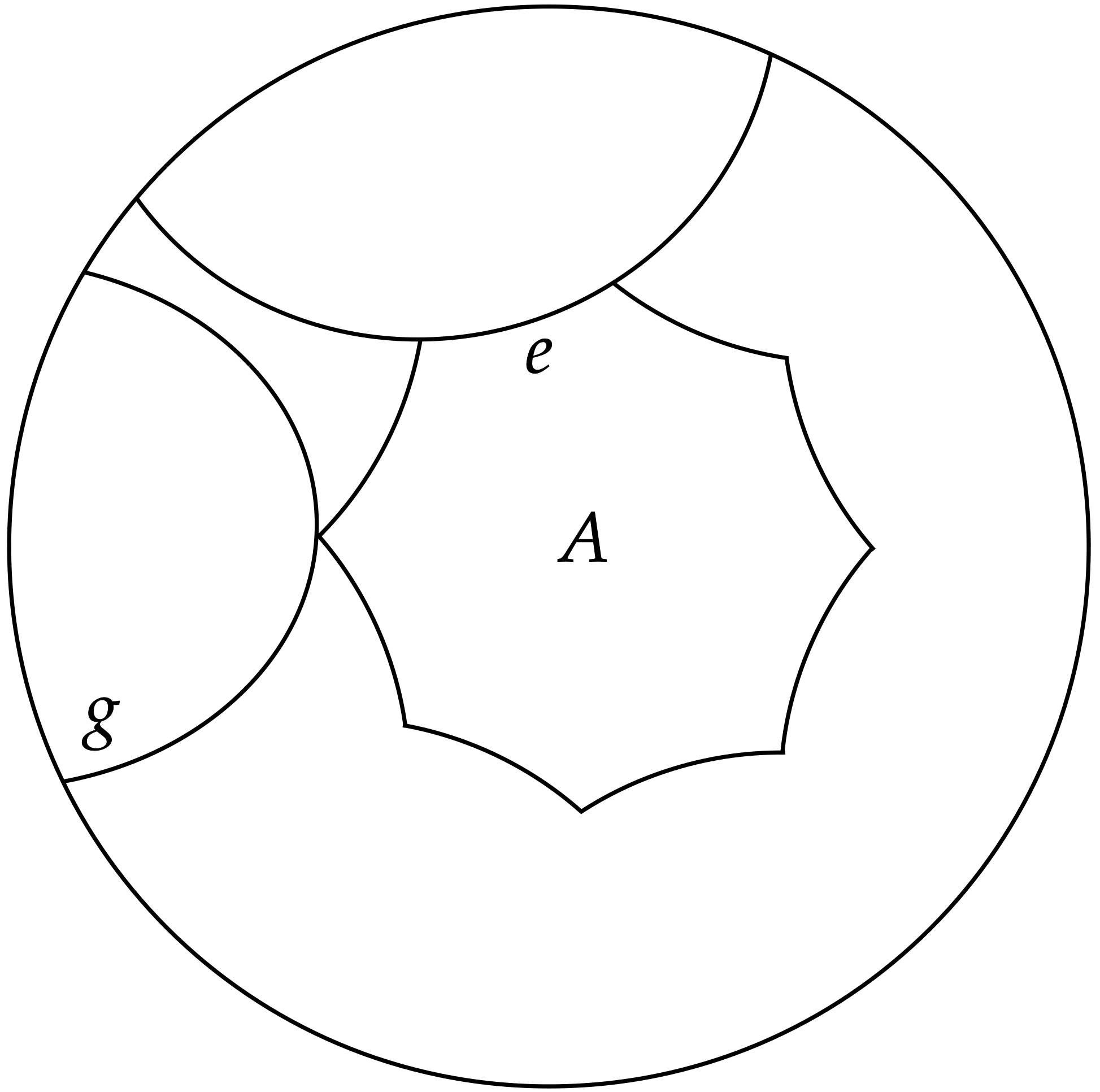}
    \caption{Edge geodesic $g$ shares just a vertex with $A$.}
    \label{fig:EdgeGeodesicNonIntersection2}
    \end{subfigure}
    \caption{A schematic illustrating Lemma~\ref{lem:EdgeGeodDontIntersectTwoCases}. In both cases, the edge geodesic $g$ does not intersect the geodesic extension of edge $e$.}
    \label{fig:HypGeomLemmas}
\end{figure}

This next lemma gives a restriction on the behavior of minimal tiling paths with the same beginning and ending tiles. 

\begin{lemma} 
\label{lem:minimal_adjacent_exits}
Let $\mathcal{A} = (A_0, A_1, \ldots, A_n)$ and $\mathcal{B} = (B_0, B_1, \ldots, B_n)$ be minimal tiling paths with the same beginning and ending tiles $A_0 = B_0$ and $A_n = B_n$. Let $e$ and $f$ be edges between $A_0$ and $A_1$ and $B_0$ and $B_1$ respectively. Then $e$ and $f$ are adjacent on $A_0 = B_0$. In other words, the paths $\mathcal{A}$ and $\mathcal{B}$ exit the starting tile through adjacent edges of $A_0 = B_0$. 
\end{lemma}
\begin{proof}
Assume for the sake of contradiction that $e$ and $f$ are non-adjacent edges on $A_0=B_0$. Consider edge geodesics $g_e$ and $g_f$. Note that $\mathcal{A}$ crosses $g_e$ and $\mathcal{B}$ crosses $g_f$ at edges $e$ and $f$ respectively. 

From Proposition~\ref{prop:minimal_tiling_path_characterization}, $\calA$ crosses $g_e$ precisely once. So, the initial tile $A_0$ and the final tile $A_n$ are on opposite half-spaces of $g_e$. By Lemma~\ref{lem:EdgeGeodDontIntersectTwoCases} $g_e$ and $g_f$ are parallel. So, the half-space of $g_e$ that contains $A_0 = B_0$ also contains the geodesic $g_f$. See Figure~\ref{fig:HypGeomLemmas}. This  implies that the half-space of $g_f$ that contains $B_n=A_n$ is a subset of the half-space of $g_e$ that contains $A_0=B_0$, a contradiction to $A_0$ and $A_n$ being on opposite half-spaces of $g_e$. Hence, edges $e$ and $f$ are adjacent in $A_0$.
\end{proof}

This next lemma restricts the behavior of minimal tiling paths that traverse two tiles adjacent to an edge geodesic. It uses the concept of a fellow travelling tiling path, defined in Definition~\ref{def:fellowtravel}.

\begin{lemma}[Necessarily fellow traveling]\label{lem:necessarilyfellowtraveling}
Let $\mathcal{A} = (A_0, \dots, A_n)$ be a tiling path so that there exists edges $ e \in A_0$ and $f \in A_n$ that have the same geodesic extension $g$. If $A_0$ and $A_n$ are on the same halfspace of $g$, then the tiling path $\mathcal{A}$ is minimal if and only if it fellow travels $g$. If $A_0$ and $A_n$ are on opposite halfspaces, then $\mathcal{A}$ is minimal if and only if it fellow travels $g$ on one side in the direction of $A_n$, crosses $g$ and then fellow travels $g$ on the other side in the same direction. 
\end{lemma}

\begin{proof}
We first consider the case when $A_0$ and $A_n$ are on the same side of $g$. Since $\mathcal{A}$ is minimal, all tiles of $\mathcal{A}$ must be on the same side of $g$ or else the tiling path $\calA$ would double cross $g$, contradicting Proposition~\ref{prop:minimal_tiling_path_characterization}. Assume now for the sake of contradiction that $\mathcal{A} = (A_0, \dots, A_n)$ does not fellow travel $g$. This implies that there exists a minimal $0 < j < n$ such that the tile $A_j$ on the path $\calA$ does not have any vertex on $g$. By minimality of $j$, the tile $A_{j-1}$ has either exactly one or two vertices on $g$.

In either case, let $e$ be the edge shared by $A_{j-1}$ and $A_j$. This edge cannot have a vertex on $g$ or else $A_j$ would have a vertex on $g$. By Lemma~\ref{lem:EdgeGeodDontIntersectTwoCases} the corresponding edge geodesic $g_e$ does not intersect $g$. See Figure~\ref{fig:fellowtraveling}.

Furthermore, by definition of $g_e$, $A_j$ is in the opposite halfspace of $g_e$ as the geodesic $g$ and all tiles that share a vertex with $g$. In particular, $A_{j-1}$ and $A_n$ are in the opposite halfspace as $A_j$. Thus $\calA$ is minimal but double-crosses $g_e$, which contradicts Proposition~\ref{prop:minimal_tiling_path_characterization}. Therefore, the minimal path $\calA$ must necessarily fellow travel $g$. Conversely, since a minimal tiling path between $A_0$ and $A_n$ must exist, it follows that the fellow traveling path between $A_0$ and $A_n$ is minimal. 

Now we consider the case when $A_0$ and $A_n$ are on opposite sides of $g$. Suppose first that $\mathcal{A}$ is minimal. It must cross $g$ exactly once at some subpath $(A_j,A_{j+1})$ for $0 \leq j \leq n-1$. Considering the subpaths from $A_0$ to $A_j$, as well as from $A_{j+1}$ to $A_n$, the first case applies and these subpaths must fellow travel. Conversely, we suppose that $\mathcal{A}$ fellow travels on one side of $g$, crosses $g$, and then fellow travels on the other side of $g$. Since there must exist some minimal tiling path between $A_0$ and $A_n$, at least one of these paths must be minimal. Since all such paths differ by vertex sequences, they all have the same length and thus are all minimal. 
\end{proof}

\begin{figure}[h!]
\centering
\includegraphics[scale=1.2]{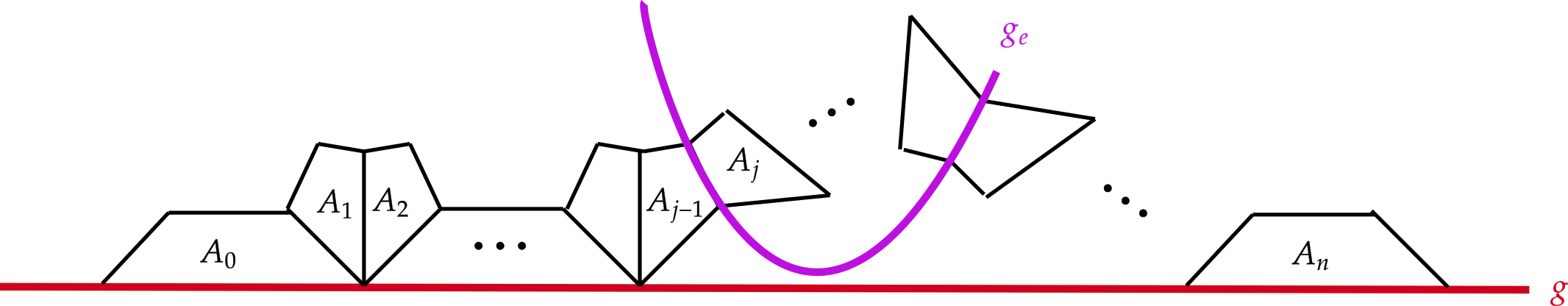}
\caption{If the tiling path $\calA$ does not fellow travel the geodesic $g$, then it must double cross the edge geodesic $g_e$}
\label{fig:fellowtraveling}
    
\end{figure}

Having proven these lemmas, we can now prove Proposition~\ref{prop:minimal_tiling_paths_word_equivalent}, which states that minimal tiling paths with the same beginning and ending tiles are word equivalent.

\begin{proof}[Proof of Proposition~\ref{prop:minimal_tiling_paths_word_equivalent}]

We suppose for contradiction that there exist minimal tiling paths with the same initial and ending tiles that are not word equivalent. Then, let $n$ denote the least length of such paths, and let $\mathcal{A} = (A_0, A_1, \ldots, A_n)$ and $\mathcal{B} = (B_0, B_1, \ldots, B_n)$ be two minimal tiling paths starting and ending at $A_0=B_0$ and $A_n=B_n$ respectively, that are not word equivalent (see Figure~\ref{fig:minimal_adjacent_exits}). 

\begin{figure}[h!]
    \centering
    \includegraphics[width=12cm]{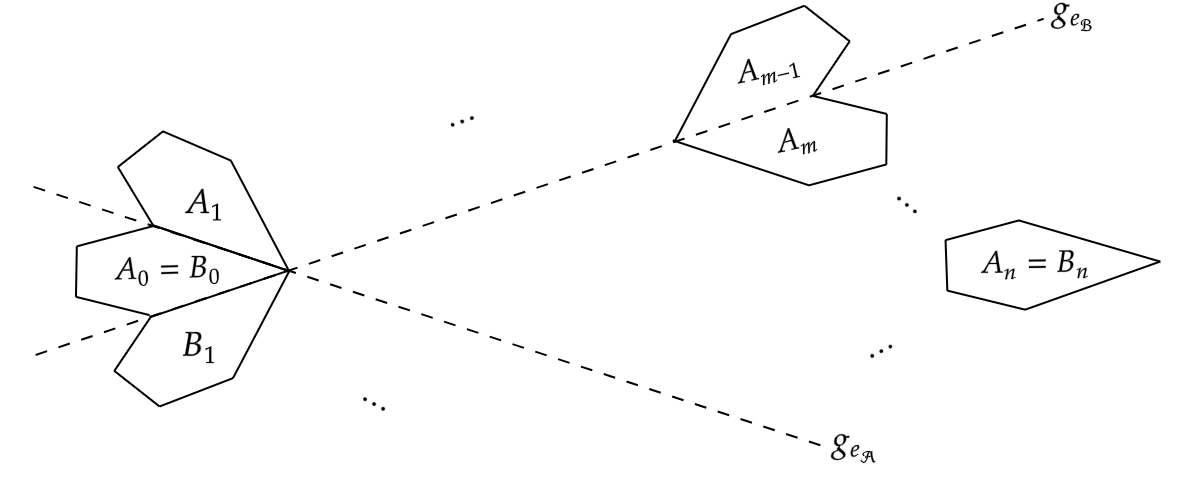}
    \caption{Two minimal tiling paths $\mathcal{A}$ and $\mathcal{B}$ with the same starting and ending tile.}
    \label{fig:minimal_adjacent_exits}
\end{figure}

Let $e_\mathcal{A}$ and $e_{\mathcal{B}}$ be the edges between $A_0$ and $A_1$, and $B_0$ and $B_1$ respectively. These edges must be adjacent edges of $A_0$ by Lemma~\ref{lem:minimal_adjacent_exits}.  We have that $\mathcal{B}$ does not double cross the edge geodesic $g_{e_\mathcal{B}}$ by Proposition~\ref{prop:minimal_tiling_path_characterization} and therefore $A_n$ must be on the opposite side of $g_{e_\mathcal{B}}$ than $A_0$. Let $m \leq n$ be such that $A_m$ is the first tile of $\mathcal{A}$ that is on the opposite side of $g_{e_{\mathcal{B}}}$ from $A_0$. By Lemma~\ref{lem:necessarilyfellowtraveling}, $\mathcal{A}$ must fellow travel $g_{e_\mathcal{B}}$ before crossing the geodesic at $(A_{m-1}, A_m)$. 

\begin{figure}[h!]
    \centering
    \includegraphics[width=15cm]{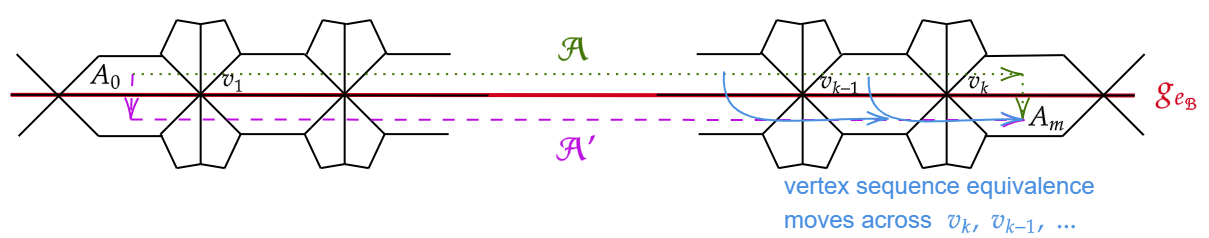}
    \caption{A schematic depiction of how the path $\mathcal{A}$ that fellow travels an edge geodesic $g_{e_\mathcal{B}}$ and then crosses it is vertex sequence equivalent to $\mathcal{A}'$ that first crosses $g_{e_\mathcal{B}}$ and then fellow travels it, via vertex sequence equivalence moves across vertices $v_k, v_{k-1}, \ldots, v_1$ along $g_{e_\mathcal{B}}$.}
    \label{fig:fellow_travel_conversion}
\end{figure}

By performing repeated vertex sequence equivalence moves across the vertices of $g_{e_{\mathcal{B}}}$ as depicted in Figure~\ref{fig:fellow_travel_conversion}, the path $(A_0, A_1', \ldots, A_{m-1}', A_m)$ that crosses $g_{e_{\mathcal{B}}}$ at $(A_0, A_1')$ and then fellow travels $g_{e_{\mathcal{B}}}$ up until $A_m$ is equivalent to $(A_0, A_1, \ldots, A_{m-1},A_m)$. If we let $\mathcal{A}' = (A_0, A_1', \ldots, A_{m-1}', A_m, \ldots, A_n)$ we must have that $\mathcal{A}'$ is also a minimal path between $A_0$ and $A_n$. But $A_1' = B_1$, so $\mathcal{A}'$ and $\mathcal{B}$ without the initial tile $A_0 = B_0$ must be length $n-1$ minimal tiling paths with the same initial and ending tile but that are not word equivalent (since if they were word equivalent, then $\mathcal{A}$ and $\mathcal{A}'$ and $\mathcal{B}$ would be word equivalent.) This contradicts that $\mathcal{A}$ and $\mathcal{B}$ were least length examples of this phenomenon, and so we are done. 
\end{proof}

Finally we prove Proposition~\ref{prop:minpathimpliesadmissibleword}, which states that a tiling path $\mathcal{A}$ is minimal if and only if its word class $[w_{\mathcal{A}}]$ is admissible. 

\begin{proof}[Proof of Proposition~\ref{prop:minpathimpliesadmissibleword}]

First, suppose that $\mathcal{A}$ is a minimal tiling path. By Remark \ref{rem:even_observations}, the word $w_\mathcal{A}$ must be admissible and every word $w$ in the word class $[w_{\mathcal{A}}]$ is of the same length as $w_{\mathcal{A}}$. Since each such $w$ defines a tiling path of the same length and with the same starting and ending tiles as $\mathcal{A}$, it follows that $w$ is also admissible. 

Now we suppose that a tiling path $\mathcal{A}$ is not minimal. By Proposition~\ref{prop:minimal_tiling_path_characterization}, $\mathcal{A}$ must double cross some edge geodesic. Without loss of generality, we can replace $\mathcal{A}$ with its shortest length tiling subpath $(A_0, A_1, \ldots, A_m)$ that double crosses some edge geodesic $g$, and therefore does not double cross any other edge geodesic. It follows then $\mathcal{A}$ crosses $g$ at $(A_0,A_1)$, stays on one side of $g$, and then recrosses $g$ at $(A_{m-1}, A_m)$. Furthermore, every proper subpath of $\mathcal{A}$ must be minimal since it does not double-cross any edge geodesics. By Lemma~\ref{lem:necessarilyfellowtraveling}, $(A_1, A_2, \ldots, A_{m-1})$ fellow travels $g$, as depicted in Figure~\ref{fig:fellow_travel_double_cross}.

\begin{figure}[h!]
    \centering
    \includegraphics[width=15cm]{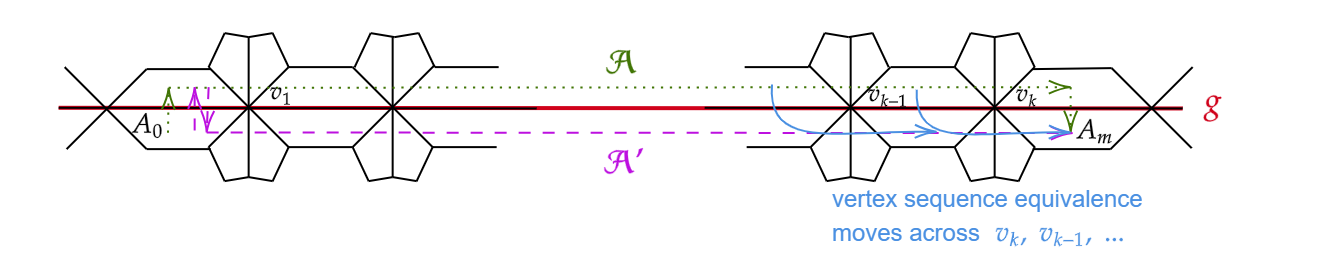}
    \caption{If $\mathcal{A}$ crosses $g$ and then fellow travels before crossing $g$ again, then $\mathcal{A}$ is word equivalent to a tiling path $\mathcal{A}'$ that backtracks and violates condition E1 for admissibility.}
    \label{fig:fellow_travel_double_cross}
\end{figure}

By a similar process as in the proof of Proposition~\ref{prop:minimal_tiling_paths_word_equivalent}, $(A_1, A_2, \ldots, A_{m})$, which fellow travels and then crosses $g$ is word equivalent to a path $(A_1, A_2', \ldots, A_{m-1}', A_m)$ that first crosses and then fellow travels $g$, via a sequence of vertex sequence equivalence moves across the vertices $v_k, v_{k-1}, \ldots, v_1$ of $g$. But then, letting $\mathcal{A}' = (A_0, A_1, A_2'=A_0, A_3', \ldots, A_{m-1}', A_m)$, we have that $\mathcal{A}$ and $\mathcal{A}'$ are word equivalent, but $\mathcal{A}'$ is not admissible since $(A_0, A_1, A_2'=A_0)$ is a backtracking move, with an associated word that violates condition E1. Thus, $\mathcal{A}$ not being minimal implies that $[w_\mathcal{A}]$ is not an admissible word class. \end{proof}

\subsection{A characterization of infinite billiard words when $q$ is even}\label{subsec:qevenOneSidedInfinite}

In this section, we will give a characterization of (one-directional) infinite billiard words in $(p,q)$-tilings where $q$ is even. To do so, we will show that after fixing a base tile, there is a bijection between equivalence classes of geodesic rays starting in the base tile and admissible word classes. 

\begin{definition}[Equivalence classes of geodesic rays] After fixing a base tile $A_0$, we can consider the set of all geodesic rays beginning in $A_0$. We say that two such geodesic rays are equivalent if they are bounded distance apart. Let $\mathcal{G}$ denote the set of such equivalence classes of geodesic rays. The set $\mathcal{G}$ is then in bijection with $S^1 \cong \partial \mathbb{H}^2$, the circle at infinity.  
\end{definition}

We wish to find a bijection between these equivalence classes of geodesic rays starting at a tile $A_0$ and admissible word classes infinite (in the forward direction only) words $[w_1w_2w_3\cdots]$. 

\GUinfinite*

As an immediate corollary of this theorem, we have that following, which we can obtain by folding up a geodesic ray that does not pass through any vertices into a billiard path: 

\begin{corollary}
\label{cor:infinite_billiard_word}
Every admissible word class is achieved by some (not necessarily unique) infinite billiard path. 
\end{corollary}

To prove Theorem~\ref{thm:infinite_bijection}, we build up a sequence of propositions for infinite paths/words, many which are generalizations of of propositions from the previous section when we considered the relationship between finite minimal tiling paths and finite admissible word classes. For the first of these,  recall that an infinite tiling path is said to be minimal if every finite subpath is minimal.

\begin{proposition}
\label{prop:infinite_minimal_admissible}
Consider a hyperbolic $(p,q)$-tiling with $q$ even. Given an infinite tiling path $\mathcal{A} = (A_0, A_1, A_2, \ldots)$, $\mathcal{A}$ is minimal if and only if $[w_\mathcal{A}]$ is admissible. 
\end{proposition}
\begin{proof}
Let $\mathcal{A} = (A_0, A_1, A_2, \ldots)$ be an infinite tiling path and let $[w_\mathcal{A}]$ be the word class of its associated word. 
Suppose that $[w_\mathcal{A}]$ is admissible. It follows that the word class of any finite substring of $w_{\mathcal{A}}$ is admissible. But then, by Proposition~\ref{prop:minpathimpliesadmissibleword}, every finite tiling subpath of $\mathcal{A}$ is minimal, which implies the minimality of $\mathcal{A}$ by definition. 

Now suppose that $[w_\mathcal{A}]$ is not admissible. Then, there exists a word $w$ that is word equivalent to $w_\mathcal{A}$ that is not admissible and therefore contains a subword that violates rule E1 or E2. Consider an initial segment of $v$ that contains that subword. By definition, after finitely many vertex sequence equivalence moves, an initial segment of $w_{\mathcal{A}}$ will match that initial segment of $v$. But then, after finitely many vertex sequence moves $\mathcal{A}$ contains a sub-tiling path that violates E1 or E2 and therefore can be shortened by Remark \ref{rem:even_observations}. Therefore, $\mathcal{A}$ has a finite subpath that is not minimal, which means that $\mathcal{A}$ is not minimal by definition. 
\end{proof}

The following proposition shows that tiling paths with equivalent associated words limit to the same point on the boundary of hyperbolic space.

\begin{proposition}
\label{prop:endpoints}
Consider a $(p,q)$-tiling with $q$ be even. Given a base tile and an admissible word class $[w]$, there exists a unique endpoint in the boundary of the hyperbolic disk that is the limit point of any tiling path corresponding to a word in $[w]$. 
\end{proposition}

\begin{proof}
Let $\mathcal{A} = (A_0, A_{1}, A_2, \ldots)$ be an infinite tiling path corresponding to a word from an admissible word class. We show that the forward tiling path has a unique limit point on the boundary of the hyperbolic disk. 

By Proposition~\ref{prop:infinite_minimal_admissible}, $\calA$ is minimal and therefore by Proposition~\ref{prop:minimal_tiling_path_characterization}, $\calA$ cannot double cross any edge geodesics nor double visit any tiles. For each $i$, let $p_i \in A_i$ be any point. Then, for $C$ the diameter of each tile, every point in each $A_i$ is at most hyperbolic distance $C$ away from $p_i$. We will show that $\{p_0, p_1, p_2, \ldots\}$ has a unique limit point in $\overline{\mathbb{H}}$, the Poincar\'{e} disk with its boundary. 

Since $\overline{\mathbb{H}}$ is compact, there exist limit points to the sequence $\{p_0, p_1, p_2, \ldots\}$. Let $\eta$ be any such limit point. Then, $\eta$ cannot be in the interior $\mathbb{H}$, as a finite radius open ball around $L$ would only intersect finitely many tiles and therefore only contain finitely many $p_i$. Thus, all limits points $\eta$ must be on the boundary $\partial \mathbb{H}$. There can only be one limit point $\eta$ since if there were two limit points $\eta_1 \neq \eta_2$ on $\partial \mathbb{H}$, we could find an edge geodesic that separates $\eta_1$ and $\eta_2$, and the forward tiling path $(A_0, A_1, A_2, \ldots)$ would then have to cross the edge geodesic infinitely often, contradicting Proposition~\ref{prop:minimal_tiling_path_characterization}. Thus, the sequence $\{p_0, p_1, p_2, \ldots\}$ and therefore the forward tiling path $(A_0, A_1, A_2, \ldots)$ limits to $\eta \in \partial \mathbb{H}$. 
\end{proof}

The next proposition is analogous to Proposition~\ref{prop:minimal_tiling_paths_word_equivalent} and has a similar proof. 

\begin{proposition}
\label{prop:minimal_same_endpoint}
Consider a hyperbolic $(p,q)$-tiling with $q$ even. Any two infinite minimal tiling paths $\mathcal{A}$ and $\mathcal{B}$ with the same initial tile and the same endpoint on $S^1 \cong \partial \mathbb{H}^2$ are word equivalent. 
\end{proposition}
\begin{proof}
    Let $\mathcal{A} = (A_0, A_1, A_2, \ldots)$ and $\mathcal{B} = (B_0, B_1, B_2, \ldots)$ be two minimal tiling paths with the same initial tile $A_0 = B_0$ and the same limit point $\eta_\mathcal{A} = \eta_\mathcal{B} \in S^1 \cong \partial \mathbb{H}^2$. Let $A_{k+1} \neq B_{k+1}$ be the first time that these tiling paths traverse different tiles, so that $\mathcal{A}$ and $\mathcal{B}$ have their first $k+1$ initial tiles in common. Note that $k \geq 0$. This scenario is depicted in Figure~\ref{fig:infinite_minimal_equivalent}.

    \begin{figure}[h!]
    \centering
    \includegraphics[width=12cm]{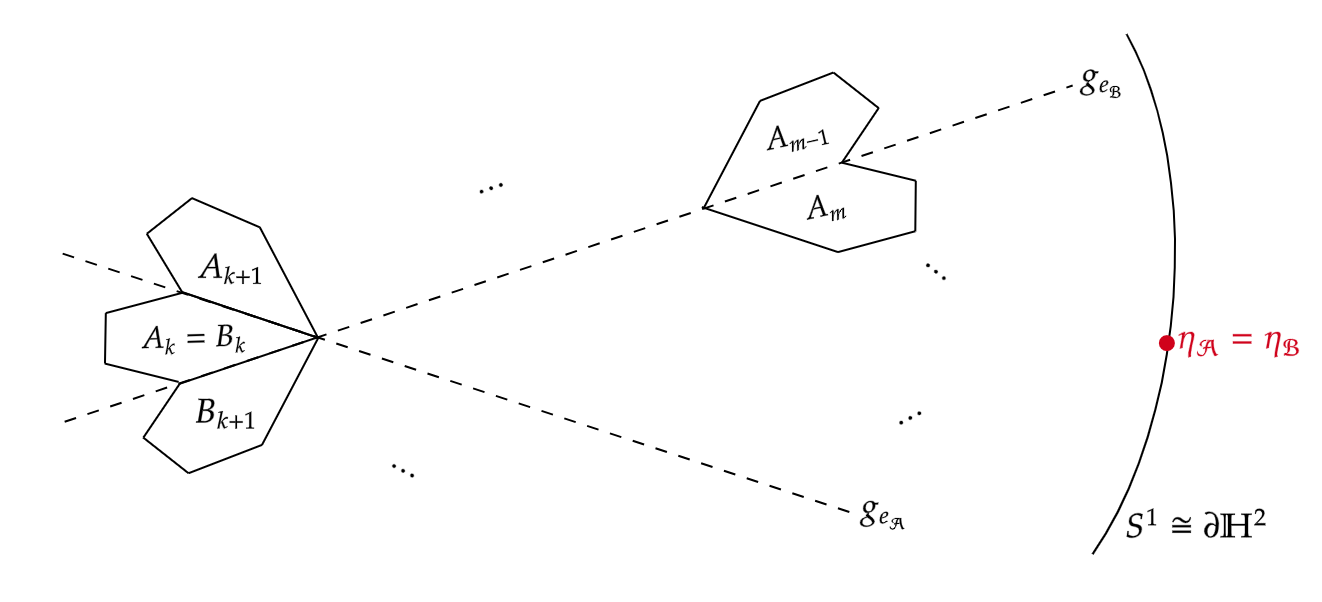}
    \caption{Two minimal tiling paths $\mathcal{A}$ and $\mathcal{B}$ with the same starting tile and endpoint in $\partial \mathbb{H}^2$.}
    \label{fig:infinite_minimal_equivalent}
\end{figure}

By a similar argument as in the proof of Lemma~\ref{lem:minimal_adjacent_exits} but with $\eta_{\mathcal{A}}$ and $\eta_{\mathcal{B}}$ taking the places of $A_n$ and $B_n$ respectively, the edges separating $(A_k, A_{k+1})$ and $(B_k, B_{k+1})$ must be adjacent on $A_k = B_k$. Let $g_{e_\mathcal{A}}$ and $g_{e_\mathcal{B}}$ denote the edge geodesics of these edges respectively. By a similar argument as in the proof of Proposition~\ref{prop:minimal_tiling_paths_word_equivalent}, if $(A_{m-1}, A_m)$ is the first time that $\mathcal{A}$ crosses $g_{e_{\mathcal{B}}}$, then $(A_k, A_{k+1}, \ldots, A_{m-1}, A_m)$ fellow travels $g_{e_\mathcal{B}}$ and then crosses it. This segment is word equivalent to a tiling subpath $(A_k, A_{k+1}', \ldots, A_{m-1}', A_m)$ that crosses $g_{e_\mathcal{B}}$ and then fellow travels it. But then, we have done finitely many vertex sequence equivalence moves to convert $\mathcal{A}$ to $\mathcal{A}'$, where $\mathcal{A}'$ and $\mathcal{B}$ have at least their first $k+2$ tiles in common. This is at least one more than $\mathcal{A}$ and $\mathcal{B}$ had in common. Inductively repeating this process and using the equivalence between base-tiled tiling paths and words, we find that $\mathcal{A}$ and $\mathcal{B}$ are word equivalent. 
\end{proof}

With these propositions in hand, we can now prove our main theorem of this subsection. 

\begin{proof}[Proof of Theorem~\ref{thm:infinite_bijection}]
Let $\mathcal{G}$ be the set of equivalence classes of geodesic rays beginning in the interior of a fixed based tile $A_0$, and let $\mathcal{W}$ be the set of admissible word classes of infinite words. We will first create a map from $\Phi:\mathcal{G} \rightarrow \mathcal{W}$ as follows.

Let $[\alpha] \in \mathcal{G}$ and $\eta \in \partial \HH^2$ be its corresponding endpoint in the boundary. After picking a representative geodesic ray $\alpha$, we have that $w_\alpha$ is the word associated to $\alpha$ (see Definition~\ref{def:AssociatedWordOfGeodesic}), and we define $\Phi([\alpha]) = [w_\alpha]$. 

We first show this map is well defined. Let $\alpha, \beta \in [\alpha]$ be geodesic rays starting in $A_0$ with the same endpoint $\eta \in \partial \mathbb{H}^2$. Since $\alpha$, $\beta$ are geodesics rays, the corresponding tiling paths $\calA_\alpha$ and $\calA_\beta$ are minimal tiling paths with the same base tile and endpoint, and so $w_\alpha$ and $w_\beta$ are admissible by Proposition~\ref{prop:infinite_minimal_admissible}. Furthermore, by Proposition~\ref{prop:minimal_same_endpoint}, $\mathcal{A}_\alpha$ and $\mathcal{A}_\beta$, and therefore $w_{\alpha}$ and $w_{\beta}$, are word equivalent.

We now argue $\Phi$ is bijective, starting with injectivity. Let $[\alpha], [\beta] \in \mathcal{G}$ be distinct and let $\eta_1, \eta_2 \in \partial\HH^2$ be their distinct endpoints. Let $g$ be an edge geodesic whose endpoints separate $\eta_1$ and $\eta_2$. See Figure~\ref{fig:infinitebijection}. This means that exactly one of the tiling two paths $\calA_\alpha := (A_0, A_1, \dots, )$ or $\calA_\beta := (A_0, B_1, \dots)$ crosses $g$. Without loss of generality, suppose that $\calA_\beta$ crosses $g$ between edges $B_{n-1}$ and $B_n$.  

\begin{figure}[h!]
\centering
\includegraphics[width=10cm]{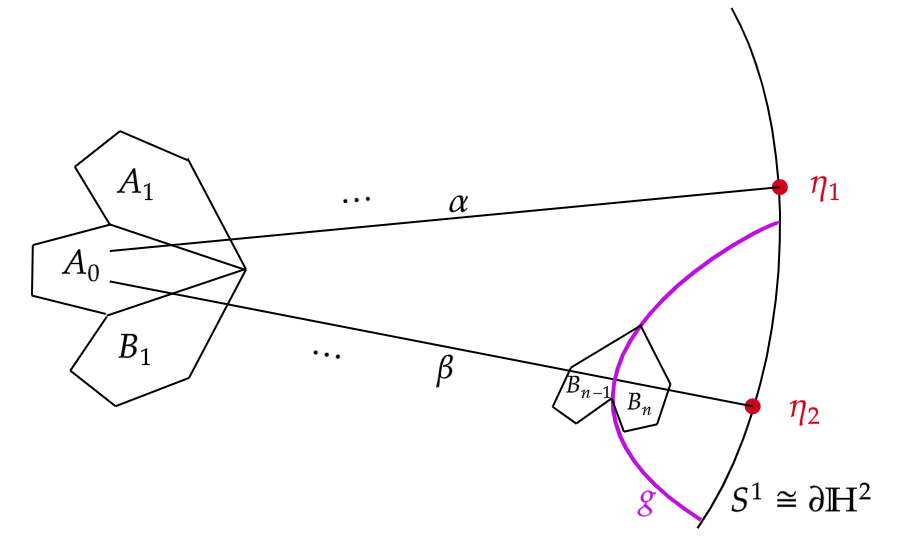}
\caption{Tiling paths $\calA_\alpha$ and $\calA_\beta$ have different endpoints $\eta_1$ and $\eta_2$ respectively that are separated by edge geodesic $g$.}
\label{fig:infinitebijection}
\end{figure}

Let $B_n$ be the first tile of $\calA_\beta$ that is on the $\eta_2$ side of $g$. If the tiling paths $\calA_\alpha$ and $\calA_\beta$ were word equivalent, then there would exist another minimal tiling path $\calA_\alpha'$ that matches with $\calA_\beta$ for at least the first $n+1$ tiles, but would differ from $\calA_\alpha$ by finitely many vertex sequence equivalence moves. Thus, the endpoint of $\calA_\alpha'$ is still $\eta_1$, the same as the endpoint of $\calA_\alpha$. But now, $\calA_\alpha'$ is a minimal tiling path that double crosses the edge geodesic $g$, which contradicts Proposition $\ref{prop:minimal_tiling_path_characterization}$. Hence, $w_\alpha$ and $w_\beta$ are not in the same admissible word class, and so $\Phi$ is injective.

Finally, to prove surjectivity, let $[w] \in \mathcal{W}$ be an infinite admissible word class and let $\eta \in \partial \HH^2$ be the corresponding endpoint whose existence is asserted in Proposition~\ref{prop:endpoints}. Let $\calA$ be a tiling path corresponding to $w$ that starts at tile $A_0$ and is minimal by Proposition~\ref{prop:infinite_minimal_admissible}. Now, let $\alpha$ be a geodesic ray from a point in the interior of tile $A$ to $\eta$. Let $\calA_\alpha$ be the corresponding tiling path, which is minimal since $\alpha$ is a geodesic ray. As $\calA$ and $\calA_\alpha$ are both minimal tiling paths that share the starting tile and ending point, by Proposition~\ref{prop:minimal_same_endpoint}, they and their corresponding words are equivalent. By definition of $\Phi$, this means, $\Phi([\alpha]) = [w]$.  
\end{proof}

\subsection{A characterization of bi-infinite billiard words when $q$ is even}\label{subsec:qevenBiInfinite}

We now use our results about what one-sided infinite word classes are realized by a billiard trajectory to prove the following theorem about bi-infinite centered word classes (see Definition~\ref{defn:centeredwordclass}).

\GUbiinfinite*

As an immediate corollary to this theorem, we have the following, which we get by folding up the geodesic into a billiard trajectory. We note that the billiard trajectory that realized a centered bi-infinite word class may be \textbf{singular}, in that it passes through one or more vertices of the billiard table. 

\begin{corollary} In a hyperbolic billiard table obtained from a tile of a hyperbolic $(p,q)$-tiling with $q$ even, every centered bi-infinite admissible word class is achieved by a unique (possibly singular) billiard trajectory. 
\end{corollary}

We note that given any vertex $v$, any bi-infinite admissible tiling path $\calA$ (i.e. one whose associated word is admissible)  must have a well-defined beginning and ending sector (see Definition~\ref{def:sector}), since the path can cross each edge geodesic through $v$ at most once. We note also that $\calA$ and the tiling path $\mathcal{B}_g$ associated to the geodesic $g$ connecting the limit points $\eta_{\infty}$ and $\eta_{-\infty}$ of $\calA$ have the same beginning and ending sectors whenever $\eta_{-\infty}$ and $\eta_\infty$ are not at the endpoints of the edge geodesics through $v$. When $\eta_{-\infty}$ or $\eta_\infty$ are at the endpoints of these edge geodesics, we notice that the beginning or ending sectors respectively of $\calA$ and $\mathcal{B}_g$ are at most one apart. 

We will need the following hyperbolic geometry lemmas, whose proofs can be found in Appendix \ref{appdx:hyperbolictilings}.

\begin{restatable}{lemma}{fournonintersecting}
\label{lem:p4_nonintersecting}
    Consider a hyperbolic $(p,q)$-tiling with $p \geq 4$ and $q$ even. Let $g_1$ and $g_2$ be two edge geodesics that intersect at a vertex $v$. Let $g_3$ be an edge geodesic that intersects $g_2$ at exactly one vertex $w \neq v$. Then $g_3$ does not intersect $g_1$. 
\end{restatable}

\begin{restatable}{lemma}{threenonintersecting}
\label{lem:p3_nonintersecting}
    Consider a hyperbolic $(p,q)$-tiling with $p =3$ and $q$ even. Let $g_1$ and $g_2$ be edge geodesic rays emanating from a vertex $v$ with angle at $v$ that is at least $6\pi/q$. Let $g_3$ be any edge geodesic that intersects $g_2$ at exactly one vertex $w \neq v$. Then $g_3$ does not intersect $g_1$. 
\end{restatable}

We can then use these lemmas to prove the following result about tiling paths that end in far enough apart sectors.

\begin{lemma} Consider a hyperbolic $(p,q)$-tiling with $q$ even. Let $v$ be a vertex that defines $q$ sectors.
\label{lem:nonadjacent_sectors}
Suppose $\calA$ is an admissible bi-infinite tiling path that begins and ends in non-adjacent sectors when $p \geq 4$ or in sectors separated by at least two sectors when $p = 3$. Then $\calA$ is equivalent to a tiling path $\mathcal{B}$ which contains a subpath that traverses $v$ from the initial sector to the ending sector and is completely in the initial sector before and is completely in the final sector after the vertex traversal.
\end{lemma}
\begin{proof}
We suppose that $p \geq 4$, and that $\calA = (\ldots, A_{-1}, A_0, A_1, \ldots)$ is an admissible bi-infinite tiling path that begins and ends in non-adjacent sectors relative to a vertex $v$. Then, $\calA$ must have a first and last time that the path crosses the edge geodesics through $v$. Without loss of generality, we suppose that those crossings happen at the subpaths $(A_0,A_1)$ and $(A_{n-1},A_n)$. We label the edge geodesics crossed as $g_1$ and $g_2$ respectively. We claim then that $\calA$ is vertex equivalent to a tiling path $\calA'$ where $(A_0,A_1, \ldots, A_n)$ is replaced by a tiling path $(A_0' =A_0, A_1', \ldots, A_{n-1}', A_n' = A_n)$ that fellow travels $g_1$ toward $v$, does a vertex traversal of $v$, and then fellow travels $g_2$ back to $A_n'=A_n$. See Figure~\ref{fig:vertex_traversal_replacement} for a depiction of this new tiling path. 

\begin{figure}[h!]
\centering
\includegraphics[width=12cm]{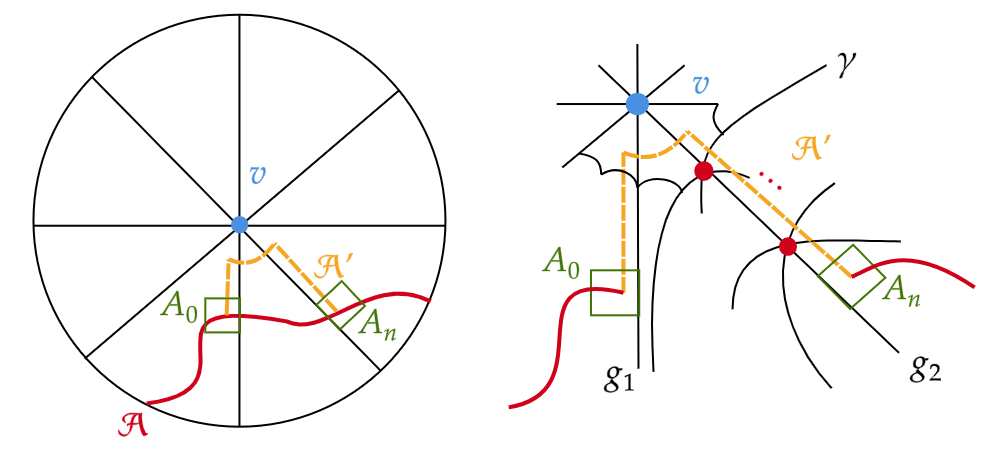}
\caption{A schematic for how a tiling path $\mathcal{A}$ that begins and ends in sufficiently far apart sectors relative to a vertex $v$ is vertex sequence equivalent to a tiling path that includes a vertex traversal of $v$.}
\label{fig:vertex_traversal_replacement}
\end{figure}

To prove that $\calA$ and $\calA'$ are vertex sequence equivalent, by Proposition~\ref{prop:minimal_tiling_paths_word_equivalent} it suffices to show that $(A_0', A_1', \ldots, A_n')$ is minimal and therefore is vertex sequence equivalent to $(A_0, A_1, \ldots,A_n)$. By Proposition~\ref{prop:minimal_tiling_path_characterization}, it suffices to show that $(A_0', A_1', \ldots, A_n')$ does not double cross any edge geodesics. The only candidate edge geodesics that the path could double cross are the edge geodesics separating sequential tiles. By construction, $(A_0', A_1', \ldots, A_n')$ does not double cross the edge geodesics throuh $v$ and by Lemma~\ref{lem:p4_nonintersecting}, it follows that the path does not double cross any edge geodesics passing through $g_1$ or $g_2$. Hence, the lemma in the $p \geq 4$ case follows. 

We note that the proof in the $p \geq 3$ case is similar but uses Lemma~\ref{lem:p3_nonintersecting} instead of Lemma~\ref{lem:p4_nonintersecting} in the last step. 
\end{proof}

\begin{proof}[Proof of Theorem~\ref{thm:biinfinite_realization}] 

Given a centered, bi-infinite, admissible word class, choose a centered element $w$ of the word class, which corresponds to a tiling path $(\ldots, A_{-1}, A_0, A_1,\ldots)$ centered at the tile $A_0$. By Proposition~\ref{prop:endpoints}, the associated forward and backward tiling paths $(A_0, A_1, A_2, \ldots)$ and $(\ldots, A_{-2}, A_{-1}, A_0)$  then have unique endpoints $\eta_\infty$ and $\eta_{-\infty}$ respectively. Let $g$ be the geodesic connecting $\eta_{-\infty}$ to $\eta_\infty$, and let $\mathcal{B}_g  = (\ldots, B_{-1}, B_0, B_1,\ldots)$ be the tiling path of $g$, which also gives rise to a centered word $w_g$. We will show that the centered word $w_g$ of $g$ is in the centered word class $[w]$. 

\textbf{Case 1 ($\mathcal{A}$ and $\mathcal{B}_g$ share a tile):} If the original tiling path $\mathcal{A}$ and the geodesic tiling path $\mathcal{B}_g$ ever share a tile $A_k$, then by Proposition~\ref{prop:minimal_same_endpoint}, the forward and backward subwords of $w_g$ and $w$, based at tile $A_k$, must then be equivalent. It follows then that $w_g$ and $w$ are equivalent.

\textbf{Case 2 ($\mathcal{A}$ and $\mathcal{B}_g$ do not share a tile, and $\mathcal{A}$ contains adjacent exits):}
We suppose $\mathcal{A}$ contains a tiling subpath $(A_{k-1}, A_{k}, A_{k+1})$ that exits out to adjacent sides of $A_{k}$ and therefore vertex traverses the vertex $v$ shared by $A_{k-1},A_k,$ and $A_{k+1}$. This vertex $v$ splits the tiling into $q$ sectors. Since $\calA$ does not double cross any of the edge geodesics through $v$, we see that $\mathcal{A}$ cannot begin and end in the same or adjacent sectors. 

\textbf{Case 2a:} We first suppose that both $\mathcal{A}$ and $\mathcal{B}_g$ start and end in sectors that are non-adjacent if $p \geq 4$ and separated by at least two sectors when $p = 3$. Then, by Lemma~\ref{lem:nonadjacent_sectors}, both $\mathcal{A}$ and $\mathcal{B}_g$ are vertex sequence equivalent to tiling paths that include a vertex traversal of $v$ starting and ending in their respective starting and ending sectors. Since the starting sectors of $\mathcal{A}$  and $\mathcal{B}_g$ must be the same or adjacent (and likewise for the ending sectors), it follows that their respective vertex traversals either share a tile or one can be modified by a vertex sequence equivalence so that they share a tile. Thus, in this case, $\mathcal{A}$ and $\mathcal{B}_g$ are vertex sequence equivalent to tiling paths that share a tile, and we can appeal to Case 1. 

\textbf{Case 2b:} Now we work on the first of two edge cases (see Figure~\ref{fig:biinfinite_realization}, left image). Assume $\mathcal{A}$ starts and ends in sectors that are non-adjacent when $p \geq 4$ and separated by at least two sectors when $p = 3$, but $\mathcal{B}_g$ does not. This can only happen if $\eta_{-\infty}$ or $\eta_\infty$ is at the endpoint of an edge geodesic through $v$, the center of the sectors. We suppose without loss of generality that this happens at $\eta_\infty$, which is the endpoint of a ray $g_\infty$ from $v$ to $\partial \mathbb{H}^2$. Let $A_j$ be the first tile in $\calA$ adjacent to $g_\infty$. By Theorem~\ref{thm:infinite_bijection}, $\calA$ is equivalent to a path $\calA'$ that matches $\calA$ until tile $A_j$ and then fellow travels $g_\infty$ without crossing it afterward. Likewise, after choosing a tile $B_k$ in the same sector as the forward tail of $\mathcal{B}_g$, we can apply Theorem~\ref{thm:infinite_bijection} to see that $\calB_g$ is vertex sequence equivalent to a tiling path $\calB'_g$ that matches $\calB_g$ until $B_k$ and then travels toward $g_\infty$ and fellow travels $g_\infty$ towards $n_\infty$ without crossing $g_\infty$. Now, $\calA'$ and $\calB'_g$ fellow travel $g_\infty$ on the same side, and so they must share a tile and we are back in Case 1. 


\begin{figure}[h!]
\centering
\includegraphics[width=5cm]{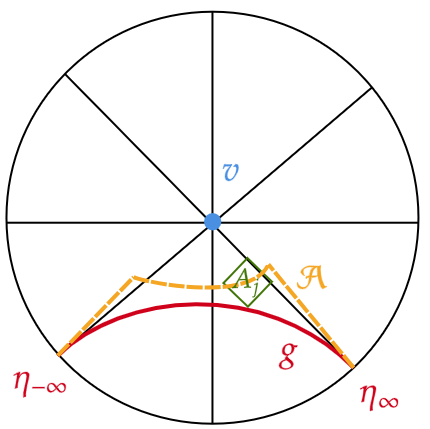}
\includegraphics[width=5.3cm]{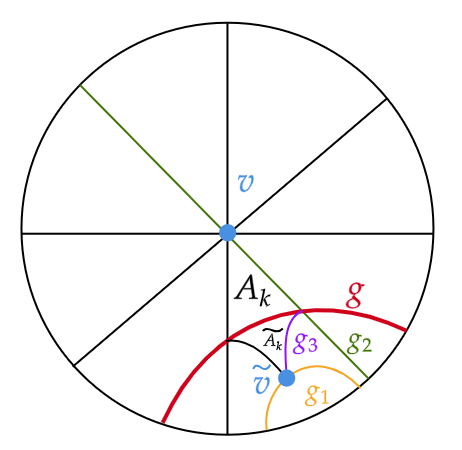}
\includegraphics[width=5cm]{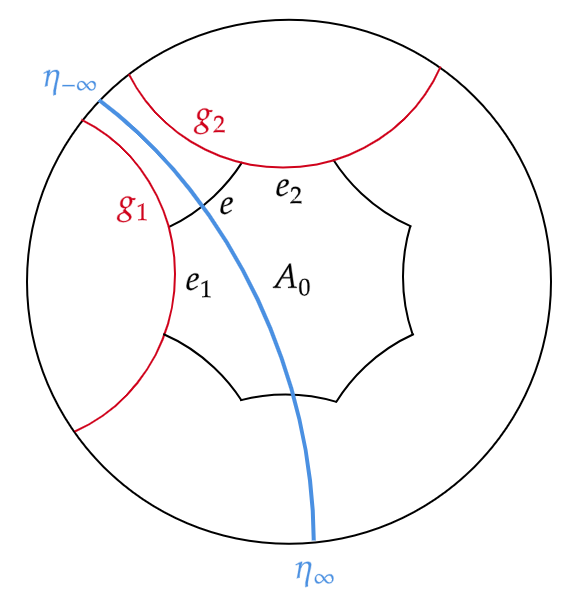}
\caption{Depictions of three of the cases in the proof of Theorem~\ref{thm:biinfinite_realization}: Case 2b (left), Case 2c (middle), and Case 3 (right)}
\label{fig:biinfinite_realization}
\end{figure}

\textbf{Case 2c:} We now consider the last edge case when $p=3$ but $\mathcal{A}$ ends in sectors that are only one apart (see Figure~\ref{fig:biinfinite_realization}, middle image). Let $g$ denote the edge geodesic through the edge opposite $v$ in $A_k$, and let $\tilde{A_k}$ denote the tile that is the reflection of $A_k$ through $g$. 

We will first show that any geodesic connecting the backward and forward limit points $\eta_{-\infty}$ and $\eta_\infty$ must pass through $A_k$ or $\tilde{A_k}$. To see this, let $g_1$ be an edge geodesic through the vertex $\tilde{v}$ on $\tilde{A_k}$ that is the reflection of $v$ through $g$, such that $g_1$ makes an angle of at least $\frac{2\pi}{q}$ with each edge of $\tilde{A_k}$ through $\tilde{v}$. Let $g_3$ be one of the edges of $\tilde{A_k}$, and let $g_2$ be the edge geodesic through $v$ that intersects $g_3$. We notice that $g_3$ and $g_2$ meet at an angle of $\frac{(q-2)\cdot 2\pi}{q}$. Thus, $g_1$ and $g_2$ cannot intersect because their intersection angle would be $\geq \frac{2\pi}{q}$ and thus the resulting triangle angle sum of the triangle with sides $g_1,g_2, g_3$ would be $\geq \pi$. Symmetrically, $g_1$ does not intersect the other edge geodesic through $v$ that is adjacent to $A_k$. It follows that for any endpoints $\eta_{-\infty}$ and $\eta_{\infty}$ in the two sectors adjacent to $A_k$, the geodesic connecting them must pass through either $A_k$ or $\tilde{A_k}$. In the former case, $\mathcal{B}_g$ and $\mathcal{A}$ share the tile $A_k$ and we are back in Case 1. In the latter case, arguing as in the end of Case 2b shows us that $\mathcal{A}$ is vertex sequence adjacent to a word that passes through $\tilde{A_k}$, which $\mathcal{B}_g$ also passes through, and we can again appeal to Case 1.

\textbf{Case 3 ($\mathcal{A}$ and $\mathcal{B}_g$ do not share a tile, and $\mathcal{A}$ does not contain adjacent exits):} Assume $\calA = (\dots, A_{-1}, A_0, A_1, \dots)$ does not contain any adjacent tile exists (see Figure~\ref{fig:biinfinite_realization}, right image). Note by Proposition~\ref{prop:minimal_tiling_path_characterization}, as $\calA$ is minimal. Let $e$ be the edge separating $A_{-1}$ and $A_0$, and let $g_1$ and $g_2$ be the two edge geodesics extending the edges $e_1$ and $e_2$ respectively on $A_0$ adjacent to $e$. Assume first $\mathcal{A}$ crosses $g_1$ or $g_2$. By construction it cannot do so at $e_1$ or $e_2$. Thus, if $\calA$ crosses $g_1$ or $g_2$, it is vertex sequence equivalent to a path that fellow travels $g_1$ or $g_2$ and then crosses, by Lemma~\ref{lem:necessarilyfellowtraveling}. Fellow traveling includes adjacent exits, so we appeal to Case 2 in this case.

We suppose now that $\mathcal{A}$ does not cross $g_1$ or $g_2$. This implies that the backward and forward limit points $\eta_{-\infty}$ and $\eta_\infty$ of $\mathcal{A}$ are between the backward and forward endpoints of $g_1$ and $g_2$. A geodesic connecting such endpoints must pass through $A_0$, and so the tiling path $\mathcal{B}_g$ associated and $\mathcal{A}$ would share a tile, putting us back in Case 1. 
\end{proof}

%% file: computations.tex
\section{Comparing various growth rates for the $q$ odd case}
\label{sec:comparison}
In this section we compare the different types of growth rates discussed in the earlier section. More specifically, we compare the upper and lower bounds to billiard language complexity (for billiards in a regular hyperbolic polygon with $p$ sides and internal angles $2\pi/q$) that can be deduced via Theorem~\ref{thm:GUodd} by Giannoni and Ullmo to the upper bound to the same quantity obtained via Theorem~\ref{thm:complexity_evenodd}. As in the $q$ even case, Giannoni and Ullmo did not include a proof of their theorem in their published paper and we found it difficult to independently verify that their proof methods would work. Proposition~\ref{prop:growthratecomparisons}  in this section proves that our growth rate upper bounds from Theorem~\ref{thm:complexity_evenodd} are better or at least as good as those given by the theorem of Giannoni and Ullmo, assuming their theorem is correct. For small $p$ and $q$, our computational results suggest that our growth rate upper bounds are indeed strictly better.

\subsection{A result of Giannoni and Ullmo}

In the case where $q$ is odd, Giannoni and Ullmo give approximate grammar rules that do not determine the complete set of achieved bi-infinite words, but allows them to say that certain bi-infinite words are achieved and certain other bi-infinite words are not achieved. 

\begin{restatable}[Language rules for $q$ odd (\cite{GiannoniUllmo})]{theorem}{GUodd}\label{thm:GUodd} 
In a hyperbolic polygonal billiard corresponding to a $(p, q)$-tiling with q odd, bi-infinite words coming from billiard trajectories must follow the following rules:

\begin{enumerate}
    \item[O1]\label{grule1} Any word containing a two-letter subword $kk$  for any $k \in \{1,\ldots,p\}$ is not allowed.
    \item[O2]\label{grule2} Any word containing a subword of length greater than or equal to $\frac{q}{2}+\frac{3}{2}$ consisting of alternating letters $k$ and $k+1$ (or $p$ and $1$) is not allowed. 
\end{enumerate}

Furthermore, any word containing no two-letter subwords as in O1 and no subwords of length $\frac{q}{2}-\frac{1}{2}$ consisting of alternating letters $k$ and $k+1$ (or $p$ and $1$) is realized by a billiard path. 
\end{restatable}

This is a semi-characterization of the billiard words since, unlike the case when $q$ is even, it is not true that all words not forbidden by rules O1 and O2 are achieved with billiard trajectories. For example, when $p = 4$ and $q = 7$, the theorem asserts that words that \emph{do not contain} subwords such as $010$ are definitely realized as billiard words and words that \emph{contain} subwords such as $01010$ cannot be realized as billiard words. However, the theorem neither forbids nor asserts the realization of words that \emph{contain} a subword such as $010$ or $0101$. In fact, in Figure~\ref{fig:incompletegrammar} we demonstrate both such words can appear from billiard paths. In the same figure we also demonstrate a word, $2121323$, that is not forbidden by the theorem but one that cannot be realized as a billiard word, by applying Proposition~\ref{prop:consecutivezigzagintersection}. 

We note that due to the semi-characterization of the billiard words, Theorem~\ref{thm:GUodd} yields non-sharp upper and lower bounds on the language complexity function of hyperbolic billiards associated to a $(p,q)$-tiling when $q$ is odd:
\begin{enumerate}
    \item (Upper Bound) We can obtain an upper bound on the language complexity of billiards by computing the growth of all words in $p$ letters except for those forbidden by rules O1 and O2. 

    \item (Lower Bound) We can obtain a lower bound on the language complexity of billiards by computing the growth of all finite words that are guaranteed to be realized in Theorem~\ref{thm:GUodd}. In other words, a lower bound is given by the growth of all words in $p$ letters that do not have a two-letter subword as in O1 and do not have any subword of length $\frac{q}{2}-\frac{1}{2}$ consisting of alternating letters $k$ and $k+1$ (or $p$ and $1$). 
\end{enumerate}

In the next subsection we explore these bounds further.

    \begin{figure}[h!]
    \centering
    \begin{subfigure}{.45\textwidth}
    \centering
    \includegraphics[scale=0.6]{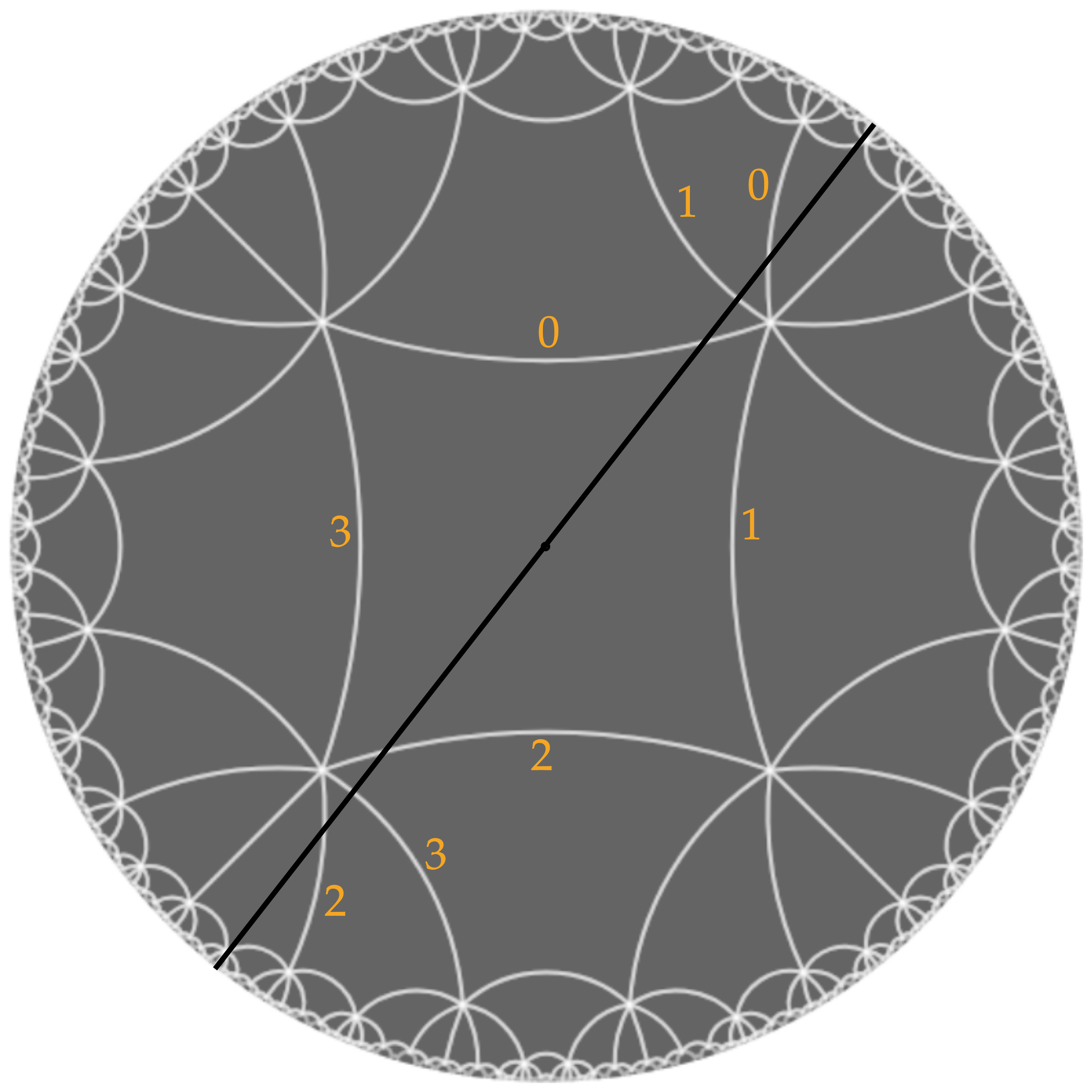}
        \caption{A billiard path whose billiard word contains 010 as a subword.}
        \label{fig:010word}
    \end{subfigure}
    \begin{subfigure}{.45\textwidth}
    \centering
    \includegraphics[scale=0.6]{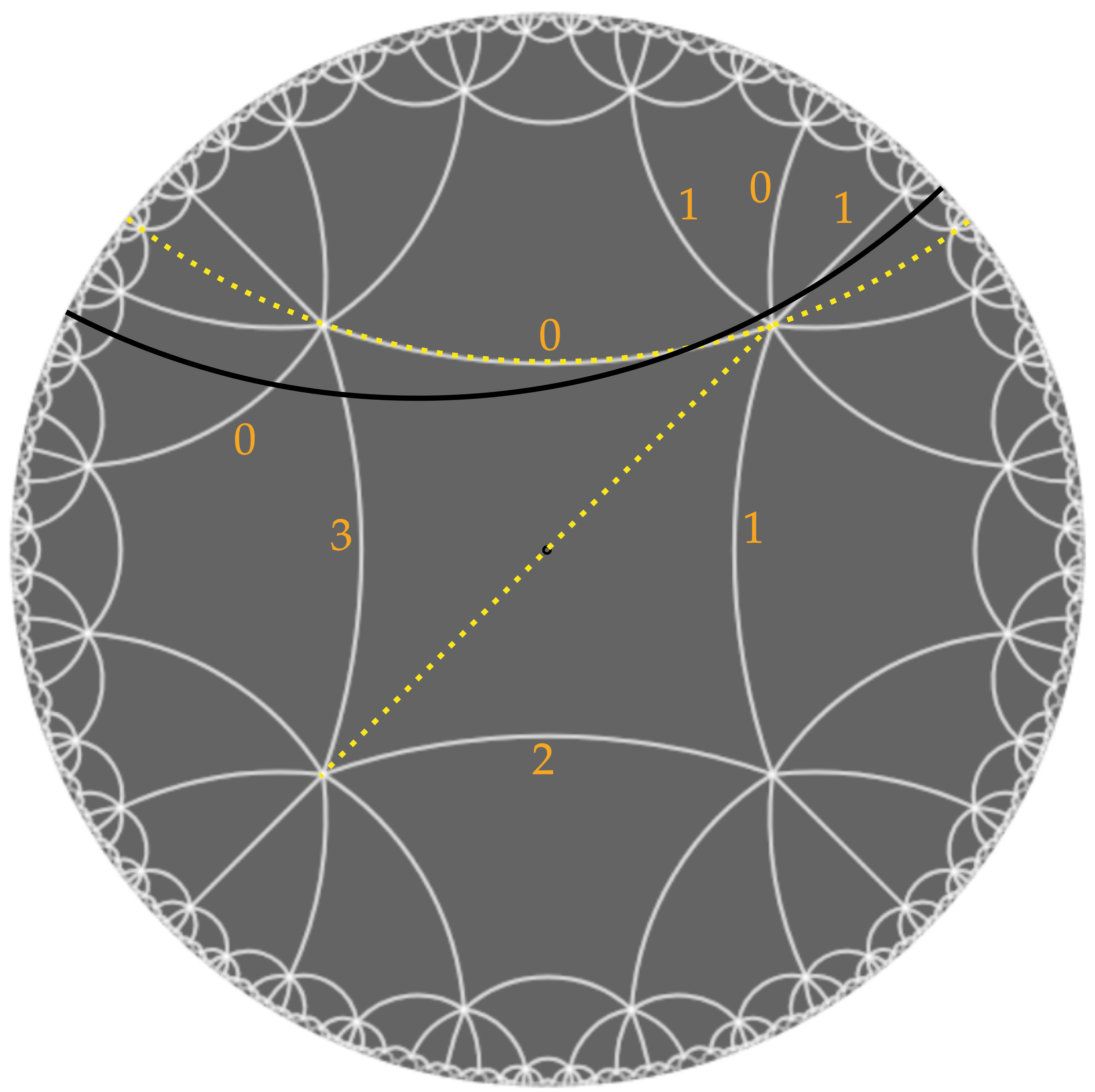}
        \caption{A billiard path whose billiard word contains 0101 as a subword.}
        \label{fig:0101word}
    \end{subfigure}\\
    \begin{subfigure}{.5\textwidth}
    \centering
    \includegraphics[scale=0.6]{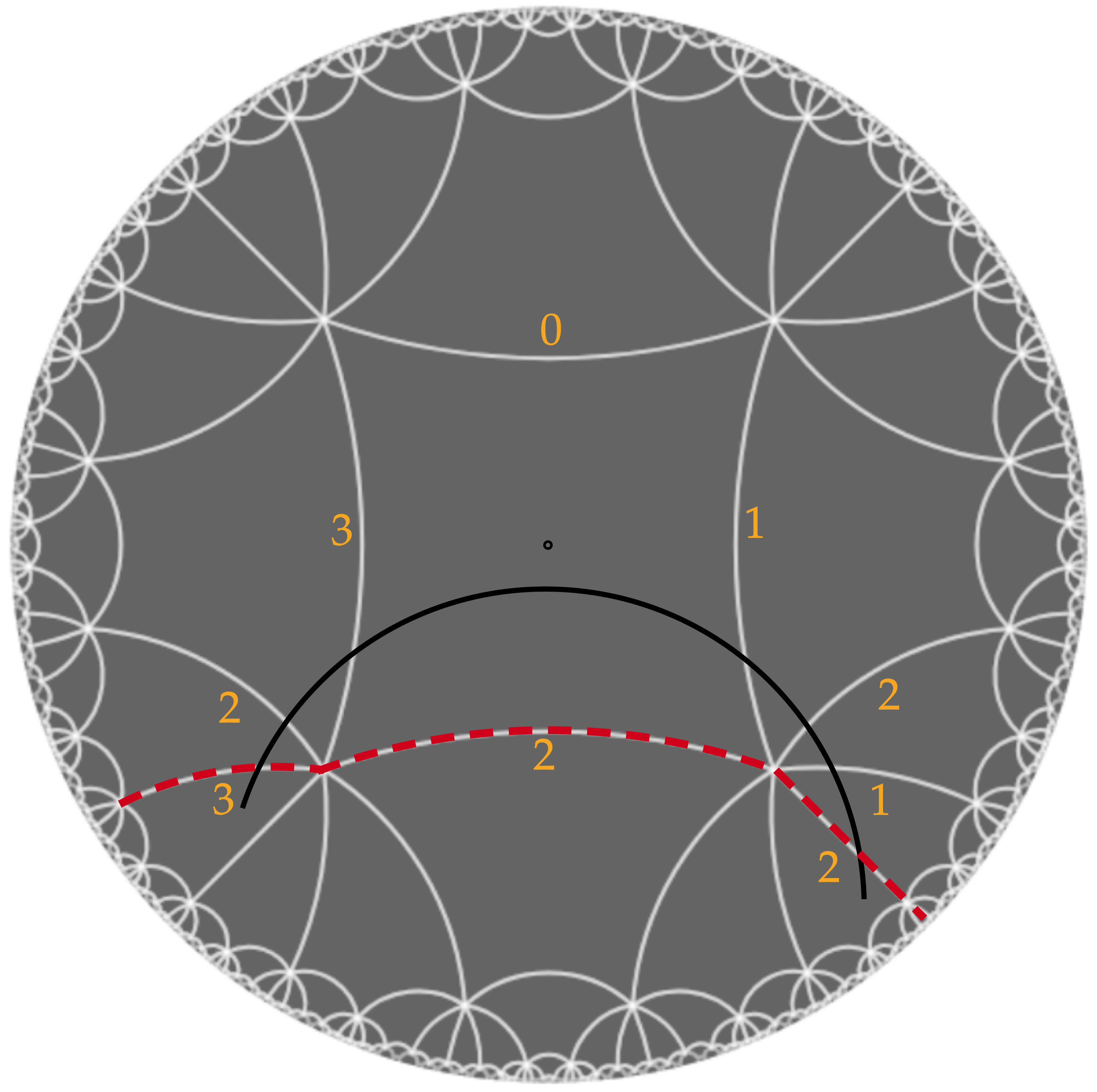}
        \caption{Word 2121323 that cannot be realized by a billiard trajectory since it would have to give rise to a geodesic segment that double intersects a zigzag (shown as dashed) non-consecutively violating Proposition~\ref{prop:consecutivezigzagintersection}.}
        \label{fig:0101word}
    \end{subfigure}

    \caption{The first two figures show examples of billiard trajectories in a regular hyperbolic polygon with 4 sides and $2\pi/7$ internal angle unfolded onto the hyperbolic plane. The third figure shows a word that is not forbidden by Theorem~\ref{thm:GUodd} but one that cannot be realized as a billiard trajectory.}
    \label{fig:incompletegrammar}
    \end{figure}

\subsection{Bounds on growth rates}

Consider a hyperbolic $(p,q)$-tiling with $q$ odd. Fix a base tile $A$ and recall $\ntd(k)$ denotes the number of tiles that are tiling distance $k$ away from $A$. Let $\nwlupper(n)$ be the number of words of length $n$ in the language with forbidden set given by O1 and O2 in Theorem~\ref{thm:GUodd}. 
Let $\nwllower(n)$ be the number of words of length $n$ containing no two-letter subwords as in O1 and no subwords of length $q/2-1/2$ of alternating letters $k$ and $k+1$ (or $p$ and 1).

Assuming Theorem~\ref{thm:GUodd}, the exponential growth rates of $\nwllower(n)$ and $\nwlupper(n)$ give lower and upper bounds for the exponential growth rate of $p(n)$. We will see that by using Theorem~\ref{thm:complexity_evenodd}, we can improve the upper bound and sometimes improve the lower bound by using the growth rate of $\ntd(n)$. 

\begin{proposition}\label{prop:growthratecomparisons}
Let $\ell, h_{top}, \alpha,$ and $u > 1$ be the exponential growth rates of $\nwllower(n)$, $p(n)$, $\ntd(n)$ and $\nwlupper(n)$ respectively. Then, $\ell \leq h_{top} \leq \alpha \leq u$, and $h_{top} \geq \alpha^{\frac{q-1}{q+1}}$.

\end{proposition}
\begin{proof}
Note that for each $n\in\NN$,  $\nwllower(n) \leq p(n)$ from Theorem~\ref{thm:GUodd}. So, $\ell \leq h_{top}$. By Theorem~\ref{thm:complexity_evenodd}, $\alpha^{\frac{q-1}{q+1}} \leq h_{top} \leq \alpha$. Finally, after fixing a base tile $A_0$, a tile $A_n$ that is counted in $\ntd(n)$ has tiling distance $n$ from $A_0$. Choosing a minimal tiling path from $A_0$ to $A_n$ gives us a word that satisfies O1 and O2 from Theorem~\ref{thm:GUodd} since otherwise, the tiling path could be made shorter by removing backtracking (if it violated O1) or by traversing a vertex the other way (if it violated O2). Two distinct tiles counted in $\ntd(n)$ must correspond to distinct words counted in $\nwlupper(n)$. Thus, we can find an injective map from $\ntd(n)$ to $\nwlupper(n)$, showing that $\ntd(n) \leq \nwlupper(n)$, which implies $\alpha \leq u$.    
\end{proof}

Next, to illustrate Proposition~\ref{prop:growthratecomparisons}, we computationally compare the growth rates of $\nwllower(n),\ntd(n)$ and $\nwlupper(n)$. We use a well known procedure to compute the growth rate of a language specified using a set of forbidden words. Given a language specified by a finite alphabet $\mathcal{A}$ and a finite set of forbidden words $\mathcal{F}$, let $m$ be the length of the longest forbidden word. To compute the growth rate of the set of all words $X$ in the language, we first create a directed graph $G$ where the vertex set is the set of all words of length at most $m-1$ that are not in $\mathcal{F}$ and there exists a (directed) edge from vertex $u$ to $v$ if and only if the $m-2$ letter suffix of $u$ matches the $m-2$ letter prefix of $v$ and the amalgamation of $u$ with $v$ along the matching suffix/prefix is not in $\mathcal{F}$. Directed walks in $G$ then represent allowed words in the language. It is known that the growth rate of such walks is given by the Perron eigenvalue (unique largest, real, positive eigenvalue guaranteed by the Perron-Frobenius theorem) of the transition matrix of $G$. See \cite{BrinStuck} for a full reference of this procedure.

The following example demonstrates this procedure in computing the growth of $\nwllower(n)$ and $\nwlupper(n)$ when $p=3$ and $q=7$.

\begin{example}\label{ex:OddGrowthRateComparisonExample}
When $(p,q) = (3,7)$, the alphabet is $\{0, 1, 2\}$. 

To compute the growth of $\nwlupper(\cdot)$, set the forbidden words is $$\mathcal{F^\mathrm{upper}} = \{00, 11, 22, 01010, 02020, 12121, 10101, 20202, 21212\}.$$ We create a graph $G$ with vertices 
\begin{align*}
    V = \{&0101, 0102, 0120, 0121, 0201, 0202, 0210, 0212, 1010, 1012, 1020, 1021,\\ &1201, 1202, 1210, 1212, 2010, 2012, 2020, 2021, 2101, 2102, 2120, 2121\}
\end{align*}
There exists a directed edge from $v$ to $w$ where $v, w \in V$ if and only if the 3-letter suffix of $v$ matches the 3-letter prefix of $w$ and the amalgamation of $v$ and $w$ over the matching suffix/prefix is not a forbidden word. For instance, there exists a directed edge from $0101$ to $1012$ since the 3-letter suffix/prefix of the words $101$ match and the amalgamation over this suffix/prefix $01012 \not \in \mathcal{F}$. However, no edge exists from $0101$ to $1010$ since the amalgamation $01010 \in \mathcal{F}$. The transition matrix $T$ of $G$ is then constructed, by setting
 $$T_{ij} = \begin{cases} 0 & \text{if no edge from }v_i\text{ to }v_j \\ 1 & \text{if there is an edge from }v_i\text{ to }v_j \end{cases} $$
The language growth rate $\approx 1.83928675521416$ is given by the top eigenvalue of the transition matrix $T$. 

For $\nwllower(\cdot)$, the set of forbidden words is  $\mathcal{F}^\mathrm{lower} = \{00, 11, 22, 010, 020, 121, 101, 202, 212\}$. The growth of the language is then similarly computed to be $1$.
\end{example}
In Table \ref{table:oddgrowthrates} we implement the prodecure from Example~\ref{ex:OddGrowthRateComparisonExample} for various $(p,q)$-tilings with $q$ odd and compare growth rates of $\nwllower(n)$ and $\nwlupper(n)$ with $\alpha$ and $\alpha^{\frac{q-1}{q+1}}$ where $\alpha$ is the exponential growth rate of $\ntd(n)$.

\begin{table}[h!]
\centering
\begin{tabular}{cc|p{3cm}|p{3cm}|p{3cm}|p{3cm}}
&  & \multicolumn{2}{c|}{Billiard Language Complexity Lower Bounds} & \multicolumn{2}{c}{Billiard Language Complexity Upper Bounds}\\
 \cline{3-6}
$p$ & $q$ & $\ell$ & $\alpha^{\frac{q-1}{q+1}}$& $\alpha$ & $u$ \\
\hline
3 &7 &1.00000000000000 & 1.39320015609277  &1.55603019132268 &1.83928675521416 \\
3 &9 &1.61803398874989 & 1.62242661883033  &1.83107582510231 &1.92756197548293 \\
4 &5 &1.00000000000000 & 1.74071386619816 &2.29663026288654 &2.83117720720834 \\
4 &7 &2.41421356237309 & 2.17799288640350  &2.82320193241387 &2.94771158684464 \\
4 &9 &2.83117720720834 & 2.37412444801062 &2.94699466977899 &2.98314000701146 \\
5 &5 &2.00000000000000 & 2.30788179437972 &3.50606805595024 &3.90057187491196 \\
5 &7 &3.56155281280883 & 2.77414842757632 &3.89797986736932 &3.97608928489373 \\
5 &9 &3.90057187491196 & 3.01683949575744 &3.97594397745373 &3.99410304410951 \\
6 &5 &3.00000000000000 & 2.77055394307089 &4.61158178930871 &4.93394490940640 \\
6 &7 &4.64575131106459 & 3.30995335049871 &4.93282638839610 &4.98708779528314 \\
6 &9 &4.93394490940640 & 3.61638522480061 &4.98704581211217 &4.99743424469419 \\
7 &3 &1.00000000000000 & 1.61803398874989 &2.61803398874989 &5.70156211871642 \\
7 &5 &4.00000000000000 & 3.18270263144865 &5.67798309021366 &5.95281778935296 \\
7 &7 &5.70156211871642 & 3.81075501946865 &5.95225287964244 &5.99224996607951 
\end{tabular}
\caption{Tiling growth rate, $\alpha$, in comparison to lower and upper bounds, $\ell$ and $u$ respectively, on the billiard language complexity.}
\label{table:oddgrowthrates}
\end{table}

From Theorem~\ref{thm:complexity_evenodd}, we know that in the odd $q$ case, the tiling exponential growth rate $\alpha$ is an upper bound for the billiard language complexity. The computational data in Table \ref{table:oddgrowthrates} illustrates our result from Proposition~\ref{prop:growthratecomparisons} by demonstrating that the upper bound for the language complexity growth that we obtain in Theorem~\ref{thm:complexity_evenodd} improves upon the upper bound deduced from Theorem~\ref{thm:GUodd}. Similarly, from Theorem~\ref{thm:complexity_evenodd}, $\alpha^\frac{q-1}{q+1}$ is a lower bound for the billiard language complexity. Table~\ref{table:oddgrowthrates} illustrates that this bound \emph{sometimes} improves upon the lower bound deduced from Theorem~\ref{thm:GUodd}.

%% file: appendix.tex
\newpage
\appendix

\section{Some facts about hyperbolic tilings}\label{appdx:hyperbolictilings}

In this section we prove some basic facts about hyperbolic $(p,q)$-tilings that we utilize in our proofs. The proofs are elementary and use classical results from hyperbolic geometry. 
\subsection{Facts about $(p,q)$-tilings when $q$ is odd}

\begin{wrapfigure}{r}{0.3\textwidth}
\centering
\includegraphics[width=4cm]{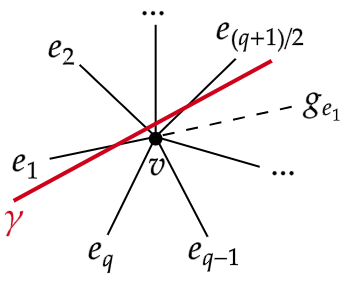}
\caption{A geodesic $\gamma$ and how it intersects edges adjacent to a vertex $v$. Here, $\gamma$ double-intersects the zigzag containing edge $e_1$ and $e_{(q+1)/2}$. }
\label{fig:q_odd_vertex_sequence}
\end{wrapfigure}

We first note that when a hyperbolic geodesic $\gamma$ intersects edges adjacent to a vertex $v$ of the tiling, it intersects edges $e_1, e_2, \ldots, e_{n}$ cyclically around $v$ with $n \leq \frac{q+1}{2}$ (since any more intersections would cause $\gamma$ to double intersect the geodesic extension $g_{e_1}$ of $e_1$ (see Figure \ref{fig:q_odd_vertex_sequence}). Furthermore, $\gamma$ cannot intersect any other edges of the tiling in between these intersections (with the edges adjacent to $v$), by the convexity of the tiles.

Now, we will prove a few results about zigzags in hyperbolic $(p,q)$-tilings with $q$ odd. We start by proving that zigzag paths (see Definition \ref{def:zigzag}) in a hyperbolic $(p,q)$-tiling with $q$ odd are quasigeodesics.

\begin{wrapfigure}[11]{r}{0.48\textwidth}
    \centering
        \includegraphics[width=8cm]{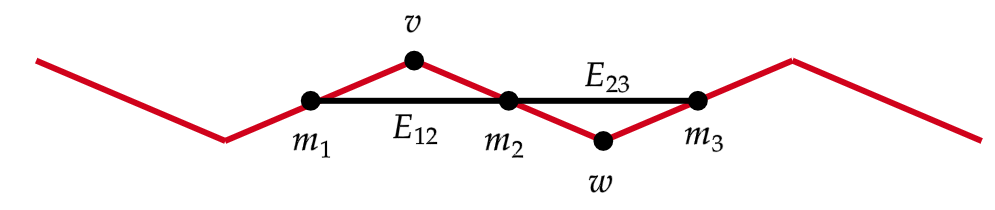}
        \caption{By the rotational symmetry of a zigzag around a midpoint $m_2$, two segments connecting successive midpoints of zigzag edges must meet at an angle of $\pi$. Thus, there is a geodesic connecting all midpoints of a zigzag.}
        \label{fig:zigzag_midpoint}
    \end{wrapfigure}

\zigzagQuasiGeod*

\begin{proof}
Let $m_1, m_2$ and $m_3$ denote the midpoints of three successive edges on a zigzag, and let $v$ and $w$ be the vertices of the zigzag between $m_1$ and $m_2$, and between $m_2$ and $m_3$ respectively (see Figure \ref{fig:zigzag_midpoint}). Let $E_{12}$ and $E_{23}$ denote the geodesic segments connecting $m_1$ and $m_2$, and $m_2$ and $m_3$ respectively. We will show that $E_{12}$ and $E_{23}$ meet at an angle of $\pi$. 

    We note that the tiling and therefore the zigzag is rotationally symmetric via a rotation of angle $\pi$ around the point $m_2$. Then, this rotation swaps $m_1$ and $m_3$, and therefore $E_{12}$ and $E_{23}$ are related by a rotation of $\pi$. It follows then that $E_{12}$ and $E_{23}$ meet at an angle of $\pi$ and therefore together form a larger geodesic segment. By the translational symmetry of the zigzag, the collection of edges connecting successive midpoints of the zigzag is a bi-infinite geodesic. Since the zigzag is bounded distance away from this geodesic through the midpoints, it is a quasigeodesic with the same endpoints and hence divides the hyperbolic plane into two unbounded halves. 
\end{proof}

In light of Proposition~\ref{prop:zigzag_halfspace}, we have the following definition.

\begin{definition}[Midpoint geodesic of a zigzag] Given a zigzag $\zeta$ in a hyperbolic $(p,q)$-tiling with $q$ odd, let $g_\zeta$ be the geodesic that intersects all midpoints of edges in the zigzag $\zeta$. We call $g_\zeta$ the \textbf{midpoint geodesic} of $\zeta$.
\end{definition}
The next proposition asserts that zigzags that intersect along an edge split the plane into 4 quarter spaces. Such zigzags are disjoint from each other, except along the edge they share.

\zigzagTransverse*
\begin{proof}
Let $\zeta_1$ and $\zeta_2$ be two zigzags that share an edge $e$ in a hyperbolic $(p,q)$-tiling with $q$ odd. By the reflectional symmetry of the tiling over the edge geodesic $g_e$ that contains the edge $e$, we have that $\zeta_2$ and $\zeta_1$ must be reflections of each other over $g_e$. 

Consider the midpoint geodesics $g_{\zeta_1}$ and $g_{\zeta_2}$, which pass through the midpoints of every edge on $\zeta_1$ and $\zeta_2$ respectively. Since $\zeta_1$ and $\zeta_2$ share the edge $e$ and hyperbolic geodesics intersect at most once, $g_{\zeta_1}$ and $g_{\zeta_2}$ intersect at the midpoint of $e$ and at no other point. Since the edges preceding and succeeding $e$ on both $\zeta_1$ and $\zeta_2$ are on opposite sides of $g_e$, it follows that the midpoint geodesics $g_{\zeta_i}$, and therefore the zigzags $\zeta_i$ and $\zeta_2$, each cross $g_e$ at the edge $e$ and therefore cannot share any edge except $e$ with $g_e$. 

Furthermore, if $\zeta_i$ intersected $g_e$ at another point not on $e$, then the intersection must be at a vertex $v$ that is on both $\zeta_1$ and $\zeta_2$. The $\zeta_i$ cannot cross $g_e$ at $v$ since the midpoints of the edges on $\zeta_i$ immediately preceding and succeeding $v$ would then be on opposite sides of $g_e$, contradicting that the midpoint geodesics $g_{\zeta_i}$ cannot double cross $g_e$. Similarly, the zigzag edges preceding and succeeding $v$ cannot be on $g_e$. Thus, by angle considerations of the zigzag and the symmetry of the tiling across $g_e$, it follows that the edges of each $\zeta_i$ preceding and succeeding $v$ each make an angle of $\frac{\pi}{q}$ with $g_e$. Thus the tiling is symmetric across a line of symmetry $g_{sym}$ that is perpendicular to $g_e$ at $v$, as depicted in Figure \ref{fig:zigzag_symmetry}. 

\begin{figure}
    \centering
    \includegraphics[width=13cm]{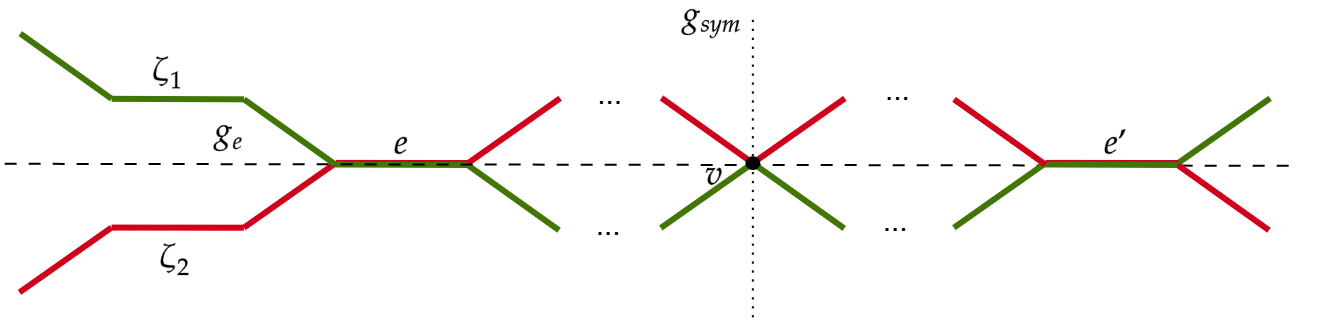}
    \caption{Two zigzags that intersect at an edge $e$ and a separate vertex $v$ must be symmetric across a line of symmetry $g_{sym}$ through $v$ and perpendicular to $g_e$.}
    \label{fig:zigzag_symmetry}
\end{figure}

A reflection of the tiling across $g_{sym}$ must then send each $\zeta_i$ to itself since this true locally near $v$. But then, the shared edge $e$ is reflected to another shared edge $e'$ on the other side of $g_{sym}$, contradicting that the zigzags can share at most one edge. 

It follows that the zigzags only share the edge $e$ and otherwise do not intersect. Since by Proposition \ref{prop:zigzag_halfspace} each zigzag divides the plane into two unbounded half spaces, we thus have that the two zigzags together divide the tiling up into four unbounded quarter spaces. 
\end{proof}

The next proposition reinforces the fact that zigzags behave like geodesics. It mimics the following hyperbolic geometry fact: if two geodesics $\gamma_1$ and $\gamma_2$ and two points $x$ and $y$ on $\gamma_1$ satisfy the property that $x$ and $y$ is within $\delta$ of $\gamma_2$, then every point on $\gamma_1$ between $x$ and $y$ is also within $\delta$ of $\gamma_2$.

\begin{proposition}\label{prop:consecutivezigzagintersection} Let $\zeta$ be a bi-infinite zigzag in a hyperbolic $(p,q)$-tiling with $q$ odd. If a geodesic (not necessarily an edge-geodesic) $\gamma$ intersects $\zeta$ at two of its edges, then $\gamma$ intersects $\zeta$ at every edge of $\zeta$ in-between the two intersection points. 
\end{proposition}

\begin{proof}
Denote $\zeta = (\dots, e_{-1}, e_0, e_1, e_2, \dots )$ where $e_i$ are the edges of the tiling. Label the vertex between $e_i$ and $e_{i+1}$ by $v_i$. Assume $\gamma$ is a geodesic that intersects $\zeta$ at two non-consecutive edges. Without loss of generality call those edges $e_0$ and $e_n$ for some $n \geq 2$, and the intersection points $p_0$ and $p_n$. Assume towards a contradiction that there are no other points of intersection in between $p_0$ and $p_n$. 

Assume without loss of generality that $p_0$ and $p_n$ are not vertices $v_0$ and $v_{n-1}$. Otherwise, if $p_0$ is $v_0$ (or $p_n$ is $v_{n-1}$), we can re-run the argument below by replacing $e_0$ with $e_1$(or $e_n$ by $e_{n-1}$) by considering the intersection point $p_0$ to be part of edge $e_1$ (or $p_n$ to be a part of edge $e_{n-1}$).

Since there are no intersecting points in between $p_0$ and $p_n$, the segment of $\gamma$ between these two points lies completely on one side of the zigzag $\zeta$. 
We can then consider the $(n+2)$-sided polygon, $P$, formed by the segment of $\gamma$ between $p_0$ and $p_n$ along with the segments $p_0v_0, v_1v_{2}, \dots, v_{n-1}p_n$ along the zigzag. Alternately, this is the polygon bounded by the vertices $p_0, v_0, v_1, \dots, v_{n-1}, p_n$.

\begin{figure}[h!]
\centering
\begin{subfigure}{0.4\textwidth}
\centering
\includegraphics[scale=0.575]{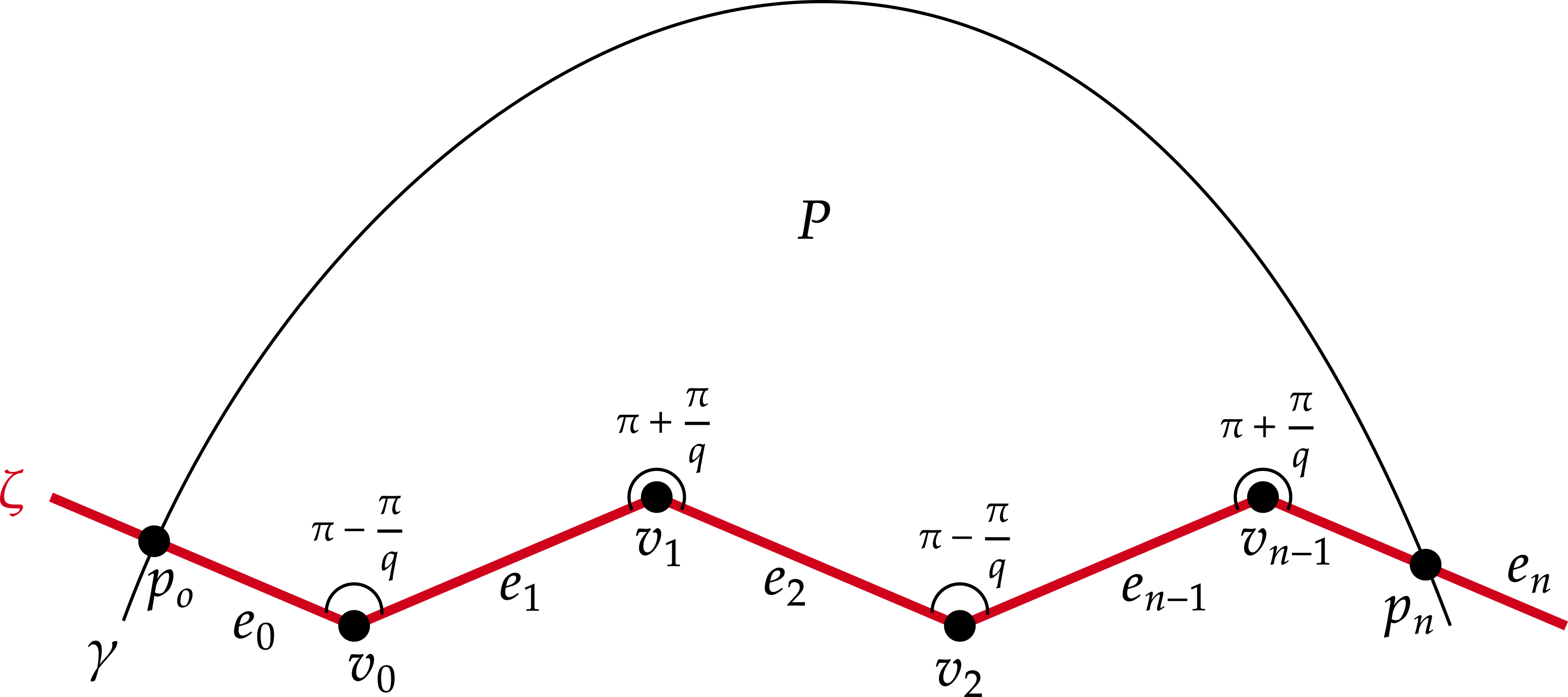}
\caption{The case when $n$ is even and the internal angle of $P$ at $v_0$ is $\pi-\pi/q$}
\end{subfigure}
\hspace{1cm}
\begin{subfigure}{0.4\textwidth}
\centering
\includegraphics[scale=0.575]{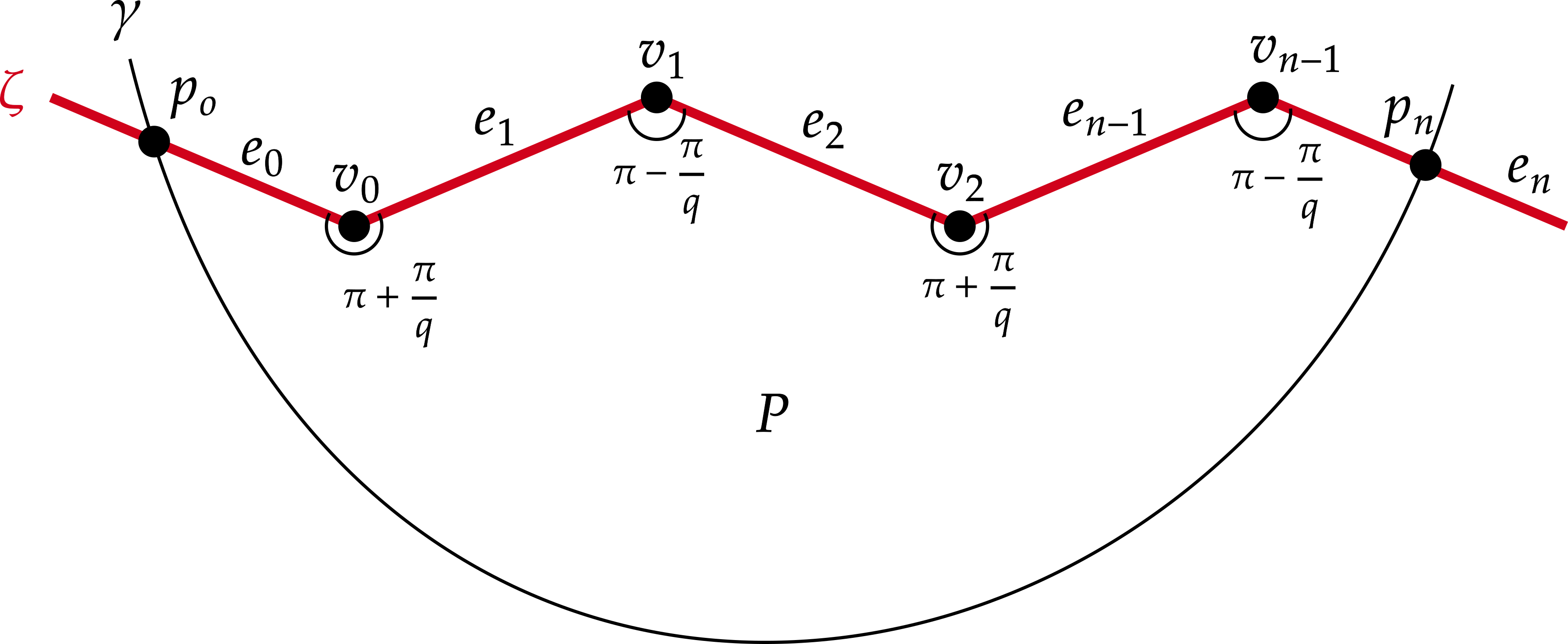}
\caption{The case when $n$ is even and the internal angle of $P$ at $v_0$ is $\pi+\pi/q$}
\end{subfigure}
\caption{When $n$ is even, the internal angles of $P$ at $v_0 \dots, v_{n-1}$ sum up to $n\pi$.}
\label{fig:nevenzigzag}
\end{figure}

First, we deal with the case that $n$ is even. By definition of a zigzag, the internal angle $\angle p_0 v_0 v_1$ of $P$ is either $\pi+\pi/q$ or $\pi-\pi/q$. See Figure \ref{fig:nevenzigzag}. Likewise, subsequent internal angles, $\angle v_0v_1v_2$, $\angle v_1v_2v_3$, $\dots \angle v_{n-2}v_{n-1}p_n$ alternate between $\pi-\pi/q$ and $\pi+\pi/q$. Since $n$ is even, the sum of the internal angles of $P$ at $v_0, v_1, \dots, v_{n-1}$ is $(\pi \pm \pi/q) + (\pi \mp \pi/q) + \dots +(\pi \mp \pi/q) = n\pi$. Hence, including the angles at $p_0$ and $p_n$, the sum of the internal angles of $P$ is greater than $n\pi$. As $P$ is a hyperbolic polygon with  $n+2$ sides, this is a contradiction. 

Now assume $n$ is odd. If $\angle p_0 v_0 v_1$, the internal angle of $P$ at $v_0$, is $\pi+ \pi/q$, then the sum of the internal angles of $P$ at $v_0, v_1, \dots v_{n-1}$ is $\pi+ \pi/q + \pi - \pi/q + \pi+ \pi/q + \dots + \pi+ \pi/q = n \pi + \pi/q$. Since $P$ has $n+2$ sides, this is again a contradiction as before.

\begin{figure}[h!]
\centering
\begin{subfigure}{0.8\textwidth}
\centering
\includegraphics[scale=0.8]{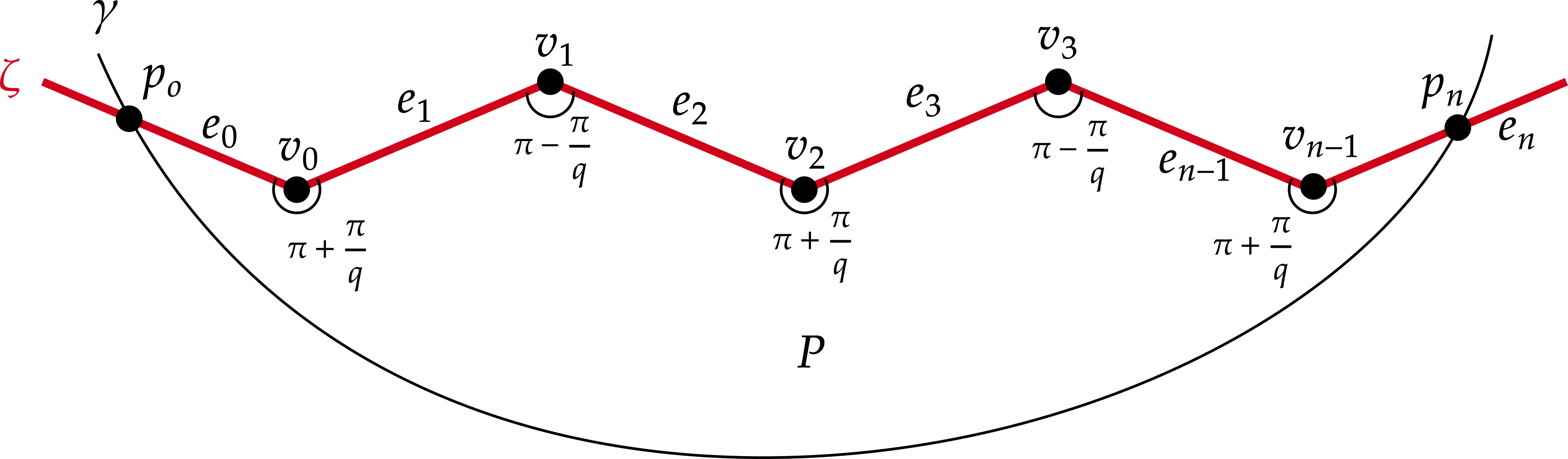}
\caption{The case when $n$ is odd and the internal angle of $P$ at $v_0$ is $\pi+\pi/q$. The sum of the internal angles of $P$ is $n \pi + \pi/q$.}    
\end{subfigure}
\begin{subfigure}{0.8\textwidth}
\centering
\includegraphics[scale=0.8]{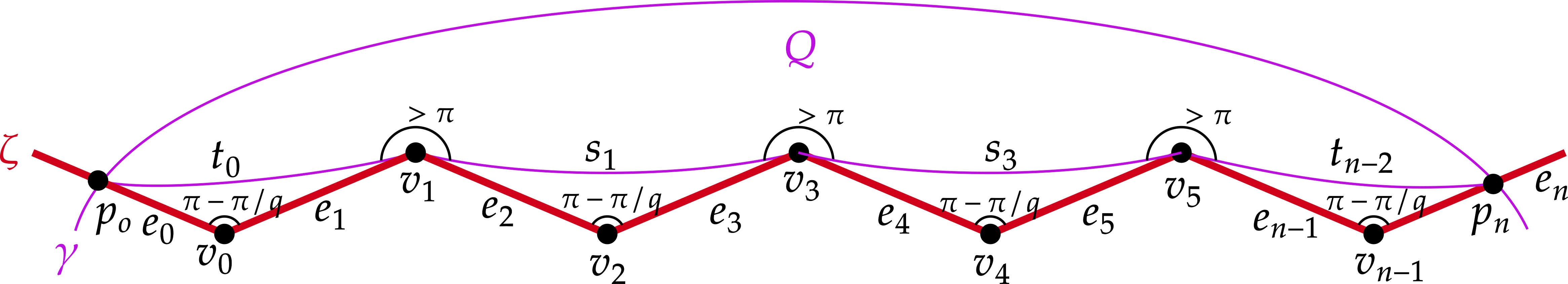}
\caption{The case when $n$ is odd and the internal angle of $P$ at $v_0$ is $\pi-\pi/q$. In this case we make use of a different polygon, $Q$, contained within $P$. The sum of the internal angles of $Q$ exceed $(\frac{n-1}{2})\pi$.}    
\end{subfigure}

\caption{When $n$ is odd, we use two polygons, $P$ or $Q$, depending on the angle at $v_0$ to obtain contradictions to hyperbolicity.}
\label{fig:noddzigzag}
\end{figure}

Finally, assume angle $\angle p_0 v_0 v_1$ is $\pi- \pi/q < \pi$. See Figure \ref{fig:noddzigzag}. Then in polygon $P$, the geodesic segment between $p_0$ and $v_1$ lies completely inside the polygon $P$. Construct this segment $t_0$.
Likewise, construct segments $s_i$ inside $P$ between vertices $v_i$ and $v_{i+2}$ for $i = 1, \dots, n-4$. Finally construct the segment $t_{n-2}$ between $v_{n-2}$ and $p_n$. 

Now consider the polygon $Q$ (containined within $P$) bounded by the vertices $p_0, v_1, v_3, v_5, \dots, v_{n-2}, p_n$ with sides $t_0, s_1, s_3 \dots, s_{n-4},  t_{n-2}.$ 

We first claim that the internal angles of $Q$ at $v_1, v_3, \dots, v_{n-2}$ are each greater than $\pi$. To see this, first let $3 \leq i \leq v_{n-4}$ and consider the isosceles triangle $\triangle v_{i-2}v_{i-1}v_{i}$ 
The internal angle to this triangle at $v_{i-1}$ is $\pi - \pi/q$ so that the sum of the other internal angles is $\angle v_{i-1}v_iv_{i-2} + \angle v_{i-1}v_{i-2}v_i < \pi/q$. As the triangle is isosceles, this means $\angle v_{i-1}v_iv_{i-2} = \angle v_{i-1}v_{i-2}v_i  < \frac{\pi}{2q}$.

The same argument in $\triangle v_{i}v_{i+1}v_{i+2}$ yields that  $\angle v_{i+1}v_{i}v_{i+2} < \frac{\pi}{2q}$. Note that triangles  $\triangle v_{i-2}v_{i-1}v_{i}$ and $\triangle v_{i}v_{i+1}v_{i+2}$ are outside $Q$, but inside $P$. Hence, 
\begin{align*}
\text{ the external angle to }Q\text{ at }v_{i} = \text{(external)} \angle v_{i-2}v_iv_{i+2} &= \angle v_{i-1}v_iv_{i-2} + \angle v_{i+1}v_{i}v_{i+2} + \angle v_{i-1}v_i v_{i+1} 
\\
&< \frac{\pi}{2q} + \frac{\pi}{2q}+ \pi - \frac{\pi}{q}   \\&= \pi
\end{align*}
So, the internal angle to $Q$ at $v_i$, $\angle v_{i-2} v_{i}v_{i+2} > \pi$. 
Note that since the segment $p_0v_0$ is shorter than $v_0v_1$ and segment $v_{n-1}p_n$ is shorter than $v_{n-2}v_{n-1}$, in triangles $\triangle p_0v_0v_1$ and $\triangle v_{n-2}v_{n-1}p_n$ respectively, we see that $\angle p_0v_1v_0 < \pi/(2q)$ and $\angle p_nv_{n-2}v_{n-1} < \pi/(2q)$. So, we also obtain that the internal angles to $Q$ at $v_1$ and $v_{n-2}$ are also each larger than $\pi$.

This yields,
$$\text{sum of interior angles of }Q > \angle p_0v_1v_3 + \angle v_1v_3v_5+ \dots +\angle v_{n-4}v_{n-2}p_n > \left(\frac{n-1}{2}\right)\pi$$ 
However, the number of sides of $Q$ is $\frac{n+1}{2}+1$ and as $Q$ is hyperbolic, the sum of the interior angles is bounded above by $(\frac{n+1}{2}-1)\pi = (\frac{n-1}{2})\pi$, a contradiction. 
\end{proof}

We can now prove the following lemma about intersections of zigzags by hyperbolic geodesics.

\ZigzigIntersections*
\begin{proof}  
    
\begin{figure}[h]
    \centering
        \includegraphics[width=10cm]{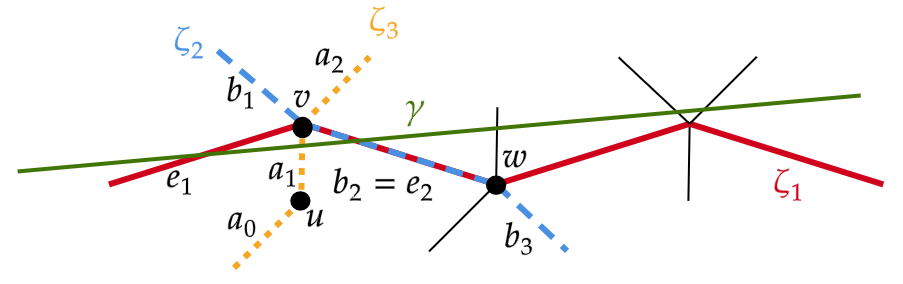}
        \caption{A diagram of consecutive zigzag intersections}
        \label{fig:zigzag_consecutive_intersections}
    \end{figure}
    
Following Figure \ref{fig:zigzag_consecutive_intersections}, suppose that $\gamma$ is a hyperbolic geodesic that intersects a zigzag $\zeta_1$ at two successive edges $e_1$ and $e_2$, meeting at a vertex $v$. Let $\zeta_2 \neq \zeta_1$ be a zigzag that passes through $v$. We will show that $\gamma$ intersects $\zeta_i$ exactly once. First suppose that $\zeta_2$ is the other zigzag containing $e_2 = b_2$, and let $b_1$ and $b_3$ be the two neighboring edges on $\zeta_2$, meeting $b_2$ at vertices $v$ and $w$ respectively. We note that $\gamma$ cannot cyclically cross (vertex traverse) more than $\frac{q-1}{2}$ edges adjacent to a vertex $v$ because this would result in $\gamma$ double crossing the edge geodesic of the first edge crossed. Then, $\gamma$ cannot intersect either $b_1$ or $b_3$ because such an intersection would cause a vertex  traversal that is too long at either vertex $v$ or $w$ respectively. By applying Proposition \ref{prop:consecutivezigzagintersection}, $\gamma$ cannot intersect $\zeta_2$ at any other edges except for $b_2$. A symmetric argument shows that $\gamma$ intersects the other zigzag $\zeta \neq \zeta_1$ containing $e_1$ exactly once as well.

Now, suppose that $\zeta_3$ is a zigzag that passes through $v$ but does not share an edge with $\zeta_1$. Let $a_0, a_1, a_2$ be three successive edges on this zigzag, with the latter two adjacent to $v$, and with $a_1$ ``between" (in the shorter vertex traversal between) edges $e_1$ and $e_2$. Then, $\gamma$ intersects $a_1$ because it is a part of the vertex traversal from $e_1$ to $e_2$. However, $\gamma$ could not intersect $a_2$ because it would cause a too long vertex traversal at vertex $v$, and $\gamma$ could not intersect $a_0$ because $\gamma$ cannot intersect any edges not adjacent to $v$ between its intersections with $e_1$ and $e_2$. Similarly, using the same argument for $u$, we conclude that $\gamma$ cannot intersect $a_0$ after having intersected with vertex $e_1$ and $e_2$. Applying \ref{prop:consecutivezigzagintersection}, $\gamma$ cannot intersect $\zeta_3$ at any other edges except for $a_1$, and we're done.\end{proof}

\subsection{Facts about $(p,q)$-tilings when $q$ is even}
In this subsection we focus on some facts about $(p,q)$-tilings when $q$ is even. We strart with a proposition that ensures consistency of edge labels in such hyperbolic tilings. We use this result in Section~\ref{subsec:qevenpreliminaries}.
\ConsistentLabelling*
\begin{proof}
Given a base tile with a cyclic labeling of the edges, every tile in the $(p,q)$-tiling is some finite tiling distance $n$ (see Definition \ref{def: tiling_length}) away from the base tile. We will prove this proposition by inducting on $n$. 

The only tiling distance $0$ tile is the base tile, which is labeled cyclically by assumption. We suppose that all tiling distance $n$ tiles are consistently labeled. For each tiling distance $n+1$ tile, we label it by reflecting the labeling of an adjacent tiling distance $n$ tile. The only worry is that a tiling distance $n+1$ tile may obtain two inconsistent labelings by two adjacent tiling distance $n$ edges. This cannot happen because the only way to have two distance $n$ tiles adjacent to a distance $n+1$ tile is if all three of these tiles share a vertex. Because $q$ is even, tiling labelings are consistent when tiles are reflected cyclically around a common vertex, since the edge labels around that vertex would reflect in an $abab...ab$ pattern with $q$ total edge labelings. Thus, the tiling distance $n+1$ tiles must also be consistently labeled, completing the induction. 
\end{proof}

The next three lemmas concern scenarios when edge geodesics of a hyperbolic $(p,q)$-tiling with $q$ even cannot intersect. The proofs hinge on the basic fact from hyperbolic geometry that the sum of angles of any hyperbolic triangle is less than $\pi$.

\EdgeGeodDontIntersectTwoCases*

\begin{proof}

Suppose that $A$, $g$ and $e$ are as in the statement. Since edge geodesics are comprised of edges of the tiling, either $g$ shares exactly one vertex with $A$ (and does not intersect $A$ otherwise) or $g$ shares exactly one edge with $A$. See Figure~\ref{fig:EdgeGeodNonIntersections} for the two cases. Suppose for contradiction that, the geodesic extension of $e$ intersects $g$. Then, we can create a hyperbolic polygon $P$ with $k$ vertices $v_1, \dots, v_k$ where $v_1$ is the intersection of the extension of $e$ and $g$, $v_2$ is the vertex of $A$ on $g$ that is closest to $v_1$, $v_3$ is the vertex of $e$ closest to $v_1$, and $v_4, \dots, v_k$ are vertices of $A$ between $v_2$ and $v_3$, not including the other vertex of $e$.  

\begin{figure}[h!]
\centering
\begin{subfigure}{0.35\textwidth}
\centering
    \includegraphics[scale=0.8]{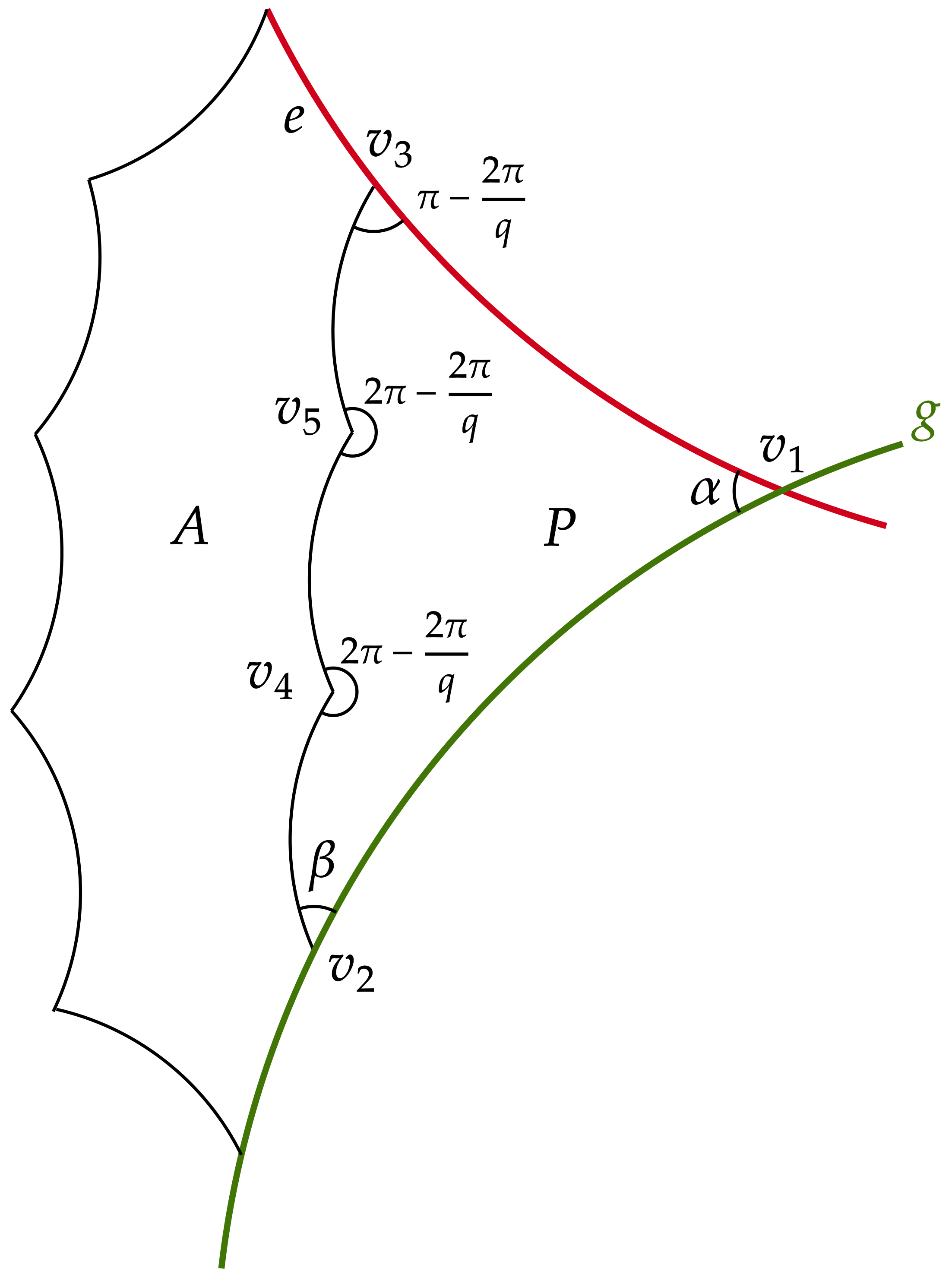}
    \caption{The edge geodesic $g$ shares an edge with tile $A$}
    \label{fig:geodnonintersectionA}
\end{subfigure}
\hspace{1cm}
\begin{subfigure}{0.35\textwidth}
\centering
\includegraphics[scale=0.8]{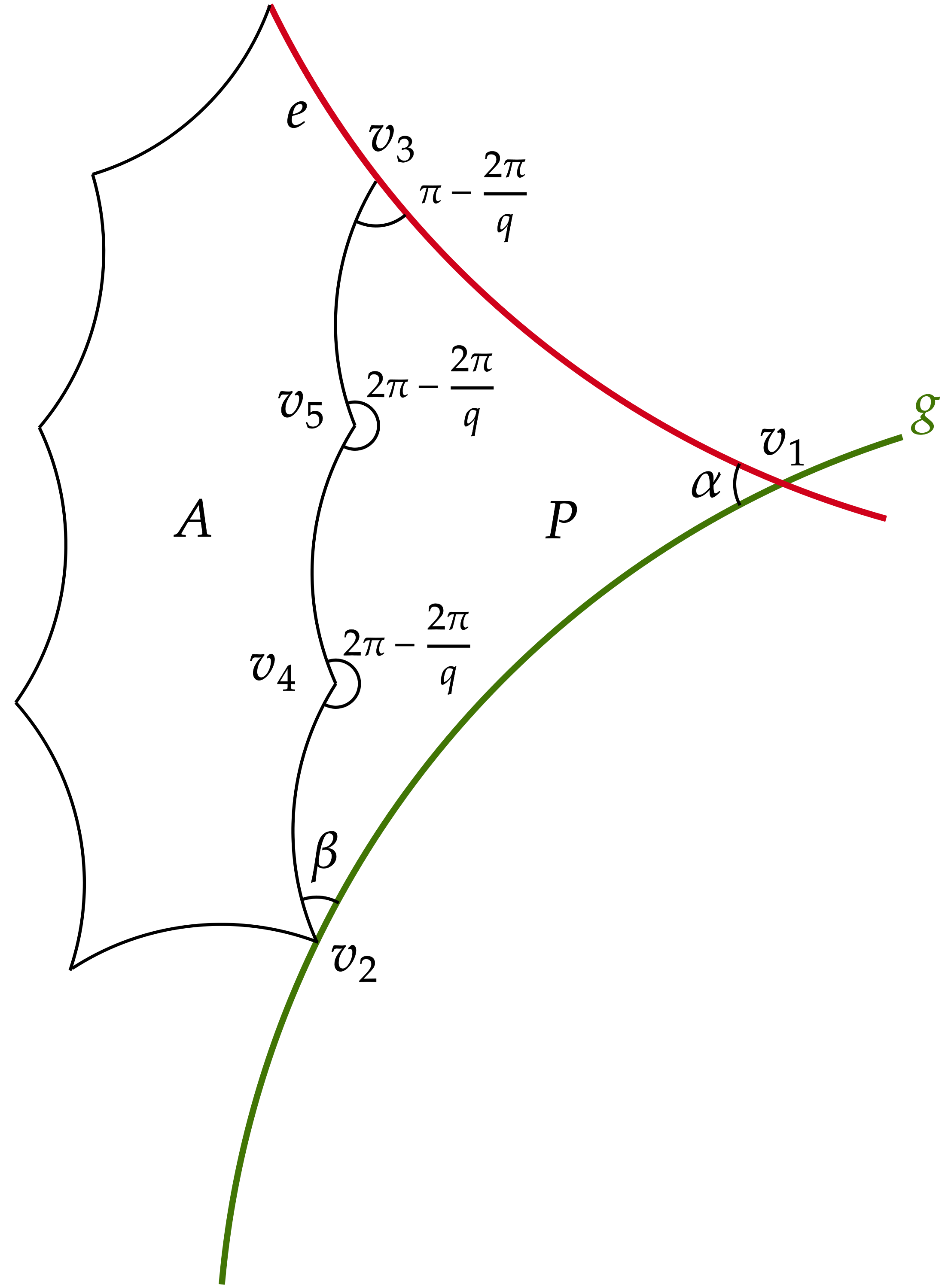}
    \caption{The edge geodesic $g$ shares a vertex with tile $A$}
    \label{fig:geodnonintersectionB}
    
\end{subfigure}
\caption{Two cases of non-intersecting edge geodesics illustrated by Lemma~\ref{lem:EdgeGeodDontIntersectTwoCases}. In both cases, assuming that $g$ intersects the geodesic extension of $e$ gives a contradiction to the hyperbolicity of the polygon $P$.}
\label{fig:EdgeGeodNonIntersections}
\end{figure}

Let $\alpha$ and $\beta$ be the measure in radians of the internal angles of $P$ at $v_1$ and $v_2$ respectively. Since the internal angles of any tile of the tiling are $2\pi/q$ each, $\alpha$ and $\beta$ are at least $2 \pi/q$. Moreover, the internal angles of polygon $P$ at $v_3$ is $\pi - \frac{2 \pi}{q}$ and at $v_4, \dots, v_k$ are $2\pi - \frac{2\pi}{q}$ each. Summing up these angles, we then get that the measure of the internal angles of $P$ satisfy,
$$\text{sum of internal angles of } P \geq 2\cdot \frac{2\pi}{q} + (\pi - \frac{2\pi}{q}) + (k-3) \left(2 \pi - \frac{2\pi}{q}\right) \geq (k-2)\pi + (k-3)\left(\pi - \frac{2\pi}{q}\right).$$ 
But for $P$ to be a polygon, $k \geq 3$, and therefore the right hand side is $\geq (k-2)\pi.$ Since the sum of the interior angles of a non-ideal hyperbolic $k$-gon must be strictly less than $(k-2)\pi$, this is a contradiction. Thus, edge geodesic extending $e$ cannot intersect $g$. \end{proof}

\fournonintersecting*
\begin{proof}
Given the conditions in the statement of the lemma, let $v=v_0, v_1, v_2, \ldots, v_n = w$ be the sequence of vertices of the tiling on $g_2$ between $v=v_0$ and $w=v_n$ inclusive (see Figure \ref{fig:p4_nonintersecting}). 

\begin{figure}[h!]
\centering
    \includegraphics[width=7cm]{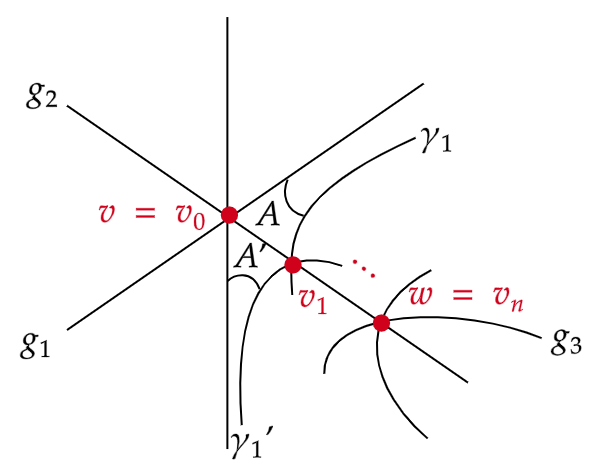}
    \caption{When $p\geq 4$ and $q$ is even, two edge geodesics $g_1$ and $g_3$ that intersect an edge geodesic $g_2$ at distinct vertices $v$ and $w$ will not intersect each other.}
    \label{fig:p4_nonintersecting}
\end{figure}

Let $\gamma_1, \gamma_1' \neq g_2$ be the edge geodesics passing through $v_1$ that shares an edge with a tile $A, A'$ adjacent to $v_0$. We note that when $q = 4, \gamma_1 = \gamma_1'$. Since $p \geq 4$, it follows that the conditions of Lemma \ref{lem:EdgeGeodDontIntersectTwoCases} are met and therefore $\gamma_1$ and $\gamma_1'$ do not intersect $g_1$. When $q \geq 8$, there are additional edge geodesics through $v_1$ that are not $\gamma_1, \gamma_1', g_2$. None of these additional edge geodesics intersect $g_1$ either because at $v_1$, they are ``between" $\gamma_1$ and $\gamma_1'$, which do not intersect $g_1$. 

At $v_2$, we can let $\gamma_1$ play the role of $g_1$ and we see that all of the edge geodesics $\neq g_2$ through $v_2$ do not intersect $\gamma_1$ and therefore do not intersect $g_1$. We repeat this process inductively until we get to $g_3$ at vertex $w=v_n$ to finish the proof of the lemma. 
\end{proof}

\threenonintersecting*
\begin{proof}
Consider a sector in $\HH^2$ defined by rays $g_1$ and $g_2$ emanating from $v$ with angle at $v$ at least $6 \pi/q$. Let $A$ be the tile in this sector that is adjacent to $v$ with one edge on $g_2$. Since $p = 3$, $A$ is a regular hyperbolic triangle with angles $2\pi/q$. Let $v = v_0, v_1, \dots, v_n = w$ be the sequence of vertices of the tiling on $g_2$ between $v=v_0$ and $w = v_n$ (see Figure \ref{fig:g1g3NonIntersectp3}).

\begin{figure}[h!]
\centering
\begin{subfigure}{0.35\textwidth}
\centering
    \includegraphics[scale=1]{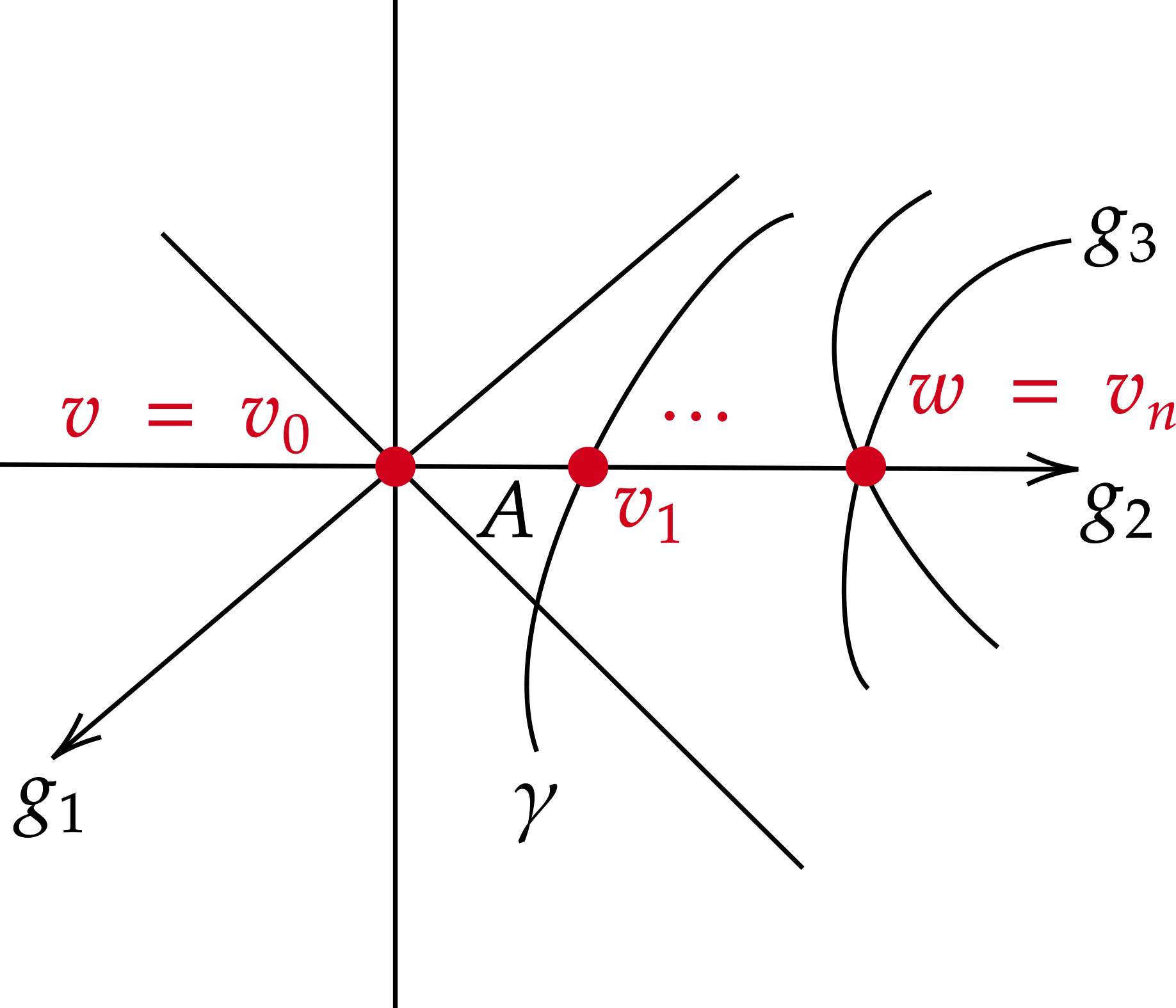}
    \caption{The geodesic $g_3$ cannot intersect the ray $g_1$.}
    \label{fig:g1g3NonIntersectp3}
\end{subfigure}
\hspace{2cm}
\begin{subfigure}{0.35\textwidth}
\centering
\includegraphics[scale=1]{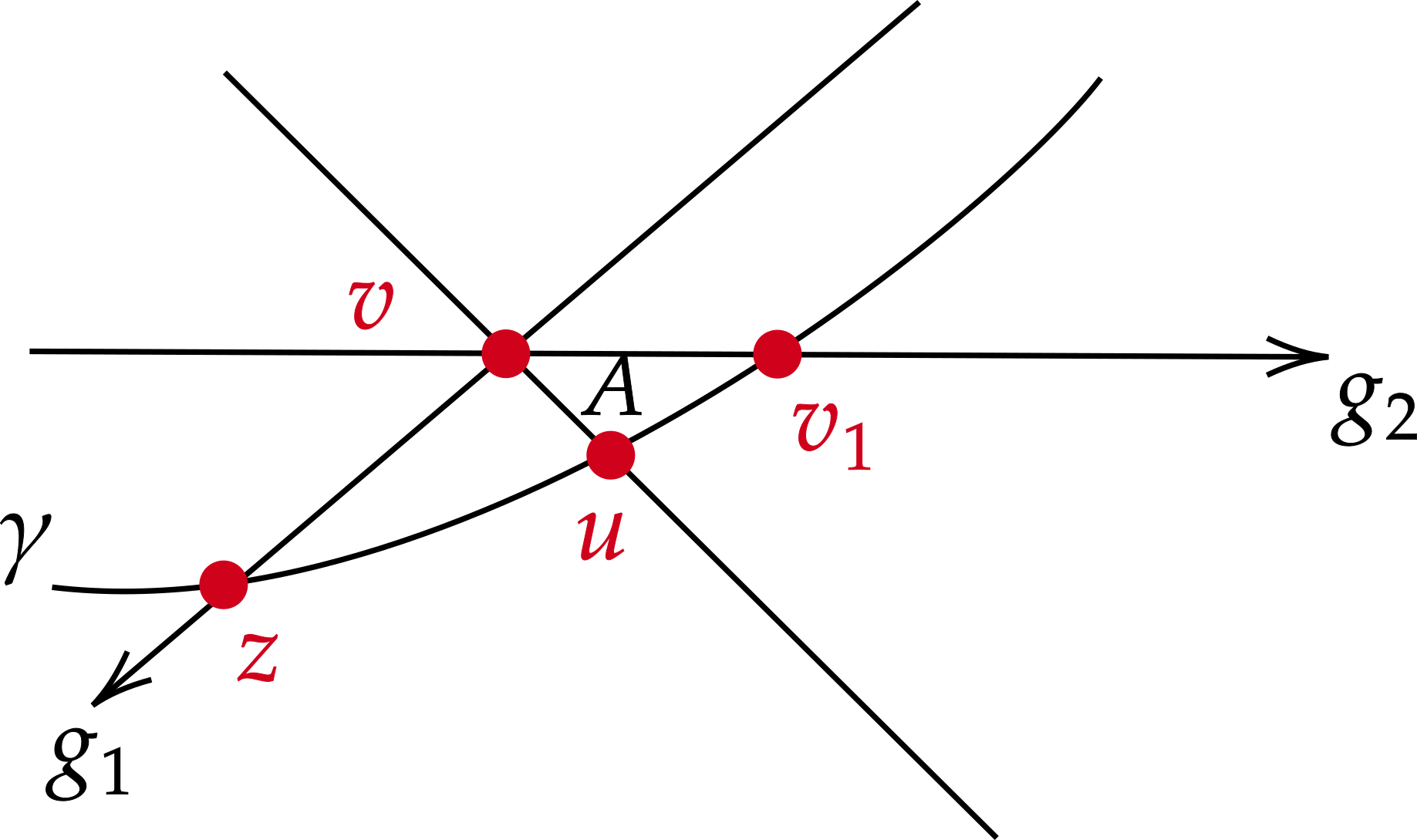}
    \caption{The triangle $\triangle v u z$ is impossible in hyperbolic geometry.}
    \label{fig:impossibletriangle}
\end{subfigure}
\caption{When $p=3$ and $q$ even, schematic depicting two edge geodesic rays $g_1$, $g_2$ emanating from vertex $v$ making angle at least $6\pi/q$. Any geodesic $g_3\neq g_2$ that intersects $g_2$ at a vertex $w$ other than $v$ cannot intersect $g_1$.}
\label{fig:Pequal3AppdxLemma}
\end{figure}

Let $\gamma$ be the edge geodesic corresponding to the edge of $A$ not containing $v$, passing through $v_1$. We first claim that $\gamma$ cannot intersect $g_1$. Assume towards contradiction that it does intersect $g_1$ at vertex $z$. Consider $\triangle v u z$ where $u$ is the vertex of $A$ on $\gamma_1$ closest to $z$. Then $\angle z u v = \pi - 2\pi/q$ and $\angle z v u  \geq 2\pi/q$. So, the sum of angles of $\triangle v u z \geq \pi - 2\pi/q + 2 \pi/q = \pi$, a contradiction (see Figure \ref{fig:impossibletriangle}). Hence, $\gamma$ cannot intersect $g_1$. Moreover, the ray of $\gamma$ contained in this sector is in between $g_1$ and any other geodesic ray (in the sector) of the tiling passing through $v_1$. Hence, these cannot intersect $g_1$ as well, otherwise they must double intersect $\gamma$. 

Note now that the angle between the half-ray of $\gamma$ and $g_2$ is $\pi-2\pi/q \geq 6 \pi/q$ as $q \geq 8$ (for $p=3$). Hence, letting $v_1$ play the role of $v$, $v_2$ play the role of $v_1$ and $\gamma$ play the role of $g_1$ we see that all the edge geodesics $\neq g_2$ through $v_2$ do not intersect $\gamma$ and therefore do not intersect $g_1$. Repeating this process inductively until we get to $g_3$ at vertex $w = v_n$, we conclude the proof. 
\end{proof}